\documentclass[12pt,reqno]{amsart}

\usepackage[T1]{fontenc}
\usepackage{lmodern}
\usepackage{amssymb,mathrsfs}
\usepackage{microtype}
\usepackage[
  letterpaper,
  left=1in,
  right=1in,
  top=1in,
  bottom=1in,
  headheight=8pt,
  headsep=14pt,
  footskip=12pt
]{geometry}
\usepackage{xcolor}
\usepackage{hyperref}
\definecolor{ReferenceBlue}{RGB}{0,0,255}
\hypersetup{
  colorlinks=true,
  linkcolor=ReferenceBlue,
  citecolor=ReferenceBlue,
  urlcolor=ReferenceBlue,
  linktoc=page,
  breaklinks=true,
  bookmarksnumbered=true,
  bookmarksopen=true,
  pdfpagemode=UseOutlines,
  pdfstartview=FitH,
  pdftitle={Domains of Monomial Convergence in Banach Sequence Spaces},
  pdfauthor={Zhouzhe Wang},
  pdfsubject={Infinite-dimensional holomorphy; Banach sequence spaces. MSC 2020: 46G20 (Primary), 46B45 (Secondary)},
  pdfkeywords={infinite-dimensional holomorphy, monomial convergence, Banach sequence spaces, Reinhardt domains, homogeneous polynomials}
}

\newtheorem{theorem}{Theorem}[section]
\newtheorem{lemma}[theorem]{Lemma}
\newtheorem{proposition}[theorem]{Proposition}
\newtheorem{corollary}[theorem]{Corollary}
\theoremstyle{definition}
\newtheorem{definition}[theorem]{Definition}

\newtheorem{problem}[theorem]{Problem}
\theoremstyle{remark}
\newtheorem{remark}[theorem]{Remark}

\numberwithin{equation}{section}

\newcommand{\itemref}[1]{\hyperref[#1]{\textup{(\ref*{#1})}}}
\newcommand{\namedref}[2]{#1~\ref{#2}}
\newcommand{\proofstep}[1]{\par\medskip\noindent\textbf{Step #1}\par\nopagebreak[4]}
\title[Domains of Monomial Convergence]{Domains of Monomial Convergence\\in Banach Sequence Spaces}
\author{Zhouzhe Wang}
\address{Beijing International Center for Mathematical Research,
Peking University, Beijing 100871, China}
\email{wangzhouzhe@pku.edu.cn}
\date{}
\makeatletter
\renewcommand{\subsection}{\@startsection{subsection}{2}%
  \z@{.5\linespacing\@plus.7\linespacing}{.5em}%
  {\normalfont\bfseries}}
\@namedef{subjclassname@2020}{\textup{2020} Mathematics Subject Classification}
\makeatother
\subjclass[2020]{Primary 46G20; Secondary 46B45.}
\keywords{infinite-dimensional holomorphy, monomial convergence, Banach sequence spaces, Reinhardt domains, homogeneous polynomials}

\begin{document}

\begin{abstract}
A classical phenomenon distinguishing infinite-dimensional holomorphy from its finite-dimensional counterpart is that monomial expansions of holomorphic functions need not converge absolutely. In this paper, we systematically investigate monomial convergence on complex Banach sequence spaces.

We establish general structural results and give a negative answer to Lempert's problem. We then determine the domains of absolute monomial convergence for holomorphic functions on the classical Banach sequence spaces and obtain a complete classification for continuous homogeneous polynomials. As a further application, we establish a sharp comparison between two notions of convergence radius for homogeneous Taylor expansions on $\ell^1$. The methods developed here are of independent interest and may be applied in other places.
\end{abstract}

\maketitle
\tableofcontents

\section{Introduction}
\label{jie-yinyan}
A central question in infinite-dimensional holomorphy concerns the absolute convergence of monomial expansions. Unlike the finite-dimensional case, such convergence need not hold throughout a Reinhardt domain in a Banach sequence space. Monomial expansions also play a role in the study of the inhomogeneous Cauchy--Riemann equation and holomorphic approximation in infinite dimensions; see \cite{Lempert99,Lempert00}. They arise in Cauchy--Kowalevski and Holmgren-type theorems for functions of infinitely many variables; see \cite{YZ19}. Through the Bohr transform, they are closely connected with Hardy spaces and ordinary Dirichlet series; see \cite{BDFMS17,DGMSP19}. They have also been used to describe natural topologies and to
construct Schauder bases in spaces of holomorphic functions on
Banach sequence spaces; see \cite{Ryan87,DM13,GL26}.

Throughout, $X$ is a complex Banach sequence space, $(e_j)_{j\in\mathbb{N}}$ denotes its
coordinate unit vectors, and $B_X$ is its open unit ball. By a Reinhardt domain $O\subset X$, we mean a complete Reinhardt domain containing $\textbf{0}$. For such a domain $O\subset X$, we write $H(O)$, $H^\infty(O)$, and $H_b(O)$
for the scalar-valued holomorphic functions, the bounded holomorphic
functions, and the holomorphic functions of bounded type, respectively.
The precise conventions are collected in
\namedref{Section}{jie-yubei}. For $f\in H(O)$, its restrictions to
finite-dimensional coordinate subspaces determine  Taylor
coefficients $c_\alpha(f)$, where $\alpha=(\alpha_j)$ ranges over the
finitely supported sequences of nonnegative integers. With
$\textbf{z}^\alpha=\prod\limits_{j=1}^{\infty}z_j^{\alpha_j}$, the domain of monomial
convergence of a family $\mathcal{F}(O)\subset H(O)$ is defined by
\begin{equation}
\mathrm{dom}\mathcal{F}(O)
\triangleq
\left\{\textbf{z}\in O:
\sum\limits_\alpha |c_\alpha(f)\textbf{z}^\alpha|<\infty
\text{ for every }f\in\mathcal{F}(O)\right\}.
\label{shi-shoulianyu}
\end{equation}

\paragraph{\textbf{Lempert's question.}}
Lempert's work on the inhomogeneous Cauchy--Riemann equation in
infinite dimensions highlights the special role of $\ell^1$ in the
convergence of power series in infinitely many variables
\cite{Lempert99}. A remarkable manifestation of this phenomenon
is that, for every Reinhardt domain $O\subset\ell^1$, the monomial
expansion of each $f\in H(O)$ converges absolutely at every point
of $O$; see \cite[Theorem~3.1, p.~19]{DMP09}.
This naturally raises the question of whether such convergence
properties characterize $\ell^1$ among infinite-dimensional
Banach sequence spaces. Following
Lempert's question, formulated explicitly in
\cite[Section~7, p.~43]{DMP09} and subsequently discussed in
\cite[Section~2.4, p.~118]{DSP14}, we consider the following problem;
see also \cite[Section~20.5, p.~526]{DGMSP19}. Here $\cong$ denotes
isomorphism of complex Banach spaces.
\begin{problem}\label{wenti-quanshoulian}
    Is it true that
    $$
        \mathrm{dom}H(B_X)=B_X
        \quad\Longleftrightarrow\quad
        X\cong\ell^1?
    $$
\end{problem}

The polynomial analogue has a negative answer. \cite[Theorem~4.3, pp.~60--62]{DK05} constructed a complex Banach space that is not isomorphic to $\ell^1$. As observed in
\cite[Example~7.2, pp.~44--45]{DMP09}, the monomial expansion of every
continuous polynomial on this space converges absolutely at every point.
This does not, however, settle the corresponding question for all
holomorphic functions.

We give a negative answer to \namedref{Problem}{wenti-quanshoulian}.
\namedref{Theorem}{dingli-fanliquanshoulian} gives a complex Banach sequence
space $Y$ with a normalized $1$-symmetric Schauder basis which is not
isomorphic to $\ell^1$. On every Reinhardt domain in $Y$, each
holomorphic function is represented at every point by its absolutely
convergent monomial expansion. Moreover, for every nondecreasing function $\Psi:\mathbb{N}\to[1,+\infty)$ with $\Psi(n)\to+\infty$, the space $Y$ can be chosen so that
$$
\frac{1}{\Psi(n)}\|\textbf{z}\|_{\ell^1}
\leqslant
\|\textbf{z}\|_Y
\leqslant
\|\textbf{z}\|_{\ell^1}
$$
whenever $|\mathrm{supp}\textbf{z}|\leqslant n$. This strengthens the ``very close to $\ell^1$'' phenomenon emphasized in connection with Lempert's problem in \cite[ p.~389]{DMSS14}; see also the discussion following \cite[Theorem~2.21, pp.~5869-5870]{DGMM}.

\paragraph{\textbf{General principles.}}
We first address the following question, raised in \cite[Remark~4.8, p.~32]{DMP09}.
\begin{problem}\label{wenti-bibaobaohan}

Let $O\subset X$ be a Reinhardt domain, and let
$\mathcal{F}(O)\subset H(O)$ contain the restrictions
to $O$ of all continuous polynomials on $X$.
Is it true that
$$
\mathrm{dom}\mathcal{F}(O)
\subset
\overline{\mathrm{span}\{e_j:j\in\mathbb{N}\}}^{\,X}?
$$
\end{problem}

We give an affirmative answer to  \namedref{Problem}{wenti-bibaobaohan}.
In fact, \namedref{Theorem}{dingli-ercizhangai}
shows that, at every point outside the closed linear
span of the coordinate unit vectors, the monomial
expansion of some continuous quadratic polynomial
fails to converge absolutely. 

Next, \namedref{Proposition}{mingti-jubuhua} shows that if
$\overline{\mathrm{span}\{e_j:j\in\mathbb{N}\}}^{\,X}=X$,
then
\begin{equation}
	\mathrm{dom}H(O)
	=
	O\cap\mathrm{dom}H(X)
	\label{shi-linshiyinyong}
\end{equation}
for every Reinhardt domain $O\subset X$.
Thus, determining monomial convergence on arbitrary
Reinhardt domains in $X$ reduces to determining
monomial convergence for entire functions on $X$.

Finally, \namedref{Theorem}{dingli-shoulianpanju}
characterizes the identity $\mathrm{dom}H(X)=X$
in terms of uniform exponential coefficient estimates
for homogeneous polynomials.

\paragraph{\textbf{Classical sequence spaces.}}
\begin{problem}\label{wenti-jingdiankongjian}
Determine the monomial convergence sets
for holomorphic functions and homogeneous polynomials
on the classical Banach sequence spaces.
\end{problem}
For related results on monomial convergence of continuous
homogeneous polynomials and bounded holomorphic functions, see
\cite{Prengel05,DMP09,BDFMS17,BDS19,GMMS21}. For an earlier contribution to the bounded holomorphic case,
see \cite{DT91}.
For holomorphic functions of bounded type, see
\cite{Ryan87,DMP09,Mansilla19,GMMS21}.
For general holomorphic functions, see
\cite{Ryan87,Matos88,MN92,Lempert99,Prengel05,DMP09,BRS25}. For related results on monomial convergence of vector-valued
holomorphic functions, see \cite{DSP12}.

We obtain exact characterizations for the classical sequence spaces
in the following sense.
\namedref{Theorem}{dingli-jingdianshoulian} determines the monomial
convergence domains for entire functions on $\ell^p$,
$1\leqslant p<\infty$, and on $c_0$, while
\namedref{Theorem}{dingli-youjiexulieshoulian} treats the
nonseparable space $\ell^\infty$.
For every fixed $m\in\mathbb{N}$,
\namedref{Theorem}{dingli-duoxiangshishoulianyu} gives a complete
characterization of the monomial convergence domains of continuous
$m$-homogeneous polynomials on all these spaces.
Moreover, \namedref{Corollary}{tuilun-jingdianquanyu} gives the
corresponding characterizations for arbitrary Reinhardt domains in
$\ell^p$, $1\leqslant p<\infty$, and in $c_0$, whereas
\namedref{Proposition}{mingti-youjiexulieqiu} determines
$\mathrm{dom}H(B_{\ell^\infty})$.
Comparing \namedref{Theorem}{dingli-youjiexulieshoulian} with
\namedref{Proposition}{mingti-youjiexulieqiu} also shows that
\eqref{shi-jubuhua} in
\namedref{Proposition}{mingti-jubuhua} fails for the unit ball of
$\ell^\infty$.
\paragraph{\textbf{Main ideas of the proofs.}}
\begin{enumerate}
	\item  The idea behind our solution to \namedref{Problem}{wenti-bibaobaohan} is inspired by
	Littlewood's finite-dimensional construction
	\cite[Section~7, pp.~171--172]{Littlewood30}.
	We transfer this construction to the infinite-dimensional setting
	by composing the finite-dimensional bilinear forms with suitable
	linear maps built from carefully chosen linear functionals and, for each
	$\textbf{x}\in
	X\setminus\overline{\mathbb{C}^{(\mathbb{N})}}^{\,X}$,
	construct a quadratic polynomial whose monomial expansion is not
	absolutely convergent at $\textbf{x}$.
	\item The construction of $Y$ is based on that in
	\cite[Section~4, pp.~58--62]{DK05}, which was also used in
	\cite[Example~7.2, pp.~44--45]{DMP09}, with modifications needed here.
	The choice of parameters gives the coefficient estimate
	\eqref{shi-fanlixishuguji}, whose constants grow at most exponentially
	with the degree by \eqref{shi-cishigenjie}. Together with the Cauchy
	estimates for the homogeneous Taylor components, this exponential bound
	is sufficient to sum over the degrees and pass from homogeneous
	polynomials to general holomorphic functions.
\item  To determine $\mathrm{dom}H(O)$ for the classical sequence spaces, we first use \eqref{shi-linshiyinyong} to reduce the problem to determining $\mathrm{dom}H(X)$. The main ingredient in the proof of \namedref{Theorem}{dingli-jingdianshoulian} is \namedref{Lemma}{yinli-chongpaixishuguji}. We partition the degree-\(m\) multi-indices according to their degrees on an initial coordinate block and subsequent dyadic blocks. After deleting one occurrence of the largest index, we apply the counting estimate from \cite[Theorem~3.2]{BDS19} to obtain a summable bound over the blocks.
\end{enumerate}

The rest of the paper is organized as follows. \namedref{Section}{jie-yubei} collects the notation and preliminary
facts. \namedref{Section}{jie-yibanyuanli} contains the general
results on monomial convergence. The space $Y$ and its properties
are presented in \namedref{Section}{jie-fanligouzao}, and the results
for the classical sequence spaces are given in
\namedref{Section}{jie-jingdiankongjian}.

\section{Preliminaries}
\label{jie-yubei}

\label{duan-jibenjihao}

Let $\mathbb{N}=\{1,2,\cdots\}$ and $\mathbb{N}_0=\mathbb{N}\cup\{0\}$.
For each $n\in\mathbb{N}$, let $S_n$ denote the permutation group of
$\{1,\cdots,n\}$. For a sequence $\textbf{x}=(x_j)_{j\in\mathbb{N}}$,
write
$$
\mathrm{supp}\textbf{x}\triangleq\{j\in\mathbb{N}:x_j\neq0\},
\qquad
|\textbf{x}|\triangleq(|x_j|)_{j\in\mathbb{N}}.
$$
We use the notation
$$
\begin{aligned}
\mathbb{C}^{(\mathbb{N})}
&\triangleq\{\textbf{x}\in\mathbb{C}^{\mathbb{N}}:
                \mathrm{supp}\textbf{x}\text{ is finite}\},\\
\mathbb{N}_0^{(\mathbb{N})}
&\triangleq\{\alpha\in\mathbb{N}_0^{\mathbb{N}}:
                \mathrm{supp}\alpha\text{ is finite}\},
   \\ \left( \mathbb{Q} +\mathrm{i}\mathbb{Q} \right)^{(\mathbb{N})}&\triangleq\{\alpha\in\left( \mathbb{Q} +\mathrm{i}\mathbb{Q} \right)^{\mathbb{N}}:
   \mathrm{supp}\alpha\text{ is finite}\}.
\end{aligned}
$$
For $\textbf{x},\textbf{y}\in\mathbb{C}^{\mathbb{N}}$, the relation
$|\textbf{x}|\leqslant|\textbf{y}|$ means that $|x_j|\leqslant|y_j|$
for every $j\in\mathbb{N}$.

For $1\leqslant p<\infty$, we write
$$
\ell^p\triangleq
\left\{\textbf{z}=(z_j)_{j\in\mathbb{N}}\in\mathbb{C}^{\mathbb{N}}:
\sum\limits_{j=1}^{\infty}|z_j|^p<\infty\right\},
\qquad
\|\textbf{z}\|_{\ell^p}
\triangleq
\left(\sum\limits_{j=1}^{\infty}|z_j|^p\right)^{\frac{1}{p}}.
$$
We also write
$$
\ell^\infty\triangleq
\left\{\textbf{z}\in\mathbb{C}^{\mathbb{N}}:
\sup\limits_{j\in\mathbb{N}}|z_j|<\infty\right\},
\qquad
\|\textbf{z}\|_{\ell^\infty}
\triangleq
\sup\limits_{j\in\mathbb{N}}|z_j|.
$$
For each $n\in\mathbb{N}$, we identify
$\textbf{u}=(u_1,\cdots,u_n)\in\mathbb{C}^n$ with
$(u_1,\cdots,u_n,0,0,\cdots)\in\mathbb{C}^{(\mathbb{N})}$.
Accordingly, $\|\textbf{u}\|_{\ell^p}$ denotes the $\ell^p$-norm
of this zero extension whenever $1\leqslant p\leqslant\infty$.

Furthermore,
$$
c_0\triangleq
\{\textbf{z}\in\mathbb{C}^{\mathbb{N}}:z_j\to0\},
$$
equipped with the supremum norm.

We call a complex Banach space $(X,\|\cdot\|_X)$ a complex Banach
sequence space if
$$
\mathbb{C}^{(\mathbb{N})}\subset X\subset\mathbb{C}^{\mathbb{N}},
$$
its vector space operations are coordinatewise, and, for every $\textbf{x}\in X$ and every
$\textbf{y}\in\mathbb{C}^{\mathbb{N}}$ satisfying
$|\textbf{y}|\leqslant|\textbf{x}|$, one has
\begin{equation}
    \textbf{y}\in X,
    \qquad
    \|\textbf{y}\|_X\leqslant\|\textbf{x}\|_X.
    \label{shi-fanshudandiao}
\end{equation}

For each $i\in\mathbb{N}$, let $e_i$ denote the sequence whose $i$th
coordinate is $1$ and whose remaining coordinates are $0$.
By \eqref{shi-fanshudandiao},
\begin{equation}
|z_i|\|e_i\|_X=\|z_ie_i\|_X\leqslant\|\textbf{z}\|_X,
\qquad \textbf{z}\in X,\quad i\in\mathbb{N}.
\label{shi-zuobiaoguji}
\end{equation}
Since $\|e_i\|_X>0$, each coordinate functional $\textbf{z} \mapsto z_{i}$ is continuous.

For a finite set $U\subset\mathbb{N}$, define
\begin{equation}
P_U\textbf{x}\triangleq\sum\limits_{j\in U}x_je_j,
\qquad \textbf{x}\in X.
\label{shi-youxiantouying}
\end{equation}
 \eqref{shi-fanshudandiao} gives
\begin{equation}
P_U^2=P_U,
\qquad \|P_U\textbf{x}\|_X\leqslant\|\textbf{x}\|_X.
\label{shi-touyingsuosuo}
\end{equation}
For $n\in\mathbb{N}$, put
\begin{equation}
P_n\textbf{x}\triangleq\sum\limits_{j=1}^{n}x_je_j,
\qquad Q_n\textbf{x}\triangleq\textbf{x}-P_n\textbf{x}.
\label{shi-shouweitouying}
\end{equation}
We also set $P_0=0$ and $Q_0=I_X$, where $I_X$ is the identity on $X$.

We write $\overline{\mathbb{C}^{(\mathbb{N})}}^{\,X}$ for the closure of
$\mathbb{C}^{(\mathbb{N})}$ in $X$ with respect to the norm $\|\cdot\|_X$.
\begin{lemma}\label{yinli-fenkuaixiajie}
If $\textbf{x}=\left( x_{i} \right)_{i\in \mathbb{N}}\in X\setminus\overline{\mathbb{C}^{(\mathbb{N})}}^{\,X}$,
there exist $\delta>0$ and integers
$0=n_0<n_1<n_2<\cdots$ such that
$$
\|(P_{n_k}-P_{n_{k-1}})\textbf{x}\|_X>\delta,
\qquad k\in\mathbb{N}.
$$
\end{lemma}
\begin{proof}
If $\{P_n\textbf{x}\}_{n\in\mathbb{N}_0}$ is a Cauchy sequence in $X$, then,
since $X$ is complete, there exists
$\textbf{y}=(y_i)_{i\in\mathbb{N}}\in X$ such that
$
\lim\limits_{n\to\infty}P_n\textbf{x}=\textbf{y}.
$
Since the coordinate functionals are continuous, we have
$$
x_i=\lim\limits_{n\to\infty}(P_n\textbf{x})_i=y_i,
\qquad i=1,2,\cdots .
$$
Thus $\textbf{x}=\textbf{y}\in
\overline{\mathbb{C}^{(\mathbb{N})}}^{\,X}$, which contradicts
$\textbf{x}\in X\setminus
\overline{\mathbb{C}^{(\mathbb{N})}}^{\,X}$. The sequence $(P_{n}\textbf{x} )_{n\in \mathbb{N}_{0}}$ is not Cauchy in $X$. Consequently, there is $\delta>0$ such that for every
$N\in\mathbb{N}_0$ there are $b>a\geqslant N$ with
$\|(P_b-P_a)\textbf{x}\|_X>\delta$.
Since
$$
|(P_b-P_a)\textbf{x}|
\leqslant|(P_b-P_N)\textbf{x}|,
$$
\eqref{shi-fanshudandiao} gives
$\|(P_b-P_N)\textbf{x}\|_X>\delta$.
Starting with $n_0=0$ and choosing $n_k=b$ for $N=n_{k-1}$
gives the required sequence.
\end{proof}

\begin{lemma}\label{yinli-xuliekongjian}
The following assertions hold.
\begin{enumerate}
\item\label{xiang-touyingxingzhi} For every $\textbf{z}\in X$ and $n\in\mathbb{N}_0$, the projections in \eqref{shi-shouweitouying} satisfy
        \begin{equation}
            \|P_n\textbf{z}\|_X\leqslant\|\textbf{z}\|_X,
            \qquad
            \|Q_n\textbf{z}\|_X\leqslant\|\textbf{z}\|_X.
            \label{shi-shouweisuosuo}
        \end{equation}
        Moreover, for every
        $\textbf{z}\in\overline{\mathbb{C}^{(\mathbb{N})}}^{\,X}$,
        \begin{equation}
            \lim\limits_{n\to\infty}\|P_n\textbf{z}-\textbf{z}\|_X
            =
            \lim\limits_{n\to\infty}\|Q_n\textbf{z}\|_X
            =
            0.
            \label{shi-weibushoulian}
        \end{equation}

        \item\label{xiang-bibaokongjian}
        The space $\overline{\mathbb{C}^{(\mathbb{N})}}^{\,X}$
        equipped with the norm inherited from $X$ is a complex Banach sequence space.
        \item\label{xiang-wutiaojianji} The sequence $(e_i)_{i\in\mathbb{N}}$ is a $1$-unconditional Schauder basis of $X$ if and only if
    $
        X=\overline{\mathbb{C}^{(\mathbb{N})}}^{\,X}.
    $
        \item\label{xiang-kefenxing} The space $X$ is separable if and only if
    $
        X=\overline{\mathbb{C}^{(\mathbb{N})}}^{\,X}.
    $
    \end{enumerate}
\end{lemma}
\begin{proof}

\proofstep{1}

    For every
$\textbf{z}\in X$,
    since
    $$
    |P_n\textbf{z}|\leqslant|\textbf{z}|,
    \qquad
    |Q_n\textbf{z}|\leqslant|\textbf{z}|,
    $$
    It follows from \eqref{shi-fanshudandiao} that \eqref{shi-shouweisuosuo} holds.

Fix $\textbf{z}\in\overline{\mathbb{C}^{(\mathbb{N})}}^{\,X}$ and
$\varepsilon>0$. Choose $\textbf{y}\in\mathbb{C}^{(\mathbb{N})}$ with
$\|\textbf{z}-\textbf{y}\|_X<\varepsilon$, and choose $N\in\mathbb{N}$
with $\mathrm{supp}\textbf{y}\subset\{1,\cdots,N\}$.
For $n\geqslant N$, one has $Q_n\textbf{y}=0$, so
\eqref{shi-shouweisuosuo} gives
$$
\|\textbf{z}-P_n\textbf{z}\|_X
=\|Q_n(\textbf{z}-\textbf{y})\|_X
\leqslant\|\textbf{z}-\textbf{y}\|_X<\varepsilon.
$$
This proves \eqref{shi-weibushoulian}.

\proofstep{2}

Let $\textbf{z}\in\overline{\mathbb{C}^{(\mathbb{N})}}^{\,X}$
and let $\textbf{w}\in\mathbb{C}^{\mathbb{N}}$ satisfy
$$
|\textbf{w}|\leqslant|\textbf{z}|.
$$
Since $X$ is a complex Banach sequence space, it follows that
$\textbf{w}\in X$. Moreover,
$$
|Q_n\textbf{w}|
\leqslant
|Q_n\textbf{z}|,
\qquad n=1,2,\cdots,
$$
and hence
$$
\|Q_n\textbf{w}\|_X
\leqslant
\|Q_n\textbf{z}\|_X,
\qquad n=1,2,\cdots.
$$
By \eqref{shi-weibushoulian}, we obtain
$$
\lim\limits_{n\to\infty}
\|P_n\textbf{w}-\textbf{w}\|_X
=
\lim\limits_{n\to\infty}
\|Q_n\textbf{w}\|_X
=
0.
$$
As $P_n\textbf{w}\in\mathrm{span}\{e_1,\cdots,e_n\}$ for every
$n\in\mathbb{N}$, it follows that
$
\textbf{w}\in
\overline{\mathbb{C}^{(\mathbb{N})}}^{\,X}.
$
Thus
$
\overline{\mathbb{C}^{(\mathbb{N})}}^{\,X}
$
equipped with the norm inherited from $X$ is a complex Banach sequence
space.

\proofstep{3}

The necessity of \itemref{xiang-wutiaojianji} is immediate. We prove the sufficiency. Assume that
$
X=\overline{\mathbb{C}^{(\mathbb{N})}}^{\,X}.
$
For any
$
\textbf{x}=(x_i)_{i\in\mathbb{N}}\in X, n\in\mathbb{N},
$
and any  $\lambda_1,\cdots,\lambda_n\in\mathbb{C}$ satisfying
$
|\lambda_i|\leqslant 1, i=1,\cdots,n,
$
we have
\begin{equation}
    \left\|
    \sum\limits_{i=1}^{n}\lambda_i x_i e_i
    \right\|_X
    \leqslant
    \left\|
    \sum\limits_{i=1}^{n}x_i e_i
    \right\|_X.\label{shi-zuobiaosuofang}
\end{equation}
On the other hand, by \eqref{shi-weibushoulian} and the assumption
$
X=\overline{\mathbb{C}^{(\mathbb{N})}}^{\,X},
$
it follows that $(e_i)_{i\in\mathbb{N}}$ is a Schauder basis of $X$. Together with the
 inequality \eqref{shi-zuobiaosuofang}, this shows that $(e_i)_{i\in\mathbb{N}}$ is a
$1$-unconditional Schauder basis of $X$. Thus the sufficiency of \itemref{xiang-wutiaojianji} follows.

\proofstep{4}

To prove \itemref{xiang-kefenxing}, suppose that $X$ is separable but
$\textbf{x}=\left( x_{j} \right)_{j\in \mathbb{N}}\in X\setminus\overline{\mathbb{C}^{(\mathbb{N})}}^{\,X}$.
Choose $\delta,(n_k)_{k\in \mathbb{N}_{0}}$ as in
\namedref{Lemma}{yinli-fenkuaixiajie}.
For each subset $A\subset\mathbb{N}$, define the sequence
$\textbf{x}^A=\left( x_{j}^{A} \right)_{j\in \mathbb{N}}$ by
$$
x_{j}^{A}=\begin{cases}x_{j},&j\in \bigcup\limits_{k\in A} \left\{ n_{k-1}+1,\cdots ,n_{k} \right\} ,\\ 0,&\text{otherwise} .\end{cases}
$$
Since $|\textbf{x}^A|\leqslant|\textbf{x}|$,
\eqref{shi-fanshudandiao} gives $\textbf{x}^A\in X$.
If $A\neq B$, choose a $k\in \mathbb{N}$ that belongs to exactly one of
$A$ and $B$.
It follows from \namedref{Lemma}{yinli-fenkuaixiajie} and \eqref{shi-fanshudandiao} that
$$
|(P_{n_k}-P_{n_{k-1}})\textbf{x}|
\leqslant|\textbf{x}^A-\textbf{x}^B|
\Rightarrow
\|\textbf{x}^A-\textbf{x}^B\|_X>\delta.
$$
Now $X$ contains an uncountable family of pairwise disjoint open sets
$
\left\{\textbf{x}^{A}+\frac{\delta}{3}B_X\right\}_{A\subset\mathbb{N}},
$
which contradicts the separability of $X$.

Conversely, suppose that
$
X=\overline{\mathbb{C}^{(\mathbb{N})}}^{\,X}.
$
Under the topology of $X$, since the countable set $\left( \mathbb{Q} +\mathrm{i}\mathbb{Q} \right)^{(\mathbb{N})}$
is dense in $\mathbb{C}^{(\mathbb{N})}$, it is also dense in $X$.
Therefore, $X$ is separable.

\end{proof}

\label{duan-yudedingyi}

Let $O\subset X$ be a nonempty open set. We call $O$ a Reinhardt domain if,
for every $\textbf{x}\in O$ and every $\textbf{y}\in X$ satisfying
$|\textbf{y}|\leqslant|\textbf{x}|$, one has $\textbf{y}\in O$.
This is also the
convention in \cite[Section~1, p.~15]{DMP09}.
Throughout the remainder of the paper, $O\subset X$ is a
Reinhardt domain. In particular, $\textbf{0}\in O$.
All holomorphic functions are complex-valued and holomorphic in the
sense of complex Fr\'echet differentiability.

For the space $H^\infty(O)$ introduced in
\namedref{Section}{jie-yinyan}, we use the supremum norm
\begin{equation}
    H^\infty(O)\triangleq
    \{f\in H(O):\|f\|_\infty<\infty\},
    \qquad
    \|f\|_\infty\triangleq
    \sup\limits_{\textbf{x}\in O}|f(\textbf{x})|.
    \label{shi-youjiequanchun}
\end{equation}
Then $H^\infty(O)$ is a Banach space with respect to the norm
\eqref{shi-youjiequanchun}.

A subset $A\subset O$ is called $O$-bounded if it is bounded in $X$ and
$
\mathrm{dist}(A,X\setminus O)>0,
$
where the distance is taken with respect to the norm of $X$, with the
convention that the distance to the empty set is $+\infty$.
We define
$$
H_b(O)\triangleq
\left\{
f\in H(O):
\sup\limits_{\textbf{x}\in A}|f(\textbf{x})|<\infty
\text{ for every $O$-bounded subset }A\subset O
\right\}.
$$
The space $H_b(O)$ is endowed with the locally convex topology generated
by the seminorms
$$
p_A(f)\triangleq
\sup\limits_{\textbf{x}\in A}|f(\textbf{x})|,
\qquad
A\subset O,\ \text{$O$-bounded}.
$$
Thus convergence in $H_b(O)$ is precisely uniform convergence on every
$O$-bounded subset of $O$.

For $r>0$ and a function defined on $rB_X$, write
$$
    \|f\|_r\triangleq
    \sup\limits_{\textbf{x}\in rB_X}|f(\textbf{x})|.
$$
In particular,
\begin{equation}
    \begin{aligned}
        H_b(X)
        &=
        \{f\in H(X):\|f\|_r<\infty
        \text{ for every }r>0\},\\
        H_b(B_X)
        &=
        \{f\in H(B_X):\|f\|_r<\infty
        \text{ for every }0<r<1\}.
    \end{aligned}
    \label{shi-youjiexing}
\end{equation}
Moreover, $H_b(X)$ and $H_b(B_X)$ are Fr\'echet spaces
(see \cite[Section~2, p.~2448]{CMV20}), and their topologies are generated,
respectively, by the seminorm families
\begin{equation}
    \bigl(\|\cdot\|_k\bigr)_{k\in\mathbb{N}}
    \qquad\text{and}\qquad
    \bigl(\|\cdot\|_{\frac{k}{k+1}}\bigr)_{k\in\mathbb{N}}.
    \label{shi-banfanshu}
\end{equation}
The local uniform limit of holomorphic functions is holomorphic; cf.
\cite[Proposition~9.13, p.~72]{Mu86}. Thus
\begin{equation}
    f_\nu\in H(O),\quad
    f_\nu\longrightarrow f\text{ locally uniformly on }O
    \quad\Longrightarrow\quad
    f\in H(O).
    \label{shi-quanchunjixian}
\end{equation}
For standard references on infinite-dimensional complex analysis, see
\cite{Di81,Di99,Mu86}.

For $\alpha\in\mathbb{N}_0^{(\mathbb{N})}$, we set
$$
|\alpha|
\triangleq
\sum\limits_{i=1}^{\infty}\alpha_i,
\qquad
\alpha!
\triangleq
\prod\limits_{i=1}^{\infty}\alpha_i!,
$$
and, for $\textbf{z}=(z_i)_{i\in\mathbb{N}}\in\mathbb{C}^{\mathbb{N}}$,
$$
\textbf{z}^{\alpha}
\triangleq
\prod\limits_{i=1}^{\infty}z_i^{\alpha_i}.
$$
We use the conventions $0^0=1,0\log 0=0,0!=1$ and that an empty product equals $1$.

For $\alpha\in\mathbb{N}_0^{(\mathbb{N})}$, $f\in H(O)$,
and any $n\in\mathbb{N}$ such that
$\mathrm{supp}\alpha\subset\{1,\cdots,n\}$, define
$$
c_{\alpha}(f)
\triangleq
\left.
\frac{1}{\alpha!}
\frac{\partial^{|\alpha|}}
{\partial z_1^{\alpha_1}\partial z_2^{\alpha_2}\cdots\partial z_n^{\alpha_n}}
f\left(\sum\limits_{i=1}^{n}z_i e_i\right)
\right|_{z_1=\cdots=z_n=0}.
$$
It is immediate that $c_{\alpha}(f)$ is independent of the choice of $n$.

The following proposition identifies when an absolutely convergent
monomial series represents the original function.
\begin{proposition}\label{mingti-jishubiaoshi}
    Let $f\in H(O)$ and let
    $\textbf{z}_0\in O\cap\overline{\mathbb{C}^{(\mathbb{N})}}^{\,X}$ satisfy
    \begin{equation}
        \sum\limits_{\alpha\in\mathbb{N}_0^{(\mathbb{N})}}
        \left|c_{\alpha}(f)\textbf{z}_0^{\alpha}\right|<\infty.
        \label{shi-juduishoulianjiading}
    \end{equation}
    Then, for every $\textbf{z}\in\mathbb{C}^{\mathbb{N}}$ satisfying
    $|\textbf{z}|\leqslant|\textbf{z}_0|$, one has
    \begin{equation}
        \sum\limits_{\alpha\in\mathbb{N}_0^{(\mathbb{N})}}
        \left|c_{\alpha}(f)\textbf{z}^{\alpha}\right|<\infty,
        \label{shi-juduishoulian}
    \end{equation}
    and
    \begin{equation}
        f(\textbf{z})
        =
        \sum\limits_{\alpha\in\mathbb{N}_0^{(\mathbb{N})}}
        c_{\alpha}(f)\textbf{z}^{\alpha}.
        \label{shi-hanshubiaoshi}
    \end{equation}
\end{proposition}

\begin{proof}

Fix $\textbf{z}\in\mathbb{C}^{\mathbb{N}}$ with
$|\textbf{z}|\leqslant|\textbf{z}_0|$. By \namedref{Lemma}{yinli-xuliekongjian}~\itemref{xiang-bibaokongjian},
$\textbf{z}\in\overline{\mathbb{C}^{(\mathbb{N})}}^{\,X}$.
The defining property of $O$ gives $\textbf{z}\in O$.
Since $|\textbf{z}^{\alpha}|\leqslant|\textbf{z}_0^{\alpha}|$ for
every multi-index $\alpha$, \eqref{shi-juduishoulianjiading} implies \eqref{shi-juduishoulian}.

For each $n\in\mathbb{N}$, the defining property of $O$ gives
$|P_n\textbf{z}|\in O$. Since $O$ is open, there exists
$\varepsilon>0$ such that
$
\sum\limits_{j=1}^n (|z_j|+\varepsilon)e_j\in O.
$
By the Reinhardt property of $O$, the polydisc
$$
\left\{
\sum\limits_{j=1}^n \zeta_j e_j:
|\zeta_j|<|z_j|+\varepsilon,\quad 1\leqslant j\leqslant n
\right\}
$$
is contained in $O$. Hence the Taylor series at the origin of the
restriction of $f$ to $\mathrm{span} \left\{ e_{1},e_{2},\cdots ,e_{n} \right\}$ converges absolutely at $P_n\textbf{z}$,
and therefore
\begin{equation}
    f(P_n\textbf{z})=
    \sum\limits_{\substack{\alpha\in\mathbb{N}_0^{(\mathbb{N})}\\
            \mathrm{supp}\alpha\subset\{1,\cdots,n\}}}
    c_{\alpha}(f)\textbf{z}^{\alpha}.
    \label{shi-touyingzhankai}
\end{equation}
Here the finite-dimensional coefficients agree with $c_\alpha(f)$
by their definition; see also
\cite[Corollary~7.8, pp.~49--50]{Mu86}.
By \eqref{shi-weibushoulian} in \namedref{Lemma}{yinli-xuliekongjian}  and the continuity of $f$,
$f(P_n\textbf{z})\to f(\textbf{z})$.
On the other hand, absolute convergence in \eqref{shi-juduishoulian} gives
$$
\begin{aligned}
&\left|
\sum\limits_{\alpha\in\mathbb{N}_0^{(\mathbb{N})}}
 c_\alpha(f)\textbf{z}^{\alpha}
-\sum\limits_{\substack{\alpha\in\mathbb{N}_0^{(\mathbb{N})}\\
\mathrm{supp}\alpha\subset\{1,\cdots,n\}}}
 c_\alpha(f)\textbf{z}^{\alpha}\right|\leqslant
\sum\limits_{\substack{\alpha\in\mathbb{N}_0^{(\mathbb{N})}\\
\mathrm{supp}\alpha\not\subset\{1,\cdots,n\}}}
 |c_\alpha(f)\textbf{z}^{\alpha}|\longrightarrow0.
\end{aligned}
$$
Passing to the limit in \eqref{shi-touyingzhankai} proves
\eqref{shi-hanshubiaoshi}.
\end{proof}

\label{duan-qiciduoxiangshi}

We define
$
\mathcal{P}^{0}(X)\triangleq \mathbb{C},
$
and, for each $k\in\mathbb{N}$, let $\mathcal{P}^{k}(X)$ denote the space of all continuous $k$-homogeneous polynomials on $X$.

For $m\in\mathbb{N}_0$ and $P\in\mathcal{P}^m(X)$, we use the norm
\begin{equation}
\|P\|_{\mathcal{P}^m(X)}\triangleq\sup\limits_{\textbf{z}\in B_X}|P(\textbf{z})|
=\sup\limits_{\substack{\textbf{z}\in X\\\|\textbf{z}\|_X\leqslant1}}|P(\textbf{z})|<\infty.
\label{shi-duoxiangshifanshu}
\end{equation}
For $m=0$ this is the absolute value of the constant polynomial. The space $\mathcal{P}^m(X)$ is a
Banach space with respect to this norm. 
We omit the subscript when the ambient space and degree are clear.
For $m\in\mathbb{N}_0$, $n\in\mathbb{N}$, and $P\in\mathcal{P}^m(X)$ satisfying
$P=P\circ P_n$, its restriction to
$\mathrm{span}\{e_1,\cdots,e_n\}$ is a finite-dimensional
$m$-homogeneous polynomial. The definition of $c_\alpha$ therefore gives
\begin{equation}
P(\textbf{z})=
\sum\limits_{\substack{|\alpha|=m\\
\mathrm{supp}\alpha\subset\{1,\cdots,n\}}}
c_\alpha(P)\textbf{z}^\alpha,
\qquad \textbf{z}\in X.
\label{shi-youxianzhankai}
\end{equation}

The following lemma is the version, under the sequence-space
convention used here, of \cite[Lemma~4.2, pp.~953--954]{GMMS21};
see also \cite[Lemma~4.1, pp.~26--27]{DMP09} for the bounded
holomorphic case. We include a proof.

\begin{lemma}\label{yinli-bituguji}
Let $O=X$ or $O=B_X$, and let
$\textbf{z}\in\mathrm{dom}H_b(O)$. There exist $A>0$ and $r>0$, with $r<1$ if $O=B_X$, such that for every
$f\in H_b(O)$,
\begin{equation}
	\sum\limits_{\alpha\in\mathbb{N}_0^{(\mathbb{N})}}
	|c_\alpha(f)\textbf{z}^\alpha|
	\leqslant A\|f\|_r .
	\label{shi-bituguji}
\end{equation}
In particular, if $O=X$, there is $K\geqslant1$ such that for every
$m\in\mathbb{N}, P\in\mathcal{P}^m(X)$,
\begin{equation}
\sum\limits_{|\alpha|=m}|c_\alpha(P)\textbf{z}^\alpha|
\leqslant K^m\|P\|_{\mathcal{P}^m(X)}.
\label{shi-dianchuzhishijie}
\end{equation}
\end{lemma}

\begin{proof}
First, every coefficient functional $f\mapsto c_\alpha(f)$ is
continuous on $H_b(O)$. To see this, fix $s>0$, with $s<1$ when
$O=B_X$. Choose $n\geqslant1$ such that
$\mathrm{supp}\alpha\subset\{1,\cdots,n\}$.
Since, for every $\left| w_{j} \right| \leqslant \frac{s}{2n\| e_{j}\|_{X}} ,j=1,2,\cdots ,n$, we have:
$$
\left\|\sum\limits_{j=1}^{n}w_je_j\right\|_X
\leqslant\sum\limits_{j=1}^{n}\frac{s}{2n\| e_{j}\|_{X}}\|e_j\|_X=\frac{s}{2}.
$$
The Cauchy integral formula gives
$$
\begin{aligned}\left| c_{\alpha}(f) \right|&=\left| \frac{1}{(2\pi \mathrm{i})^{n}} \int\limits_{|w_{1}|=\frac{s}{2n\| e_{1}\|_{X}}} \cdots \int\limits_{|w_{n}|=\frac{s}{2n\| e_{n}\|_{X}}} \frac{f\left( \sum\limits_{j=1}^{n} w_{j}e_{j} \right)}{\prod\limits_{j=1}^{n} w_{j}^{\alpha_{j} +1}} \, \mathrm{d} w_{n}\cdots \mathrm{d} w_{1} \right|\\ &\leqslant \frac{1}{(2\pi )^{n}} \int\limits_{|w_{1}|=\frac{s}{2n\| e_{1}\|_{X}}} \cdots \int\limits_{|w_{n}|=\frac{s}{2n\| e_{n}\|_{X}}} \frac{\left| f\left( \sum\limits_{j=1}^{n} w_{j}e_{j} \right) \right|}{\prod\limits_{j=1}^{n} \left| w_{j} \right|^{\alpha_{j} +1}} \, \left| \mathrm{d} w_{n} \right| \cdots \left| \mathrm{d} w_{1} \right|\\ &\leqslant \frac{\|f\|_{s}}{\left( 2\pi \right)^{n}} \prod\limits_{j=1}^{n} \frac{2\pi \left( \frac{s}{2n\| e_{j}\|_{X}} \right)}{\left( \frac{s}{2n\| e_{j}\|_{X}} \right)^{\alpha_{j} +1}} =\|f\|_{s} \prod\limits_{j=1}^{n} \frac{1}{\left( \frac{s}{2n\| e_{j}\|_{X}} \right)^{\alpha_{j}}}.\end{aligned}
$$
Thus $f\mapsto c_\alpha(f)$ is continuous with respect to the
seminorm $\|\cdot\|_s$, and therefore continuous on $H_b(O)$.

The set $\mathbb{N}_0^{(\mathbb{N})}$ is countable. Let
$\ell^1(\mathbb{N}_0^{(\mathbb{N})})$ be the Banach space of
absolutely summable complex families indexed by this set, with norm
$\|(a_\alpha)_\alpha\|_1=\sum\limits_{\alpha\in \mathbb{N}_{0}^{(\mathbb{N})}}|a_\alpha|$.
 $\textbf{z}\in\mathrm{dom}H_b(O)$ makes
$$
T_{\textbf{z}}:H_b(O)\longrightarrow
\ell^1(\mathbb{N}_0^{(\mathbb{N})}),
\qquad
T_{\textbf{z}}f=(c_\alpha(f)\textbf{z}^\alpha)_{\alpha\in\mathbb{N}_{0}^{(\mathbb{N})}},
$$
a well-defined linear map. Its graph is closed. Indeed, if
$f_j\to f$ in $H_b(O)$ and $T_{\textbf{z}}f_j\to(a_\alpha)_\alpha$
in $\ell^1(\mathbb{N}_{0}^{(\mathbb{N})})$. For every $\alpha \in \mathbb{N}_{0}^{(\mathbb{N})}$, we have
$$c_{\alpha}\left( f \right) \textbf{z}^{\alpha} =\lim\limits_{j\to \infty} c_{\alpha}\left( f_{j} \right) \textbf{z}^{\alpha} =\lim\limits_{j\to \infty} \left( T_{\textbf{z}}f_{j} \right)_{\alpha} =a_{\alpha}.$$

Both spaces are Fr\'echet. Hence the closed graph theorem \cite[52.10 and 52.12, p.~577]{KM97} implies that $T_{\textbf{z}}$ is continuous.

The seminorms in
\eqref{shi-banfanshu} form an increasing sequence. Hence, there exist
$A>0$ and $r>0$, with $r<1$ if $O=B_X$, such that
\eqref{shi-bituguji} holds.

Assume now that $O=X$ and let
$P\in\mathcal{P}^m(X)\subset H_{b}\left( X \right)$, where $m\in\mathbb{N}$. Since $P$ is
$m$-homogeneous, $c_\alpha(P)=0$ whenever $|\alpha|\neq m$, and
homogeneity gives
$\|P\|_r=r^m\|P\|_{\mathcal{P}^m(X)}$. Hence
$$
\sum\limits_{|\alpha|=m}|c_\alpha(P)\textbf{z}^{\alpha}|
\leqslant
Ar^m\|P\|_{\mathcal{P}^m(X)}.
$$
Set
$K\triangleq\max\{1,A\}\max\{1,r\}$. Since $m\geqslant1$,
$$
Ar^m
\leqslant
\max\{1,A\}^m\max\{1,r\}^m
=
K^m.
$$
Therefore
$$
\sum\limits_{|\alpha|=m}|c_\alpha(P)\textbf{z}^{\alpha}|
\leqslant
K^m\|P\|_{\mathcal{P}^m(X)},
$$
which proves
\eqref{shi-dianchuzhishijie}.
\end{proof}

\section{General principles of monomial convergence}
\label{jie-yibanyuanli}

\begin{theorem}\label{dingli-ercizhangai}
    If $\mathbb{C}^{(\mathbb{N})}$ is not dense in $X$, then  for each $\textbf{x} \in X\setminus \overline{\mathbb{C}^{(\mathbb{N})}}^{\,X}$, there exists
    $G\in\mathcal{P}^{2}(X)$ such that
    \begin{equation}
        \sum\limits_{\substack{
                \alpha\in\mathbb{N}_0^{(\mathbb{N})}\\
                |\alpha|=2
        }}
        \left|c_{\alpha}(G)\textbf{x}^{\alpha}\right|
        =+\infty.
        \label{shi-ercifasan}
    \end{equation}
\end{theorem}

\begin{proof}

    \proofstep{1}

By \namedref{Lemma}{yinli-fenkuaixiajie}, choose $\delta>0$
and integers $0=n_0<n_1<n_2<\cdots$ such that
$$
\|(P_{n_k}-P_{n_{k-1}})\textbf{x}\|_X>\delta,
\qquad k\in\mathbb{N}.
$$

    For $k\in\mathbb{N}$, put
    $$
    E_k
    \triangleq
    \{n_{k-1}+1,\cdots,n_k\},
    \qquad
    P_{E_{k}}\textbf{u}
    \triangleq
    \sum\limits_{i\in E_k}u_i e_i.
    $$
    Then the sets $E_k$ are pairwise disjoint and
$
    \bigcup\limits_{k=1}^{\infty}E_k=\mathbb{N}.
$
    Moreover,
$
    \|P_{E_{k}}\textbf{u}\|_X
    \leqslant
    \|\textbf{u}\|_X,   \textbf{u}\in X,
$
    and
$
    \|P_{E_{k}}\textbf{x}\|_X>\delta.
    $

    By the complex Hahn--Banach theorem, for each $k\in\mathbb{N}$ there
    exists $\psi_k\in X^*$ such that
    $$
    \|\psi_k\|_{X^*}=1,
    \qquad
    \psi_k(P_{E_{k}}\textbf{x})
    =
    \|P_{E_{k}}\textbf{x}\|_X.
    $$

    Then, for every $\textbf{u}\in X$,
\begin{equation}
    \psi_{k} (P_{E_{k}}\textbf{u} )=\sum\limits_{i\in E_{k}} \psi_{k} \left( e_{i} \right) u_{i}.\label{shi-fenkuaifanhan}
\end{equation}
    In particular,
    \begin{equation}\label{shi-fenkuaixiajie}
    \sum\limits_{i\in E_{k}} |\psi_{k} \left( e_{i} \right) x_{i}|\geqslant \left| \sum\limits_{i\in E_{k}} \psi_{k} \left( e_{i} \right) x_{i} \right| =|\psi_{k} (P_{E_k}\textbf{x} )|>
        \delta,
        \qquad
        k\in\mathbb{N}.
    \end{equation}

    \proofstep{2}

    For $j\in\mathbb{N}$, set
    $$
h_{r,s}^{(j)}\triangleq e^{\frac{2\pi \mathrm{i} (r-1)(s-1)}{4^{j}}},\qquad 1\leqslant r,s\leqslant 4^{j}.
    $$
    Then
    $
    |h_{r,s}^{(j)}|=1
$
    and
    $$
    \sum\limits_{r=1}^{4^{j}} \overline{h_{r,s}^{(j)}} h_{r,t}^{(j)}=\begin{cases}4^{j},&s=t,\\ 0,&s\neq t.\end{cases}
    $$
    Consequently, for every
    $\xi_1,\cdots,\xi_{4^{j}}\in\mathbb{C}$,
    \begin{equation}\label{shi-zhengjiaohengdeng}
    \begin{aligned}\sum\limits_{r=1}^{4^{j}} \left| \sum\limits_{s=1}^{4^{j}} h_{r,s}^{(j)}\xi_{s} \right|^{2}&=\sum\limits_{r=1}^{4^{j}} \left( \sum\limits_{s=1}^{4^{j}} h_{r,s}^{(j)}\xi_{s} \sum\limits_{s^{\prime}=1}^{4^{j}} \overline{h_{r,s^{\prime}}^{(j)}} \cdot \overline{\xi_{s^{\prime}}} \right)\\ &=\sum\limits_{s=1}^{4^{j}} \sum\limits_{s^{\prime}=1}^{4^{j}} \sum\limits_{r=1}^{4^{j}} \left( h_{r,s}^{(j)}\overline{h_{r,s^{\prime}}^{(j)}} \xi_{s} \overline{\xi_{s^{\prime}}} \right)\\ &=\sum\limits_{s=1}^{4^{j}} \left( 4^{j}\xi_{s} \overline{\xi_{s}} \right) =4^{j}\sum\limits_{s=1}^{4^{j}} |\xi_{s} |^{2}.\end{aligned}
    \end{equation}

    Put
    $$
    d_0\triangleq0,
    \qquad
    d_j\triangleq d_{j-1}+2\cdot4^{j},
    \qquad
    j\in\mathbb{N}.
    $$
    For every $j\in\mathbb{N}$, define a bilinear form
    $A_j:X\times X\to\mathbb{C}$ by
    $$
A_{j}(\textbf{u} ,\textbf{v} )\triangleq \frac{1}{8^{j}} \sum\limits_{r,s=1}^{4^{j}} h_{r,s}^{(j)}\psi_{d_{j-1}+r} \left( P_{E_{d_{j-1}+r}}\textbf{u} \right) \psi_{d_{j-1}+4^{j}+s} \left( P_{E_{d_{j-1}+4^{j}+s}}\textbf{v} \right) .
    $$
    By the Cauchy--Schwarz inequality and
    \eqref{shi-zhengjiaohengdeng},
    $$
\begin{aligned}|A_{j}(\textbf{u} ,\textbf{v} )|&\leqslant \frac{1}{8^{j}} \left( \sum\limits_{r=1}^{4^{j}} |\psi_{d_{j-1}+r} \left( P_{E_{d_{j-1}+r}}\textbf{u} \right) |^{2} \right)^{\frac{1}{2}}\\ &\qquad \cdot \left( \sum\limits_{r=1}^{4^{j}} \left| \sum\limits_{s=1}^{4^{j}} h_{r,s}^{(j)}\psi_{d_{j-1}+4^{j}+s} \left( P_{E_{d_{j-1}+4^{j}+s}}\textbf{v} \right) \right|^{2} \right)^{\frac{1}{2}}\\ &=\frac{1}{4^{j}} \left( \sum\limits_{r=1}^{4^{j}} |\psi_{d_{j-1}+r} \left( P_{E_{d_{j-1}+r}}\textbf{u} \right) |^{2} \right)^{\frac{1}{2}} \left( \sum\limits_{s=1}^{4^{j}} |\psi_{d_{j-1}+4^{j}+s} \left( P_{E_{d_{j-1}+4^{j}+s}}\textbf{v} \right) |^{2} \right)^{\frac{1}{2}}\\ &\leqslant \| \textbf{u} \|_{X} \| \textbf{v} \|_{X} .\end{aligned}
    $$

    Hence the series
    $$
    A(\textbf{u},\textbf{v})
    \triangleq
    \sum\limits_{j=1}^{\infty}2^{-j}A_j(\textbf{u},\textbf{v})
    $$
    converges absolutely for every
    $(\textbf{u},\textbf{v})\in X\times X$, and
    $$
    |A(\textbf{u},\textbf{v})|
    \leqslant
    \|\textbf{u}\|_X\|\textbf{v}\|_X.
    $$
    Thus $A$ is a continuous bilinear form on $X\times X$.

    Define
$
    G(\textbf{u})
    \triangleq
    A(\textbf{u},\textbf{u}),   \textbf{u}\in X.
$
    Then
$
    G\in\mathcal{P}^{2}(X)
$
    and
    $$
    |G(\textbf{u})|
    \leqslant
    \|\textbf{u}\|_X^2,
    \qquad
    \textbf{u}\in X.
    $$

    \proofstep{3}

On every finite-dimensional coordinate subspace, only finitely many
summands in the definition of $G$ are nonzero. By \eqref{shi-fenkuaifanhan},

    $$
\begin{aligned}G\left( \textbf{u} \right)&=\sum\limits_{j=1}^{\infty} 2^{-j}A_{j}(\textbf{u} ,\textbf{u} )=\sum\limits_{j=1}^{\infty} \frac{1}{16^{j}} \sum\limits_{r,s=1}^{4^{j}} h_{r,s}^{(j)}\psi_{d_{j-1}+r} \left( P_{E_{d_{j-1}+r}}\textbf{u} \right) \psi_{d_{j-1}+4^{j}+s} \left( P_{E_{d_{j-1}+4^{j}+s}}\textbf{u} \right)\\ &=\sum\limits_{j=1}^{\infty} \frac{1}{16^{j}} \sum\limits_{r,s=1}^{4^{j}} h_{r,s}^{(j)}\sum\limits_{\begin{gathered}p\in E_{d_{j-1}+r}\\ q\in E_{d_{j-1}+4^{j}+s}\end{gathered}} \psi_{d_{j-1}+r} \left( e_{p} \right) \psi_{d_{j-1}+4^{j}+s} \left( e_{q} \right) u_{p}u_{q}.\end{aligned}
    $$

    Let $j\in\mathbb{N}$ and $1\leqslant r,s\leqslant 4^{j}$, and suppose that
    $$
    p\in E_{d_{j-1}+r},
    \qquad
    q\in E_{d_{j-1}+4^{j}+s}.
    $$
    Then $p<q$. Moreover, for every fixed pair $p<q,p,q\in \mathbb{N}$, there is at most one
    triple $(j,r,s)$ satisfying these conditions.  Indeed, let $k_p,k_q\in\mathbb{N}$ be uniquely determined by
    $p\in E_{k_p}$ and $q\in E_{k_q}$. If
    $$
    p\in E_{d_{j-1}+r},
    \qquad
    q\in E_{d_{j-1}+4^j+s},
    $$
    then
    $$
    k_p=d_{j-1}+r,
    \qquad
    k_q=d_{j-1}+4^j+s.
    $$
    Since the intervals
$
    \{d_{j-1}+1,\cdots,d_j\},j\in\mathbb{N},
  $
    are pairwise disjoint, the first equality determines $j$ and $r$ uniquely, and the
    second then determines $s$ uniquely. Hence the monomial
    $u_pu_q$ occurs in exactly one of the corresponding summands, and
    \begin{equation}\label{shi-fasanduoxiangshixishu}
    c_{e_{p}+e_{q}}(G)=\frac{1}{16^{j}} h_{r,s}^{(j)}\psi_{d_{j-1}+r} \left( e_{p} \right) \psi_{d_{j-1}+4^{j}+s} \left( e_{q} \right) .
    \end{equation}
    All the remaining monomial coefficients of degree $2$ are zero.

    Using \eqref{shi-fasanduoxiangshixishu} and
    \eqref{shi-fenkuaixiajie}, we obtain
    $$
    \begin{aligned}[b]
\sum\limits_{\substack{
                \alpha\in\mathbb{N}_0^{(\mathbb{N})}\\
                |\alpha|=2
        }}
        |c_\alpha(G)\textbf{x}^{\alpha}|
        &=\sum\limits_{j=1}^{\infty} \frac{1}{16^{j}} \sum\limits_{r,s=1}^{4^{j}} \left[ \sum\limits_{p\in E_{d_{j-1}+r}} \left| \psi_{d_{j-1}+r} \left( e_{p} \right) x_{p} \right| \cdot \sum\limits_{q\in E_{d_{j-1}+4^{j}+s}} \left| \psi_{d_{j-1}+4^{j}+s} \left( e_{q} \right) x_{q} \right| \right]
        \\&\geqslant
    \sum\limits_{j=1}^{\infty} \frac{1}{16^{j}} \sum\limits_{r,s=1}^{4^{j}} \delta^{2} =\sum\limits_{j=1}^{\infty} \delta^{2} =+\infty.
    \end{aligned}
    $$
\end{proof}

\begin{remark}\label{zhuji-bibaobaohan}
For every Reinhardt domain $O\subset X$, \namedref{Theorem}{dingli-ercizhangai} gives for each $\mathcal{P}^{2} \left( X \right) \subset \mathcal{F} \left( O \right) \subset H\left( O \right)$,
\begin{equation}
\mathrm{dom} \mathcal{F}(O)
\subset\overline{\mathrm{span}\{e_j:j\geqslant1\}}^{\,X}.
\label{shi-bibaobaohan}
\end{equation}
This gives an affirmative answer
to \namedref{Problem}{wenti-bibaobaohan}.
\end{remark}

For $n\in\mathbb{N}_0$, let $\mathscr{P}_n$ be the space of
continuous polynomials on $X$ depending on finitely many coordinates
$x_j$ with $j>n$. Constants are included in every
$\mathscr{P}_n$.

For $n\in\mathbb{N}_0$, put
$$
Y_n\triangleq Q_nX=\ker P_n,
\qquad
\Lambda_n\triangleq\{\beta\in\mathbb{N}_0^{(\mathbb{N})}:
\beta_1=\cdots=\beta_n=0\}.
$$
The space $Y_n$ is closed and has the norm inherited from $X$. For $n=0$, one has $Y_0=X$ and
$\Lambda_0=\mathbb{N}_0^{(\mathbb{N})}$.

For $r>0$, $g\in H(rB_{Y_n})$, and $\beta\in\Lambda_n$, define
$c_\beta(g)\triangleq c_\beta(g\circ Q_n)$, where
$g\circ Q_n\in H(rB_X)$ by \eqref{shi-shouweisuosuo}.
This agrees with differentiation with respect to the coordinate variables $x_j,j>n$. 
\begin{lemma}\label{yinli-weibuguji}
	Suppose that $\overline{\mathbb{C}^{(\mathbb{N})}}^{\,X}=X$ and
	$\textbf{z}\in\mathrm{dom}H(X)$. For every $r>0$, there exist
	$N\in\mathbb{N}$ and $C>0$ such that for every $n\geqslant N, p\in\mathscr{P}_n$,
\begin{equation}
	\sum\limits_{\alpha\in\mathbb{N}_0^{(\mathbb{N})}}
	|c_\alpha(p)\textbf{z}^\alpha|
	\leqslant C\|p\|_r.
	\label{shi-weibutongyijie}
\end{equation}
\end{lemma}

\begin{proof}
	
	Fix $r>0$. For $p\in\mathscr{P}_0$, set
	$$
	S_{\textbf{z}}(p)\triangleq
	\sum\limits_{\alpha\in\mathbb{N}_0^{(\mathbb{N})}}
	|c_\alpha(p)\textbf{z}^{\alpha}|<\infty.
	$$

	Since $\mathscr{P}_{n+1}\subset\mathscr{P}_n$, it suffices to prove
	the estimate \eqref{shi-weibutongyijie} for one value of $n$.
	
\proofstep{1}
	
	Suppose that no such estimate holds. Set $m_0=0$. For $j\in\mathbb{N}$, after $m_{j-1}$
	has been chosen, the failure of \eqref{shi-weibutongyijie} on
	$\mathscr{P}_{m_{j-1}}$ gives a nonzero 
	$q_j\in\mathscr{P}_{m_{j-1}}$ such that
	$$
	S_{\textbf{z}}(q_j)
	>
	(1+2^{j+1})\|q_j\|_r.
	$$
	Choose $m_j>m_{j-1}$ so that $q_j$ depends only on the coordinates
	$x_{m_{j-1}+1},\cdots ,x_{m_{j}}$, and define
	$$
	p_j\triangleq
	\frac{q_j-q_j(\textbf{0})}{2^{j+1}\|q_j\|_r}.
	$$
	Since $|q_j(\textbf{0})|\leqslant\|q_j\|_r$, we have
	\begin{equation}
		p_j(\textbf{0})=0,
		\qquad
		\|p_j\|_r\leqslant2^{-j},
		\qquad
		S_{\textbf{z}}(p_{j})=\frac{S_{\textbf{z}}(q_{j})-\left| q_{j}\left( \textbf{0} \right) \right|}{2^{j+1}\|q_j\|_r} >1.
		\label{shi-fasanxiangguifan}
	\end{equation}
	
	Consider the series
	\begin{equation}
		F(\textbf{x})\triangleq\sum\limits_{j=1}^{\infty}p_j(\textbf{x}).
		\label{shi-zhenghanshupinjie}
	\end{equation}
	Fix $\textbf{a}\in X$. By \eqref{shi-weibushoulian}, choose
	$N_0\in\mathbb{N}$ with $\|Q_{N_0}\textbf{a}\|_X<\frac{r}{3}$, and then
	choose $j_0$ so that $m_{j-1}\geqslant N_0$ for every $j\geqslant j_0$.
	If
	$\textbf{x}\in\textbf{a}+\frac{r}{3} B_X$ and $j\geqslant j_0$, then
	$$
	\begin{aligned}
		\|(P_{m_j}-P_{m_{j-1}})\textbf{x}\|_X
		&\leqslant\|Q_{N_0}\textbf{x}\|_X\leqslant\|Q_{N_0}\textbf{a}\|_X+\|\textbf{x}-\textbf{a}\|_X
		<\frac{2r}{3}.
	\end{aligned}
	$$
	The first inequality follows from $\left| ( P_{m_{j}}-P_{m_{j-1}}) \textbf{x} \right| \leqslant \left| Q_{N_{0}}\textbf{x} \right|$ and \eqref{shi-fanshudandiao}. Since $p_j$ depends only on the coordinates $x_{m_{j-1}+1},\cdots ,x_{m_{j}}$, \eqref{shi-fasanxiangguifan} gives
	$$
		|p_j(\textbf{x})|
		=
		\bigl|p_j((P_{m_j}-P_{m_{j-1}})\textbf{x})\bigr|
		\leqslant2^{-j},
		\quad
		\textbf{x}\in\textbf{a}+\frac{r}{3} B_X,\quad j\geqslant j_0.
	$$
	Thus \eqref{shi-zhenghanshupinjie} converges locally uniformly
	on $X$. Since every $p_j$ is a continuous polynomial,
	\eqref{shi-quanchunjixian} gives $F\in H(X)$.
	
\proofstep{2}

We have
	$$
\begin{aligned}\sum\limits_{\alpha \in \mathbb{N}_{0}^{(\mathbb{N} )}} |c_{\alpha}(F)\textbf{z}^{\alpha} |&=\left| \sum\limits_{j=1}^{\infty} p_{j}\left( \textbf{0} \right) \right| +\sum\limits_{j=1}^{\infty} \left( S_{\textbf{z}}(p_{j})-\left| p_{j}\left( \textbf{0} \right) \right| \right)\\ &=\sum\limits_{j=1}^{\infty} S_{\textbf{z}}(p_{j})=+\infty .\end{aligned}
	$$
	This contradicts $\textbf{z}\in\mathrm{dom}H(X)$.
	Thus \eqref{shi-weibutongyijie} holds.
\end{proof}
\begin{proposition}\label{mingti-jubuhua}
Suppose that $\overline{\mathbb{C}^{(\mathbb{N})}}^{\,X}=X$.
For every Reinhardt domain $O\subset X$,
\begin{equation}
\mathrm{dom}H(O)=O\cap\mathrm{dom}H(X).
\label{shi-jubuhua}
\end{equation}
More precisely, for each $\textbf{z}\in O$ the following conditions
are equivalent:
\begin{enumerate}
\item\label{xiang-yuneishoulian} $\textbf{z}\in\mathrm{dom}H(O)$.
\item\label{xiang-zhenghanshushoulian} $\textbf{z}\in\mathrm{dom}H(X)$.
\item\label{xiang-weibupanju} For every $r>0$, there exist $N\in\mathbb{N}$ and $C>0$ such that
\begin{equation}
\sum\limits_{\alpha\in\mathbb{N}_0^{(\mathbb{N})}}
|c_\alpha(p)\textbf{z}^\alpha|
\leqslant C\|p\|_r,
\qquad p\in\mathscr{P}_N.
\label{shi-weibupanju}
\end{equation}
\end{enumerate}
\end{proposition}

\begin{lemma}\label{yinli-feijiebijin}
Let $\textbf{z}\in X$, $n\in\mathbb{N}_0$, and $r,C>0$.
Suppose that
\begin{equation}
	\sum\limits_{\alpha\in\mathbb{N}_0^{(\mathbb{N})}}
	|c_\alpha(p)\textbf{z}^{\alpha}|
	\leqslant C\|p\|_r,
	\qquad p\in\mathscr{P}_n.\label{shi-weibubijinjiading}
\end{equation}
Then
\begin{equation}
\sum\limits_{\beta\in\Lambda_n}
|c_\beta(g)(Q_n\textbf{z})^\beta|
\leqslant C\|g\|_\infty,
\qquad g\in H^\infty(rB_{Y_n}).
\label{shi-weibuquanchunguji}
\end{equation}
\end{lemma}

\begin{proof}
For $m>n$ and $d\in\mathbb{N}_0$, put
$$
p_{m,d}(\textbf{x})
\triangleq
\sum\limits_{\substack{\beta\in\Lambda_n\\
\mathrm{supp}\beta\subset\{n+1,\cdots,m\}\\
|\beta|\leqslant d}}
\left(1-\frac{|\beta|}{d+1}\right)
c_\beta(g)\textbf{x}^\beta\in\mathscr{P}_n.
$$
Define
$$
K_d(\theta)
\triangleq
\frac{1}{d+1}\left|\sum\limits_{k=0}^{d}e^{\mathrm{i}k\theta}\right|^2,
\qquad \theta\in[0,2\pi].
$$
We have
\begin{equation}
    K_d(\theta)
    =
    \sum\limits_{\ell=-d}^{d}
    \left(1-\frac{|\ell|}{d+1}\right)e^{\mathrm{i}\ell\theta},
    \qquad
    K_d\geqslant0,
    \qquad
    \frac{1}{2\pi}\int\limits_0^{2\pi}K_d(\theta)\,\mathrm{d}\theta=1,\label{shi-feijiehe}
\end{equation}
and for each $k\in\mathbb{N}_0$, 
\begin{equation}\label{shi-feijiexishu}
    \begin{aligned}
        \frac{1}{2\pi}\int\limits_0^{2\pi}
        K_d(\theta)e^{\mathrm{i}k\theta}\,\mathrm{d}\theta
        &=\sum\limits_{\ell=-d}^{d}\left(1-\frac{|\ell|}{d+1}\right)
        \frac{1}{2\pi}\int\limits_0^{2\pi}
        e^{\mathrm{i}(\ell+k)\theta}\,\mathrm{d}\theta=
        \begin{cases}
            1-\frac{k}{d+1},&0\leqslant k\leqslant d,\\[2pt]
            0,&k>d.
        \end{cases}
    \end{aligned}
\end{equation}

Fix $g\in H^\infty(rB_{Y_n})$. For $\textbf{x}\in rB_X$, by \eqref{shi-feijiexishu}, we have
\begin{equation}\label{shi-feijiejifen}
	\begin{aligned}
		p_{m,d}(\textbf{x})
		&=
		\sum\limits_{\substack{\beta\in\Lambda_n\\
				\mathrm{supp}\beta\subset\{n+1,\cdots,m\}\\
				|\beta|\leqslant d}}
		\left(1-\frac{|\beta|}{d+1}\right)
		c_\beta(g)\textbf{x}^\beta\\
		&=
		\frac{1}{2\pi}
		\sum\limits_{\substack{\beta\in\Lambda_n\\
				\mathrm{supp}\beta\subset\{n+1,\cdots,m\}}}
		c_\beta(g)\textbf{x}^\beta
		\int\limits_0^{2\pi}
		K_d(\theta)e^{\mathrm{i}|\beta|\theta}\,\mathrm{d}\theta\\
		&=
		\frac{1}{2\pi}\int\limits_0^{2\pi}
		K_d(\theta)
		\sum\limits_{\substack{\beta\in\Lambda_n\\
				\mathrm{supp}\beta\subset\{n+1,\cdots,m\}}}
		c_\beta(g)
		\bigl(e^{\mathrm{i}\theta}(P_m-P_n)\textbf{x}\bigr)^\beta
		\,\mathrm{d}\theta\\
		&=
		\frac{1}{2\pi}\int\limits_0^{2\pi}
		K_d(\theta)\,
		g\bigl(e^{\mathrm{i}\theta}(P_m-P_n)\textbf{x}\bigr)
		\,\mathrm{d}\theta.
	\end{aligned}
\end{equation}
By \eqref{shi-feijiehe} and \eqref{shi-feijiejifen},
for every $\textbf{x}\in rB_X$,
$$
\begin{aligned}
|p_{m,d}(\textbf{x})|
&\leqslant\frac{1}{2\pi}\int\limits_0^{2\pi}
 K_d(\theta)
 \left|g\bigl(e^{\mathrm{i}\theta}(P_m-P_n)\textbf{x}\bigr)\right|
 \,\mathrm{d}\theta\leqslant\frac{\|g\|_\infty}{2\pi}
\int\limits_0^{2\pi}K_d(\theta)\,\mathrm{d}\theta
=\|g\|_\infty,
\end{aligned}
$$
which gives $\|p_{m,d}\|_r\leqslant\|g\|_\infty$. Thus
\eqref{shi-weibubijinjiading} gives
$$
\sum\limits_{\substack{\beta\in\Lambda_n\\
\mathrm{supp}\beta\subset\{n+1,\cdots,m\}\\
|\beta|\leqslant d}}
\left(1-\frac{|\beta|}{d+1}\right)
|c_\beta(g)(Q_n\textbf{z})^\beta|\leqslant C\| p_{m,d}\|_{r}
\leqslant C\|g\|_\infty.
$$
Letting $m,d\to\infty$ gives 
\eqref{shi-weibuquanchunguji}.
\end{proof}

\begin{proof}[Proof of \namedref{Proposition}{mingti-jubuhua}]

Restriction of an entire function to $O$ does not change its
coefficients, so \itemref{xiang-yuneishoulian} implies \itemref{xiang-zhenghanshushoulian}. At $\textbf{z}=\textbf{0}$ all
three conditions hold. Fix a nonzero $\textbf{z}\in O$.
By \namedref{Lemma}{yinli-weibuguji},
\itemref{xiang-zhenghanshushoulian} implies \itemref{xiang-weibupanju}.
Now suppose that \itemref{xiang-weibupanju} holds.
The finite-coordinate decomposition used below is inspired by the proofs of
\cite[Proposition~3.4, pp.~21--22]{DMP09} and
\cite[Proposition~5.4, pp.~961--962]{GMMS21}.

\proofstep{1}

Fix $f\in H(O)$, and define
$$
K\triangleq\{\textbf{x}\in X:|\textbf{x}|\leqslant|\textbf{z}|\}.
$$
The set $K$ is closed by continuity of the coordinate functionals.
For each $m\in\mathbb{N}$, we have
$$
P_mK=
\left\{\sum\limits_{j=1}^{m}w_je_j:
|w_j|\leqslant|z_j|\ (1\leqslant j\leqslant m)\right\}.
$$
Thus $P_mK$ is compact. Let $\left\{ \textbf{x}_{n} \right\}_{n\in \mathbb{N}} \subset K$ be arbitrary. By a diagonal argument, there exist a subsequence $\left\{ \textbf{x}_{n_{k}} \right\}_{k\in \mathbb{N}}$ and a sequence $\left\{ \textbf{x}^{\left( m \right)} \right\}_{m=1}^{\infty} \subset K$ such that
$$
\lim\limits_{k\to \infty} P_{m}\textbf{x}_{n_{k}} =\textbf{x}^{\left( m \right)} ,
\qquad m=1,2,\cdots.
$$
For $m_1,m_2\in \mathbb{N}$, we have
$$
\begin{aligned}
	\|\textbf{x}^{\left( m_{1} \right)} -\textbf{x}^{\left( m_{2} \right)}\|_{X}
	&=\lim\limits_{k\to \infty} \|P_{m_{1}}\textbf{x}_{n_{k}} -P_{m_{2}}\textbf{x}_{n_{k}}\|_{X}\\
	&=\lim\limits_{k\to \infty} \|Q_{m_{1}}\textbf{x}_{n_{k}} -Q_{m_{2}}\textbf{x}_{n_{k}}\|_{X}\\
	&\leqslant \|Q_{m_{1}}\textbf{z}\|_X
	+\|Q_{m_{2}}\textbf{z}\|_X.
\end{aligned}
$$
Hence, by \eqref{shi-weibushoulian} in \namedref{Lemma}{yinli-xuliekongjian}, $\left\{ \textbf{x}^{\left( m \right)} \right\}_{m\in\mathbb{N}}$ is a Cauchy sequence in $X$. Since $X$ is complete, there exists $\textbf{x}\in K$ such that $\textbf{x}^{\left( m \right)}\to \textbf{x}$. For $k,m\in\mathbb{N}$, we have
$$
\begin{aligned}
	\|\textbf{x}_{n_{k}} -\textbf{x}\|_{X}
	&\leqslant \|\textbf{x}^{\left( m \right)} -\textbf{x}\|_{X}
	+\|\textbf{x}^{\left( m \right)} -P_{m}\textbf{x}_{n_{k}}\|_{X}
	+\|Q_{m}\textbf{x}_{n_{k}}\|_{X}\\
	&\leqslant \|\textbf{x}^{\left( m \right)} -\textbf{x}\|_{X}
	+\|\textbf{x}^{\left( m \right)} -P_{m}\textbf{x}_{n_{k}}\|_{X}
	+\|Q_{m}\textbf{z}\|_{X}.
\end{aligned}
$$
Therefore,
$$
0\leqslant \limsup\limits_{k\to \infty}
\|\textbf{x}_{n_{k}} -\textbf{x}\|_{X}
\leqslant
\|\textbf{x}^{\left( m \right)} -\textbf{x}\|_{X}
+\|Q_{m}\textbf{z}\|_{X}.
$$
Letting $m\to\infty$ and applying \eqref{shi-weibushoulian} in \namedref{Lemma}{yinli-xuliekongjian}, we obtain
$\textbf{x}_{n_{k}}\to \textbf{x}$. Therefore, $K$ is compact.

Choose $t>1$ with $t\textbf{z}\in O$, using openness of $O$.
The Reinhardt property gives $tK\subset O$.
For each $\textbf{a}\in tK$, openness of $O$ and continuity of $f$
give $\delta_{\textbf{a}},M_{\textbf{a}}>0$ such that
$$
\textbf{a}+3\delta_{\textbf{a}}B_X\subset O,
\qquad
|f|\leqslant M_{\textbf{a}}
\quad\text{on }\textbf{a}+3\delta_{\textbf{a}}B_X.
$$
By compactness, choose $\textbf{a}_1,\cdots,\textbf{a}_s\in tK$
whose balls $\textbf{a}_i+\delta_{\textbf{a}_i}B_X$ cover $tK$.
Set
$$
r\triangleq\frac{1}{2}\min\limits_{1\leqslant i\leqslant s}
\delta_{\textbf{a}_i},
\qquad
M\triangleq\max\limits_{1\leqslant i\leqslant s}M_{\textbf{a}_i}.
$$
If $\textbf{w}=\textbf{y}+\textbf{h}$ with $\textbf{y}\in tK$ and
$\|\textbf{h}\|_X<2r$, choose $i$ with
$\|\textbf{y}-\textbf{a}_i\|_X<\delta_{\textbf{a}_i}$. Then
$$
\|\textbf{w}-\textbf{a}_i\|_X
<\delta_{\textbf{a}_i}+2r
\leqslant2\delta_{\textbf{a}_i}<3\delta_{\textbf{a}_i}.
$$
Consequently, for each $\textbf{w}\in tK+2rB_X$, 
\begin{equation}
tK+2rB_X\subset O,
\qquad |f(\textbf{w})|\leqslant M.
\label{shi-jinlinyuyoujie}
\end{equation}

\proofstep{2}

Choose $N,C$ as in \eqref{shi-weibupanju} for the value of
$r$ in \eqref{shi-jinlinyuyoujie}. Since
$\mathscr{P}_n\subset\mathscr{P}_N$ for $n\geqslant N$, the same
estimate holds on every such $\mathscr{P}_n$. 
Choose $n\geqslant N$ so that
$$
I\triangleq\{i\in\{1,\cdots,n\}:z_i\neq0\}\neq\emptyset.
$$
Put
$$
\Gamma_I\triangleq\{\gamma\in\mathbb{N}_0^{(\mathbb{N})}:
\mathrm{supp}\gamma\subset I\}.
$$

For $\gamma=\left( \gamma_{i} \right)_{i\in \mathbb{N}}\in\Gamma_I$ and $\textbf{y}\in rB_{Y_n}$, define
\begin{equation}
g_\gamma(\textbf{y})\triangleq
\frac{1}{(2\pi)^{|I|}}\int\limits_{[0,2\pi]^I}
f(\sum\limits_{i\in I}tz_i e^{\mathrm{i}\theta_i}e_i+\textbf{y})
e^{-\mathrm{i}\sum\limits_{i\in I}\gamma_i\theta_i}\,\prod\limits_{i\in I}\mathrm{d}\theta_i.
\label{shi-toubuxishuhanshu}
\end{equation}
\eqref{shi-jinlinyuyoujie} gives
\begin{equation}
g_\gamma\in H^\infty(rB_{Y_n}),
\qquad\|g_\gamma\|_\infty\leqslant M.
\label{shi-toubuxishuhanshujie}
\end{equation}
For $\gamma\in\Gamma_I$ and $\beta\in\Lambda_n$, we claim that
\begin{equation}
	c_\beta(g_\gamma)
	=
	t^{|\gamma|}\textbf{z}^{\gamma}c_{\gamma+\beta}(f).
	\label{shi-xishutiqu}
\end{equation}
Put
$$
J\triangleq\mathrm{supp}\beta,
\Gamma_J\triangleq
\{\nu\in\mathbb{N}_0^{(\mathbb{N})}:
\mathrm{supp}\nu\subset J\}.
$$
Then $I\cap J=\emptyset$.
Choose $\left\{ \delta_{j} \right\}_{j\in J} \subset \left( 0,+\infty \right)$ such that
$\sum\limits_{j\in J}\delta_j\|e_j\|_X<r$.
We have
$$
\left\{
\sum\limits_{i\in I}u_i e_i+\sum\limits_{j\in J}v_j e_j:
|u_i|\leqslant t|z_i|, i\in I, \
|v_j|\leqslant\delta_j,\ j\in J
\right\}
\subset tK+rB_X\subset O.
$$
For $\textbf{v}=\sum\limits_{j\in J}v_j e_j$ with
$|v_j|\leqslant\delta_j$, the  Taylor expansion gives
$$
f(\sum\limits_{i\in I}tz_i e^{\mathrm{i}\theta_i}e_i+\textbf{v})
=
\sum\limits_{\eta\in\Gamma_I}
\sum\limits_{\nu\in\Gamma_J}
c_{\eta+\nu}(f)
t^{|\eta|}\textbf{z}^{\eta}\textbf{v}^{\nu}
e^{\mathrm{i}\sum\limits_{i\in I}\eta_i\theta_i}.
$$
We have

$$
\begin{aligned}
	g_\gamma(\textbf{v})
	&=
	\frac{1}{(2\pi)^{|I|}}
	\int\limits_{[0,2\pi]^I}
	f(\sum\limits_{i\in I}tz_i e^{\mathrm{i}\theta_i}e_i+\textbf{v})
	e^{-\mathrm{i}\sum\limits_{i\in I}\gamma_i\theta_i}\,\prod\limits_{i\in I}\mathrm{d}\theta_i\\
	&=
	\sum\limits_{\eta\in\Gamma_I}
	\sum\limits_{\nu\in\Gamma_J}
	c_{\eta+\nu}(f)t^{|\eta|}\textbf{z}^{\eta}\textbf{v}^{\nu}
	\frac{1}{(2\pi)^{|I|}}
	\int\limits_{[0,2\pi]^I}
	e^{\mathrm{i}\sum\limits_{i\in I}(\eta_i-\gamma_i)\theta_i}\,\prod\limits_{i\in I}\mathrm{d}\theta_i
\\&=
	t^{|\gamma|}\textbf{z}^{\gamma}
	\sum\limits_{\nu\in\Gamma_J}
	c_{\gamma+\nu}(f)\textbf{v}^{\nu}.
\end{aligned}
$$

Comparing the coefficients of $\textbf{v}^{\beta}$ proves
\eqref{shi-xishutiqu}.

Applying the \eqref{shi-weibuquanchunguji} of \namedref{Lemma}{yinli-feijiebijin} to
$g_\gamma$ and using
\eqref{shi-toubuxishuhanshujie} and
\eqref{shi-xishutiqu}, we obtain
\begin{equation}\label{shi-hunhexishujie}
	\begin{aligned}
		t^{|\gamma|}
		\sum\limits_{\beta\in\Lambda_n}
		|c_{\gamma+\beta}(f)\textbf{z}^{\gamma+\beta}|
		&=
		\sum\limits_{\beta\in\Lambda_n}
		|c_\beta(g_\gamma)(Q_n\textbf{z})^\beta|\leqslant CM.
	\end{aligned}
\end{equation}

If $\alpha_i>0$ for some
$i\in\{1,\cdots,n\}\setminus I$, then $\textbf{z}^{\alpha}=0$.
Every other multi-index $\alpha\in\mathbb{N}_0^{(\mathbb{N})}$
has a unique decomposition $\alpha=\gamma+\beta$ with
$\gamma\in\Gamma_I$ and $\beta\in\Lambda_n$.
Consequently, by \eqref{shi-hunhexishujie}, we have
\begin{equation}
\begin{aligned}
	&\sum\limits_{\alpha \in \mathbb{N}_{0}^{(\mathbb{N} )}} |c_{\alpha}(f)\textbf{z}^{\alpha} |=\sum\limits_{\gamma \in \Gamma_{I}} \sum\limits_{\beta \in \Lambda_{n}} |c_{\gamma +\beta}(f)\textbf{z}^{\gamma +\beta} |\\ &\quad \leqslant CM\sum\limits_{\gamma \in \Gamma_{I}} t^{-\left| \gamma \right|}=CM\prod\limits_{i\in I} \left( \sum\limits_{k=0}^{\infty} t^{-k} \right) =\frac{CM}{\left( 1-\frac{1}{t} \right)^{\left| I \right|}} <\infty ,\end{aligned}
\label{shi-jubuqiuhe}
\end{equation}which gives $\textbf{z}\in\mathrm{dom}H(O)$.
Thus \itemref{xiang-weibupanju} implies \itemref{xiang-yuneishoulian}, proving all three equivalences and
\eqref{shi-jubuhua}.
\end{proof}

\begin{corollary}\label{tuilun-quanyushoulian}
If
\begin{equation}
\mathrm{dom}H(B_X)=B_X,
\label{shi-qiushoulianjiading}
\end{equation}
then, for every Reinhardt domain $O\subset X$,
\begin{equation}
\mathrm{dom}H(O)=O.
\label{shi-quanyushoulian}
\end{equation}
\end{corollary}

\begin{proof}

By \eqref{shi-qiushoulianjiading} and
\namedref{Remark}{zhuji-bibaobaohan},
$B_X\subset\overline{\mathbb{C}^{(\mathbb{N})}}^{\,X}$,
it follows that
$\overline{\mathbb{C}^{(\mathbb{N})}}^{\,X}=X$.

Fix $\textbf{x}\in X$ and $F\in H(X)$, and choose
$R>\|\textbf{x}\|_X$. Define
$$
F_R(\textbf{y})\triangleq F(R\textbf{y}),
\qquad\textbf{y}\in B_X.
$$
Then $F_R\in H(B_X)$, and 
$
c_\alpha(F_R)=R^{|\alpha|}c_\alpha(F),
\alpha\in\mathbb{N}_0^{(\mathbb{N})}.
$
Since $\frac{\textbf{x}}{R}\in B_X$, assumption \eqref{shi-qiushoulianjiading} implies
$$
\sum\limits_{\alpha\in\mathbb{N}_0^{(\mathbb{N})}}
|c_\alpha(F)\textbf{x}^\alpha|
=
\sum\limits_{\alpha\in\mathbb{N}_0^{(\mathbb{N})}}
\left|c_\alpha(F_R)\left(\frac{\textbf{x}}{R}\right)^\alpha\right|
<\infty.
$$
Thus $\textbf{x}\in\mathrm{dom}H(X)$, and hence
$\mathrm{dom}H(X)=X$.
\namedref{Proposition}{mingti-jubuhua} now yields \eqref{shi-quanyushoulian}.

\end{proof}

For $r>0$ and $h\in H^\infty(rB_X)$, let $H_m$ be its homogeneous
Taylor polynomials at zero:
\begin{equation}
	H_0(\textbf{u})\triangleq h(\textbf{0}),
	\qquad
	H_m(\textbf{u})\triangleq
	\frac{1}{m!}D^mh(\textbf{0})[\textbf{u},\cdots,\textbf{u}]
	\quad(m\geqslant1).
	\label{shi-qicitaile}
\end{equation}
Here $D^mh(\textbf{0})$ is the continuous symmetric complex
$m$-linear derivative; see \cite[Theorem~7.19, pp.~88--89]{KM97}.
It is easy to see that
\begin{equation}
	c_\alpha(H_m)=c_\alpha(h),\qquad |\alpha|=m.
	\label{shi-qicixishuyizhi}
\end{equation}
Fix $\textbf{u}\in B_X$ and $m\in\mathbb{N}_0$. The function $\lambda\mapsto h(\lambda\textbf{u})$ is holomorphic on $|\lambda|<r$. By \eqref{shi-qicitaile}, the chain rule and the one-variable Cauchy integral formula, for every $0<\rho<r$ we have
$$
H_m(\textbf{u})
=\frac{1}{m!}\left.\frac{\mathrm{d}^m}{\mathrm{d}\lambda^m}h(\lambda\textbf{u})\right|_{\lambda=0}
=\frac{1}{2\pi \mathrm{i}}\int\limits_{|\lambda|=\rho}
\frac{h(\lambda\textbf{u})}{\lambda^{m+1}}\,\mathrm{d}\lambda.
$$
Consequently,
$$
|H_m(\textbf{u})|
\leqslant
\frac{1}{2\pi\rho^{m+1}}\int\limits_{|\lambda|=\rho}
|h(\lambda\textbf{u})|\,|\mathrm{d}\lambda|
\leqslant
\frac{2\pi\rho}{2\pi\rho^{m+1}}\|h\|_\infty
=\frac{\|h\|_\infty}{\rho^m}.
$$
Letting $\rho\uparrow r$ and taking the supremum over $\textbf{u}\in B_X$, we obtain
\begin{equation}
	\|H_m\|_{\mathcal{P}^m(X)}
	=\sup\limits_{\textbf{u}\in B_X}|H_m(\textbf{u})|
	\leqslant\frac{\|h\|_\infty}{r^m},
	\qquad m\in\mathbb{N}_0.
	\label{shi-qicikexiguji}
\end{equation}

The forward implication in the following result also follows from
\cite[Theorem~5.1\,(4) and equation~(5.1), p.~34]{DMP09}.
Both implications are proved below.
\begin{theorem}\label{dingli-shoulianpanju}
One has
\begin{equation}
\mathrm{dom}H(X)=X
\label{shi-zhenghanshuquanshoulian}
\end{equation}
if and only if there exists $C\geqslant1$ such that, for every
$m,n\in\mathbb{N}$ and every $P\in\mathcal{P}^m(X)$ satisfying
$P=P\circ P_n$,
\begin{equation}
\sup\limits_{\textbf{z}\in B_X}
\sum\limits_{\substack{\alpha\in\mathbb{N}_0^{(\mathbb{N})}\\|\alpha|=m}}
|c_\alpha(P)\textbf{z}^\alpha|
\leqslant C^m\sup\limits_{\textbf{z}\in B_X}|P(\textbf{z})|.
\label{shi-tongyizhishiguji}
\end{equation}
Here $P_n$ is given on all of $X$ by the formula in \eqref{shi-shouweitouying},
and $C$ is independent of $m$, $n$, and $P$.
\end{theorem}
\begin{remark}
  The relation between
\eqref{shi-tongyizhishiguji}
and Lempert's problem is discussed in
\cite[discussion following Theorem~2.21, p.~5870]{DGMM}.
\end{remark}
\begin{lemma}\label{yinli-gujituiguang}
Suppose that there is $C\geqslant1$ such that \eqref{shi-tongyizhishiguji} holds for every $m,n\in\mathbb{N}$ and every $P\in\mathcal{P}^m(X)$ satisfying $P=P\circ P_n$. Then, for every $m\in\mathbb{N}$, $P\in\mathcal{P}^m(X)$, and $\textbf{z}\in X$,
\begin{equation}
\sum\limits_{|\alpha|=m}|c_\alpha(P)\textbf{z}^\alpha|
\leqslant C^m\|\textbf{z}\|_X^m
\sup\limits_{\textbf{w}\in B_X}|P(\textbf{w})|.
\label{shi-yibanduoxiangshijie}
\end{equation}
Moreover, $\overline{\mathbb{C}^{(\mathbb{N})}}^{\,X}=X$.
\end{lemma}

\begin{proof}

Assume that \eqref{shi-tongyizhishiguji} holds.  If $P\in\mathcal{P}^m(X)$ depends only on the
first $n$ coordinates, apply  \eqref{shi-tongyizhishiguji} to $\frac{\textbf{z}}{s}$, where
$s>\|\textbf{z}\|_X$. Multiplying by $s^m$ and letting
$s\downarrow\|\textbf{z}\|_X$ gives
\begin{equation}
\sum\limits_{|\alpha|=m}|c_\alpha(P)\textbf{z}^\alpha|
\leqslant C^m\|\textbf{z}\|_X^m
\sup\limits_{\textbf{w}\in B_X}|P(\textbf{w})|,
\qquad\textbf{z}\in X.
\label{shi-qicisuofang}
\end{equation}

Let $P\in\mathcal{P}^m(X)$ now be arbitrary. By \eqref{shi-qicisuofang}, we obtain
$$
\begin{aligned}
	&\sum\limits_{|\alpha|=m}
	\left|c_\alpha(P\circ P_n)\textbf{z}^\alpha\right|
	=
	\sum\limits_{\substack{|\alpha|=m\\
			\mathrm{supp}\alpha\subset\{1,2,\cdots,n\}}}
	\left|c_\alpha(P)\textbf{z}^\alpha\right|\\
	& \quad\leqslant
	C^m\|\textbf{z}\|_X^m
	\sup\limits_{\textbf{w}\in B_X}
	\left|P(P_n\textbf{w})\right|
\leqslant
	C^m\|\textbf{z}\|_X^m
	\sup\limits_{\textbf{w}\in B_X}
	\left|P(\textbf{w})\right|.
\end{aligned}
$$
Letting $n\to\infty$ gives \eqref{shi-yibanduoxiangshijie}.
Taking $m=2$ in \eqref{shi-yibanduoxiangshijie} and applying
\namedref{Theorem}{dingli-ercizhangai} gives
$\overline{\mathbb{C}^{(\mathbb{N})}}^{\,X}=X$.
\end{proof}

\begin{proof}[Proof of \namedref{Theorem}{dingli-shoulianpanju}]
	
	\proofstep{1}
	
	Assume that \eqref{shi-zhenghanshuquanshoulian} holds. For $\textbf{z}\in X$ and
	$P\in\mathscr{P}_0$, put
$
	S_{\textbf{z}}(P)
	\triangleq
	\sum\limits_{\alpha\in\mathbb{N}_0^{(\mathbb{N})}}
	|c_\alpha(P)\textbf{z}^\alpha|.
$
	For $k\in\mathbb{N}$, define
	\begin{equation}
		E_k
		\triangleq
		\bigcap\limits_{m,n\in\mathbb{N}}
		\bigcap\limits_{\substack{
				P\in\mathcal{P}^m(X)\\
				P=P\circ P_n
		}}
		\left\{
		\textbf{z}\in X:
		S_{\textbf{z}}(P)
		\leqslant
		k^m\|P\|_{\mathcal{P}^m(X)}
		\right\}.
		\label{shi-beierbiji}
	\end{equation}

Recall \eqref{shi-zuobiaoguji}, the coordinate functionals are continuous. Since $P$ depends on only finitely many coordinates, it follows from \eqref{shi-youxianzhankai} that the map
$\textbf{z}\mapsto S_{\textbf{z}}(P)$
is continuous on $X$. Hence $E_k$ is closed.

	For $\textbf{z}\in X$ and
	$\textbf{w}\in\mathbb{C}^{\mathbb{N}}$, \eqref{shi-fanshudandiao} yields
	$$
	|\textbf{w}|\leqslant|\textbf{z}|
	\quad\Longrightarrow\quad
	\textbf{w}\in X,\qquad
	S_{\textbf{w}}(P)\leqslant S_{\textbf{z}}(P)
	\quad(P\in\mathscr{P}_0),
	$$which gives
	\begin{equation}
		\textbf{z}\in E_k,\qquad
		|\textbf{w}|\leqslant|\textbf{z}|
		\quad\Longrightarrow\quad
		\textbf{w}\in E_k.
		\label{shi-bijiguxing}
	\end{equation}
	
	Fix $\textbf{z}\in X$. By the \eqref{shi-dianchuzhishijie} of \namedref{Lemma}{yinli-bituguji}, there exists
	$A_{\textbf{z}}\geqslant1$ such that for every $m\in\mathbb{N},
	P\in\mathcal{P}^m(X)$,
	$$
	S_{\textbf{z}}(P)
	\leqslant
	A_{\textbf{z}}^m\|P\|_{\mathcal{P}^m(X)}.
	$$
Notice that $\textbf{z}\in E_{\lceil A_{\textbf{z}}\rceil}$, hence
$X=\bigcup\limits_{k=1}^{\infty}E_k$.
	Since $X$ is Banach,  Baire's theorem
	(see \cite[52.12, p.~577]{KM97}) yields
	$K\in\mathbb{N}$, $\textbf{a}=\left( a_{j} \right)_{j\in \mathbb{N}}\in X$, and $\rho>0$ such that
	\begin{equation}
		\textbf{a}+\rho B_X\subseteq E_K.
		\label{shi-beierneidian}
	\end{equation}
	
	Fix $\textbf{h}\in\rho B_X$ and define
	$\textbf{u}=(u_j)_{j\in\mathbb{N}}$ by
	$$
	u_j
	\triangleq
	\begin{cases}
		\frac{a_j}{|a_j|}|h_j|,&a_j\neq0,\\[4pt]
		|h_j|,&a_j=0.
	\end{cases}
	$$
	By \eqref{shi-fanshudandiao},
	$$
	|\textbf{u}|=|\textbf{h}|,\qquad
	\textbf{u}\in X,\qquad
	\|\textbf{u}\|_X=\|\textbf{h}\|_X<\rho,\qquad	|\textbf{a}+\textbf{u}|
	=
	|\textbf{a}|+|\textbf{h}|
	\geqslant|\textbf{h}|.
	$$
	Combining \eqref{shi-beierneidian} and \eqref{shi-bijiguxing},
	we obtain 
	$$
	\textbf{a}+\textbf{u}
	\in\textbf{a}+\rho B_X
	\subseteq E_K\Rightarrow \textbf{h} \in E_{K}\Rightarrow\rho B_X\subseteq E_K.
	$$

	Set $C\triangleq\max\left\{1,\frac{K}{\rho}\right\}$.
	For $m,n\in\mathbb{N}$ and $P\in\mathcal{P}^m(X)$
	satisfying $P=P\circ P_n$, one has
	$$
\sup\limits_{\textbf{z} \in B_{X}} S_{\textbf{z}}(P)=\frac{1}{\rho^{m}} \sup\limits_{\textbf{z} \in B_{X}} S_{\rho \textbf{z}}(P)\leqslant \left( \frac{K}{\rho} \right)^{m} \| P\|_{\mathcal{P}^{m} (X)} \leqslant C^{m}\| P\|_{\mathcal{P}^{m} (X)} ,
	$$
	which proves \eqref{shi-tongyizhishiguji}.
	
	\proofstep{2}
	
	Assume that \eqref{shi-tongyizhishiguji} holds, with $C\geqslant1$ fixed.
	By \namedref{Lemma}{yinli-gujituiguang}, \eqref{shi-yibanduoxiangshijie} holds
	and $\overline{\mathbb{C}^{(\mathbb{N})}}^{\,X}=X$. 	Fix $\textbf{z}\in X$ and $r>0$. By  equation \eqref{shi-weibushoulian} in \namedref{Lemma}{yinli-xuliekongjian},
	choose $N\in\mathbb{N}$ such that
$
	C\|Q_N\textbf{z}\|_X<\frac{r}{2}.
$
	Let $p\in\mathscr{P}_N$. By the definition of $\mathscr{P}_N$,
	there exist $L\in\mathbb{N}$ and $d\in\mathbb{N}_0$, with $L>N$,
	such that
	$$
	p(\textbf{w})
	=
	\sum\limits_{m=0}^{d}
	\sum\limits_{\substack{
			|\alpha|=m\\
			\mathrm{supp}\alpha\subseteq\{N+1,\cdots,L\}
	}}
	c_\alpha(p)\textbf{w}^{\alpha},
	\qquad
	\textbf{w}\in X.
	$$
	For $m=0,\cdots,d$, define
	$$
	H_m(\textbf{w})
	\triangleq
	\sum\limits_{\substack{
			|\alpha|=m\\
			\mathrm{supp}\alpha\subseteq\{N+1,\cdots,L\}
	}}
	c_\alpha(p)\textbf{w}^{\alpha},
	\qquad
	\textbf{w}\in X.
	$$
By \eqref{shi-qicikexiguji},
	$$
	\|H_m\|_{\mathcal{P}^m(X)}
	\leqslant
	\frac{\sup\limits_{\textbf{w}\in rB_X}|p(\textbf{w})|}{r^m}
	=
	\frac{\|p\|_r}{r^m},
	\qquad
	0\leqslant m\leqslant d.
	$$
	Combining \eqref{shi-yibanduoxiangshijie} with
\eqref{shi-qicixishuyizhi}, we obtain
	$$
	\begin{aligned}
		S_{\textbf{z}}(p)
		&=
		\sum\limits_{m=0}^{d}
		\sum\limits_{|\alpha|=m}
		|c_\alpha(H_m)\textbf{z}^{\alpha}|=
		\sum\limits_{m=0}^{d}
		\sum\limits_{|\alpha|=m}
		|c_\alpha(H_m)(Q_N\textbf{z})^{\alpha}|\\
		&=
		|p(\textbf{0})|
		+
		\sum\limits_{m=1}^{d}
		\sum\limits_{|\alpha|=m}
		|c_\alpha(H_m)(Q_N\textbf{z})^{\alpha}|\\
		&\leqslant
		\|p\|_r
		+
		\sum\limits_{m=1}^{d}
		C^m\|Q_N\textbf{z}\|_X^m
		\|H_m\|_{\mathcal{P}^m(X)}\\
		&\leqslant
		\|p\|_r
		\sum\limits_{m=0}^{d}
		\left(
		\frac{C\|Q_N\textbf{z}\|_X}{r}
		\right)^m\leqslant
		\|p\|_r
		\sum\limits_{m=0}^{\infty}\frac{1}{2^m}
		=
		2\|p\|_r.
	\end{aligned}
	$$
Thus, for every $r>0$, there exists $N\in\mathbb{N}$ such that
\eqref{shi-weibupanju} holds with constant $C=2$ for all $p\in\mathscr{P}_N$.
Applying \namedref{Proposition}{mingti-jubuhua}, we obtain
$\textbf{z}\in\mathrm{dom}H(X)$ which gives \eqref{shi-zhenghanshuquanshoulian}.
\end{proof}

\section{A negative solution to Lempert's problem}
\label{jie-fanligouzao}

\begin{definition}\label{dingyi-fanlikongjian}
    Define
    \begin{equation}
        a_0\triangleq 1,
        \qquad
        a_r\triangleq 2^{a_{r-1}+1},
        \qquad r=1,2,\cdots.
        \label{shi-cengjicanshu}
    \end{equation}
    For $r=1,2,\cdots$, set
    \begin{equation}
        \kappa(r)
        \triangleq
        \max\left\{
        j\in\mathbb{N}_0:a_j\leqslant r
        \right\}.
        \label{shi-cengjisuoyin}
    \end{equation}

Given a nondecreasing function
$
    \Psi:\mathbb{N}\longrightarrow[1,+\infty)
  $
    satisfying
$
    \lim\limits_{n\to\infty}\Psi(n)=+\infty,
$
    define inductively
    \begin{equation}
        t_0\triangleq 1, t_r
        \triangleq
        \min\left\{
        \left(1+\frac{1}{r}\right)t_{r-1},
        \,2t_{\kappa(r)},
        \,\Psi(a_{r-1})
        \right\},
        \qquad r=1,2,\cdots.
        \label{shi-huanmancanshu}
    \end{equation}

    For each $r\in\mathbb{N}_0$, define
    \begin{equation}
        \mathcal{G}_r
        \triangleq
        \left\{
        \frac{t_r}{a_r}
        \sum\limits_{j\in F}\epsilon_j e_j:
        F\subset\mathbb{N},
        \ |F|=a_r,
        \ \epsilon_j\in\{-1,1\}
        \ \text{for every }j\in F
        \right\},
        \qquad
        \mathcal{G}
        \triangleq
        \bigcup\limits_{r\in\mathbb{N}_0}\mathcal{G}_r.
        \label{shi-yuanziji}
    \end{equation}
    For $\textbf{z}\in\mathbb{C}^{(\mathbb{N})}$, define
    $$
    \|\textbf{z}\|_Y
    \triangleq
    \inf
    \left\{
    \sum\limits_{v=1}^{s}|\lambda_v|:
    \begin{array}{l}
        s\in\mathbb{N},\;
        \lambda_v\in\mathbb{R},\;
        \textbf{u}_v\in\mathcal{G},\quad v=1,\cdots,s,\\[2mm]
        |\textbf{z}|=\displaystyle\sum\limits_{v=1}^{s}\lambda_v\textbf{u}_v
    \end{array}
    \right\}.
    $$
  
\end{definition}

\begin{lemma}\label{yinli-canshuxingzhi}
The sequence $(t_r)_{r\in \mathbb{N}_{0}}$ defined by \eqref{shi-huanmancanshu} satisfies for every $r\in\mathbb{N}$,
\begin{align}
1&\leqslant t_{r-1}\leqslant t_r\leqslant r+1,
\label{shi-huanmancanshujie}\\
\frac{t_r}{a_r}&\leqslant\frac{1}{2}\frac{t_{r-1}}{a_{r-1}}\leqslant\frac{1}{2},
\label{shi-yuanzigaodu}\\
\lim\limits_{r\to +\infty} t_{r}&=+\infty,\label{shi-jixian}
\end{align}
\end{lemma}

\begin{proof}
First, it follows immediately from $a_i>i$, $i=1,2,\cdots$, that
$$
0\leqslant \kappa(i)<i,\qquad i=1,2,\cdots,
$$
and hence $(t_r)_{r\in\mathbb{N}_0}$ is well defined. Clearly, $t_{1}\geqslant t_{0}\geqslant 1$, $\kappa$ is increasing and
$$
\lim\limits_{r\to+\infty}\kappa(r)=+\infty.
$$
We prove by induction that
$$
1\leqslant t_{r-1}\leqslant t_r.
$$

Suppose that $t_0,t_1,\cdots,t_{r-1}$ have already been constructed. Then, for $r\geqslant 2$, it follows from \eqref{shi-huanmancanshu} that
$$
t_{r-1}\leqslant 2t_{\kappa(r-1)}
\leqslant 2t_{\kappa(r)},
\qquad
t_{r-1}\leqslant \Psi(a_{r-2})
\leqslant \Psi(a_{r-1}),
$$
and hence
$$
t_r\geqslant t_{r-1}\geqslant 1.
$$
This proves that
$$
1\leqslant t_{r-1}\leqslant t_r
$$
for every $r\in\mathbb{N}$. Now, for $r\in\mathbb{N}$, \eqref{shi-huanmancanshu} gives
$$
t_r
=\prod\limits_{i=1}^{r}\frac{t_i}{t_{i-1}}
\leqslant
\prod\limits_{i=1}^{r}\left(1+\frac{1}{i}\right)
=r+1,
$$
which proves \eqref{shi-huanmancanshujie}.

By the evident inequality
$
a_i\geqslant 4a_{i-1},i=1,2,\cdots,
$
and \eqref{shi-huanmancanshu}, we obtain
$$
\frac{t_r}{a_r}
\leqslant
\frac{\left(1+\frac{1}{r}\right)t_{r-1}}{4a_{r-1}}
\leqslant
\frac{1}{2}\frac{t_{r-1}}{a_{r-1}}
\leqslant
\frac{1}{2},
$$
which proves \eqref{shi-yuanzigaodu}.

Suppose that
$$
\lim\limits_{r\to\infty}t_r=t<\infty.
$$
Then, by \eqref{shi-huanmancanshu} and
$$
\lim\limits_{r\to\infty}2t_{\kappa(r)}=2t,
\qquad
\lim\limits_{r\to\infty}\Psi(a_{r-1})=+\infty,
$$
we have, for all sufficiently large $r$,
$$
t_r=\left(1+\frac{1}{r}\right)t_{r-1}.
$$
This clearly implies
$$
\lim\limits_{r\to\infty}t_r=+\infty,
$$
a contradiction. Thus we have proved \eqref{shi-jixian}.
\end{proof}

\begin{lemma}\label{yinli-fanshugouzao}
The expression $\|\cdot\|_Y$ in \namedref{Definition}{dingyi-fanlikongjian} defines a norm on $\mathbb{C}^{(\mathbb{N})}$ satisfying
\begin{equation}
\|\textbf{z}\|_{\ell^\infty}\leqslant\|\textbf{z}\|_Y\leqslant\|\textbf{z}\|_{\ell^1},
\qquad \textbf{z}\in\mathbb{C}^{(\mathbb{N})}.
\label{shi-fanlifanshubijiao}
\end{equation}
Moreover, whenever $\textbf{x},\textbf{y}\in\mathbb{C}^{(\mathbb{N})}$ and $|\textbf{x}|\leqslant|\textbf{y}|$,
\begin{equation}
\|\textbf{x}\|_Y\leqslant\|\textbf{y}\|_Y.
\label{shi-fanlifanshudandiao}
\end{equation}
\end{lemma}

\begin{proof}

$e_j\in\mathcal{G}_0$ for every $j\in\mathbb{N}$. By \eqref{shi-yuanzigaodu}, every
$\textbf{u}\in\mathcal{G}$ satisfies $\|\textbf{u}\|_{\ell^\infty}\leqslant 1$.
Thus, for each  representation
$|\textbf{z}|=\sum\limits_{v=1}^s\lambda_v\textbf{u}_v,\textbf{z}=\left( z_{j} \right)_{j\in \mathbb{N}} \in \mathbb{C}^{(\mathbb{N})}$,
$$
\| \textbf{z} \|_{\ell^{\infty}} \leqslant \sum\limits_{v=1}^{s} |\lambda_{v} |\| \textbf{u}_{v}\|_{\ell^{\infty}} \leqslant \sum\limits_{v=1}^{s} |\lambda_{v} |.
$$
Taking the infimum, we obtain $\|\textbf{z}\|_{\ell^{\infty}} \leqslant \|\textbf{z}\|_{Y}$. By $|\textbf{z}|=\sum\limits_{j\in\mathrm{supp}\textbf{z}}|z_j|e_j$, we have$$\|\textbf{z}\|_{Y} \leqslant \sum\limits_{j\in \mathrm{supp} \textbf{z}} \left| z_{j} \right| =\|\textbf{z}\|_{\ell^{1}},$$which gives \eqref{shi-fanlifanshubijiao}.

Let $r\in\mathbb{N}_0$ and
$
\textbf{u}=\frac{t_r}{a_r}\sum\limits_{j\in F}\epsilon_j e_j\in\mathcal{G}_r.
$
For any  $\theta_j\in[-1,1]$, $j\in F$, one has
\begin{equation}
    \frac{t_r}{a_r}\sum\limits_{j\in F}\theta_j\epsilon_j e_j
    =
    \sum\limits_{\delta=\left( \delta_{j} \right)_{j\in F}\in\{-1,1\}^{F}}
    \left(\prod\limits_{j\in F}\frac{1+\delta_j\theta_j}{2}\right)
    \left(\frac{t_r}{a_r}\sum\limits_{j\in F}\delta_j\epsilon_j e_j\right).\label{shi-fuhaopingjun}
\end{equation}
Indeed, we have $$
\begin{aligned}
    \sum\limits_{\delta\in\{-1,1\}^{F}}
    \prod\limits_{j\in F}\frac{1+\delta_j\theta_j}{2}
    &=
    \prod\limits_{j\in F}
    \left(
    \sum\limits_{\delta_j\in\{-1,1\}}
    \frac{1+\delta_j\theta_j}{2}
    \right) =
    \prod\limits_{j\in F}
    \left(
    \frac{1+\theta_j}{2}
    +
    \frac{1-\theta_j}{2}
    \right)
    =1.
\end{aligned}
$$
Furthermore, for every fixed $k\in F$,
$$
\begin{aligned}
    \sum\limits_{\delta\in\{-1,1\}^{F}}
    \delta_k
    \prod\limits_{j\in F}\frac{1+\delta_j\theta_j}{2}
    &=
    \left(
    \sum\limits_{\delta_k\in\{-1,1\}}
    \delta_k\frac{1+\delta_k\theta_k}{2}
    \right)
    \prod\limits_{\substack{j\in F\\ j\neq k}}
    \left(
    \sum\limits_{\delta_j\in\{-1,1\}}
    \frac{1+\delta_j\theta_j}{2}
    \right) \\
    &=
    \left(
    \frac{1+\theta_k}{2}
    -
    \frac{1-\theta_k}{2}
    \right)
    \prod\limits_{\substack{j\in F\\ j\neq k}}1 =
    \theta_k.
\end{aligned}
$$
Thus $$\begin{aligned}
    &\sum\limits_{\delta \in \{ -1,1\}^{F}} \left( \prod\limits_{j\in F} \frac{1+\delta_{j} \theta_{j}}{2} \right) \left( \frac{t_{r}}{a_{r}} \sum\limits_{j\in F} \delta_{j} \epsilon_{j} e_{j} \right)\\ &\  =\frac{t_{r}}{a_{r}} \sum\limits_{k\in F} \epsilon_{k} e_{k}\sum\limits_{\delta \in \{ -1,1\}^{F}} \left( \delta_{k} \prod\limits_{j\in F} \frac{1+\delta_{j} \theta_{j}}{2} \right)\\ &\  =\frac{t_{r}}{a_{r}} \sum\limits_{k\in F} \theta_{k} \epsilon_{k} e_{k},\end{aligned}$$
which proves \eqref{shi-fuhaopingjun}.

Now let $\textbf{x},\textbf{y}\in\mathbb{C}^{(\mathbb{N})}$ satisfy
$
|\textbf{x}|\leqslant|\textbf{y}|.
$ Put

$$
\theta_j\triangleq
\begin{cases}
    \frac{|x_j|}{|y_j|},&y_j\neq 0,\\[1mm]
    0,&y_j=0,
\end{cases}
\qquad j\in\mathbb{N}.
$$

Then $0\leqslant\theta_j\leqslant1$ and
$
|x_j|=\theta_j|y_j|, j\in\mathbb{N}.
$ Fix a representation

$$
|\textbf{y}|=\sum\limits_{v=1}^{s}\lambda_v\textbf{u}_v,
\qquad
\lambda_v\in\mathbb{R},\quad \textbf{u}_v\in\mathcal{G}.
$$

For each $v=1,2,\cdots,s$, write
$
\textbf{u}_v=
\frac{t_{r_v}}{a_{r_v}}
\sum\limits_{j\in F_v}\epsilon_{v,j}e_j.
$ By \eqref{shi-fuhaopingjun}, there exist nonnegative numbers
$\mu_{v,\delta}$, $\delta\in\{-1,1\}^{F_v}$, satisfying $
\sum\limits_{\delta\in\{-1,1\}^{F_v}}\mu_{v,\delta}=1,
$ such that
$$
\frac{t_{r_v}}{a_{r_v}}
\sum\limits_{j\in F_v}\theta_j\epsilon_{v,j}e_j
=
\sum\limits_{\delta\in\{-1,1\}^{F_v}}
\mu_{v,\delta}
\left(
\frac{t_{r_v}}{a_{r_v}}
\sum\limits_{j\in F_v}\delta_j\epsilon_{v,j}e_j
\right),
$$

where every sequence in parentheses belongs to
$\mathcal{G}_{r_v}\subset\mathcal{G}$. Consequently,
$$
\begin{aligned}
    |\textbf{x}|
    &=
    \sum\limits_{v=1}^{s}\lambda_v
    \frac{t_{r_v}}{a_{r_v}}
    \sum\limits_{j\in F_v}\theta_j\epsilon_{v,j}e_j\\
    &=
    \sum\limits_{v=1}^{s}
    \sum\limits_{\delta\in\{-1,1\}^{F_v}}
    \lambda_v\mu_{v,\delta}
    \left(
    \frac{t_{r_v}}{a_{r_v}}
    \sum\limits_{j\in F_v}\delta_j\epsilon_{v,j}e_j
    \right).
\end{aligned}
$$

Thus the last expression is a representation of
$|\textbf{x}|$, and the sum of the absolute values of its coefficients satisfies

$$
\begin{aligned}
    \sum\limits_{v=1}^{s}
    \sum\limits_{\delta\in\{-1,1\}^{F_v}}
    |\lambda_v|\mu_{v,\delta}
    &=
    \sum\limits_{v=1}^{s}
    |\lambda_v|
    \sum\limits_{\delta\in\{-1,1\}^{F_v}}
    \mu_{v,\delta}=
    \sum\limits_{v=1}^{s}|\lambda_v|.
\end{aligned}
$$

Hence $
\|\textbf{x}\|_Y
\leqslant
\sum\limits_{v=1}^{s}|\lambda_v|.
$ Since the  representation of $|\textbf{y}|$ was arbitrary, taking the infimum yields \eqref{shi-fanlifanshudandiao}.

Concatenating representations of $|\textbf{z}|$ and
$|\textbf{w}|$, and then applying
\eqref{shi-fanlifanshudandiao}, gives
$$
\|\textbf{z}+\textbf{w}\|_Y
\leqslant
\left\||\textbf{z}|+|\textbf{w}|\right\|_Y
\leqslant
\|\textbf{z}\|_Y+\|\textbf{w}\|_Y.
$$
Moreover, the definition gives
$\|\lambda\textbf{z}\|_Y=|\lambda|\|\textbf{z}\|_Y$
for $\lambda\in\mathbb{C}$, and positive definiteness follows from
\eqref{shi-fanlifanshubijiao}. Hence $\|\cdot\|_Y$ is a norm on
$\mathbb{C}^{(\mathbb{N})}$.
\end{proof}
\begin{definition}\label{dingyi-fanlikongjian-Y}
	Define
	$$
	Y
	\triangleq
	\left\{
	\textbf{z}\in c_0:
	(P_n\textbf{z})_{n\in\mathbb{N}}
	\text{ is a Cauchy sequence in }
	(\mathbb{C}^{(\mathbb{N})},\|\cdot\|_{Y})
	\right\}.
	$$
	For $\textbf{z}\in Y$, put
	\begin{equation}
		\|\textbf{z}\|_Y
		\triangleq
		\lim\limits_{n\to\infty}\|P_n\textbf{z}\|_{Y}.
		\label{shi-fanlikongjianfanshu}
	\end{equation}
\end{definition}
\begin{proposition}\label{mingti-fanlixingzhi}
Let $Y$ be the space defined in \namedref{Definition}{dingyi-fanlikongjian-Y}. Then the following
assertions hold.
\begin{enumerate}
\item\label{xiang-fanlikongjian}
$Y$ is a complex Banach sequence space and
$\mathbb{C}^{(\mathbb{N})}$ is dense in $Y$.

\item\label{xiang-duichenji}
$(e_j)_{j\in\mathbb{N}}$ is a normalized $1$-symmetric Schauder basis
of $Y$. More precisely, $\|e_j\|_Y=1$ for every $j\in\mathbb{N}$, and
\begin{equation}
\left\|\sum\limits_{j=1}^{n}\omega_jz_je_{\pi(j)}\right\|_Y
=
\left\|\sum\limits_{j=1}^{n}z_je_j\right\|_Y
\label{shi-zhihuanxiangwei}
\end{equation}
for every $n\in\mathbb{N}$, every permutation $\pi$ of $\mathbb{N}$,
and all $z_j,\omega_j\in\mathbb{C}$ with $|\omega_j|=1$,
$j=1,\cdots,n$. 

\item\label{xiang-shibuyuxubu}
For every $\textbf{z}\in Y$,
\begin{equation}
|\textbf{z}|
\in Y\cap\mathbb{R}^{\mathbb{N}},
\qquad
\left\||\textbf{z}|\right\|_Y=\|\textbf{z}\|_Y.
\label{shi-qumobaofan}
\end{equation}
Moreover,
\begin{equation}
\left\{\textbf{x}\in Y\cap\mathbb{R}^{\mathbb{N}}:
\|\textbf{x}\|_Y\leqslant1\right\}
=
\overline{\mathrm{conv}\mathcal{G}}^{\,Y},
\label{shi-shiqiutubao}
\end{equation}
where $\mathrm{conv}\mathcal{G}$ denotes the set of finite
convex combinations of elements of $\mathcal{G}$.

\end{enumerate}
\end{proposition}

\begin{proof}

\proofstep{1}

First, it is easy to see that $Y$ is a complex linear space containing $\mathbb{C}^{(\mathbb{N})}$, and that $\|\cdot\|_Y$ is a seminorm on $Y$.

For every $\textbf{z}\in Y$ and $n\in\mathbb{N}$,
\eqref{shi-fanlifanshubijiao} gives
$$
\|P_n\textbf{z}\|_{\ell^\infty}
\leqslant
\|P_n\textbf{z}\|_{Y}.
$$
Since $\textbf{z}\in c_0$,
$
P_n\textbf{z}\to\textbf{z}
$
in $\ell^\infty$. Therefore,
\begin{equation}
	\|\textbf{z}\|_{\ell^\infty}
	\leqslant
	\|\textbf{z}\|_Y,
	\qquad
	\textbf{z}\in Y.
	\label{shi-fanliqianru}
\end{equation}
In particular, $\|\textbf{z}\|_Y=0$ implies
$\textbf{z}=\textbf{0}$, so $\|\cdot\|_Y$ is a norm on $Y$.

We next prove that $Y$ is complete. We first note that, for $m\geqslant n$, \eqref{shi-fanlifanshudandiao} yields
$$
\|P_n\textbf{z}\|_{Y}
\leqslant
\|P_m\textbf{z}\|_{Y}.
$$
Letting $m\to\infty$ proves \begin{equation}
\|P_n\textbf{z}\|_{Y}
\leqslant
\|\textbf{z}\|_{Y},n=1,2,\cdots.\label{shi-fanlitouyingsuosuo}
\end{equation}

Let
$(\textbf{z}^{(k)})_{k\in\mathbb{N}}$
be a Cauchy sequence in $Y$. By \eqref{shi-fanliqianru},
it is Cauchy in $c_0$. Since $c_0$ is complete, there exists
$\textbf{z}\in c_0$ such that
\begin{equation}
	\lim\limits_{k\to\infty}
	\|\textbf{z}^{(k)}-\textbf{z}\|_{\ell^\infty}=0.
	\label{shi-fanliclingjixian}
\end{equation}
Fix $\varepsilon>0$. Choose $K\in\mathbb{N}$ such that
\begin{equation}
	\|\textbf{z}^{(k)}-\textbf{z}^{(l)}\|_Y<\varepsilon,
	\qquad
	k,l\geqslant K.
	\label{shi-fanlibibeilie}
\end{equation}
By \eqref{shi-fanlitouyingsuosuo},
\begin{equation}
\|P_n(\textbf{z}^{(k)}-\textbf{z}^{(l)})\|_{Y}
<\varepsilon,
\qquad
n\in\mathbb{N},\quad k,l\geqslant K.\label{1}
\end{equation}
For fixed $n$ and $k\geqslant K$, \eqref{shi-fanlifanshubijiao} and
\eqref{shi-fanliclingjixian} give
$$
\begin{aligned}
	\|P_n(\textbf{z}^{(l)}-\textbf{z})\|_{Y}
	&\leqslant
	\|P_n(\textbf{z}^{(l)}-\textbf{z})\|_{\ell^1}\leqslant
	n\|\textbf{z}^{(l)}-\textbf{z}\|_{\ell^\infty}
	\longrightarrow0
	\qquad(l\to\infty).
\end{aligned}
$$
Consequently, letting $l\to\infty$ in \eqref{1} yields
\begin{equation}
	\|P_n(\textbf{z}^{(k)}-\textbf{z})\|_{Y}
	\leqslant\varepsilon,
	\qquad
	n\in\mathbb{N},\quad k\geqslant K.
	\label{shi-fanlijixianyoujiekongzhi}
\end{equation}

Fix $k\geqslant K$. For $n>m$, by
\eqref{shi-fanlifanshudandiao} and
\eqref{shi-fanlijixianyoujiekongzhi},
$$
\begin{aligned}
	\|(P_n-P_m)\textbf{z}\|_{Y}
	&\leqslant
	\|(P_n-P_m)(\textbf{z}-\textbf{z}^{(k)})\|_{Y}
	+
	\|(P_n-P_m)\textbf{z}^{(k)}\|_{Y}\\
	&\leqslant
	\|P_n(\textbf{z}-\textbf{z}^{(k)})\|_{Y}
	+
	\|(P_n-P_m)\textbf{z}^{(k)}\|_{Y}\\
	&\leqslant
	\varepsilon+
	\|(P_n-P_m)\textbf{z}^{(k)}\|_{Y}.
\end{aligned}
$$
Since $\textbf{z}^{(k)}\in Y$, the sequence
$(P_n\textbf{z}^{(k)})_{n\in\mathbb{N}}$ is Cauchy with respect to
$\|\cdot\|_{Y}$. Hence
$$
\limsup\limits_{m,n\to\infty}
\|(P_n-P_m)\textbf{z}\|_{Y}
\leqslant\varepsilon.
$$Hence
$(P_n\textbf{z})_{n\in\mathbb{N}}$ is Cauchy with respect to
$\|\cdot\|_{Y}$. Therefore $\textbf{z}\in Y$.

Now \eqref{shi-fanlijixianyoujiekongzhi} and the definition of
$\|\cdot\|_Y$ give
$$
\|\textbf{z}^{(k)}-\textbf{z}\|_Y
=
\lim\limits_{n\to\infty}
\|P_n(\textbf{z}^{(k)}-\textbf{z})\|_{Y}
\leqslant\varepsilon,
\qquad
k\geqslant K.
$$
Thus
$
\textbf{z}^{(k)}\to\textbf{z}
$
in $Y$, and $Y$ is complete.

We next verify the coordinatewise monotonicity. Let
$\textbf{z}\in Y$ and
$\textbf{w}\in\mathbb{C}^{\mathbb{N}}$ satisfy
$$
|\textbf{w}|\leqslant|\textbf{z}|.
$$
Since $\textbf{z}\in c_0$, one has $\textbf{w}\in c_0$. For $n>m$,  \eqref{shi-fanlifanshudandiao} gives
$$
\|(P_n-P_m)\textbf{w}\|_{Y}
\leqslant
\|(P_n-P_m)\textbf{z}\|_{Y}.
$$
Since $(P_n\textbf{z})_{n\in\mathbb{N}}$ is Cauchy with respect to
$\|\cdot\|_{Y}$, so is
$(P_n\textbf{w})_{n\in\mathbb{N}}$. Thus $\textbf{w}\in Y$.
Moreover,
$$
\|P_n\textbf{w}\|_{Y}
\leqslant
\|P_n\textbf{z}\|_{Y},
\qquad n\in\mathbb{N},
$$
and therefore
$
\|\textbf{w}\|_Y
\leqslant
\|\textbf{z}\|_Y.
$
Hence $Y$ satisfies \eqref{shi-fanshudandiao}.

Finally, let $\textbf{z}\in Y$. Since
$(P_n\textbf{z})_{n\in\mathbb{N}}$ is Cauchy with respect to
$\|\cdot\|_{Y}$,
$$
\begin{aligned}
	\|\textbf{z}-P_n\textbf{z}\|_Y
	&=
	\lim\limits_{m\to\infty}
	\|P_m(\textbf{z}-P_n\textbf{z})\|_{Y}=
	\lim\limits_{m\to\infty}
	\|P_m\textbf{z}-P_n\textbf{z}\|_{Y}
	\longrightarrow0
	\qquad(n\to\infty).
\end{aligned}
$$
Thus $\mathbb{C}^{(\mathbb{N})}$ is dense in $Y$. This proves
\itemref{xiang-fanlikongjian}.

\proofstep{2}

By \itemref{xiang-fanlikongjian} and \namedref{Lemma}{yinli-xuliekongjian}~\itemref{xiang-wutiaojianji}, the sequence
$(e_j)_{j\in\mathbb{N}}$ is a $1$-unconditional Schauder basis of $Y$.
By \eqref{shi-fanlifanshubijiao}, $\|e_j\|_Y=1$ for every
$j\in\mathbb{N}$. A coordinate permutation maps $\mathcal{G}$ onto
itself, and hence, by the definition of $\|\cdot\|_Y$, preserves the
norm. By \eqref{shi-fanlifanshudandiao},
multiplication of the coordinates by scalars of modulus $1$ also
preserves the norm, since it does not change the coordinatewise
absolute value. This proves  \itemref{xiang-duichenji}.

\proofstep{3}
It is easy to see that \eqref{shi-qumobaofan} holds and $$\|\textbf{u}\|_{Y} \leqslant 1,\qquad \textbf{u}\in \mathcal{G}.$$

The set
$Y\cap\mathbb{R}^{\mathbb{N}}$ is closed in $Y$, because convergence
in $Y$ implies coordinatewise convergence. Hence
$$
\overline{\mathrm{conv}\mathcal{G}}^{\,Y}
\subset
\left\{\textbf{x}\in Y\cap\mathbb{R}^{\mathbb{N}}:
\|\textbf{x}\|_Y\leqslant1\right\}.
$$
To prove the reverse inclusion, first take 
$\textbf{x}\in \mathbb{R}^{(\mathbb{N})}$ with $\|\textbf{x}\|_Y\leqslant1$ and fix
$\varepsilon>0$. Choose a representation
$$
|\textbf{x}|=\sum\limits_{v=1}^{s}\lambda_v\textbf{u}_v,
\qquad
\sum\limits_{v=1}^{s}|\lambda_v|<1+\varepsilon.
$$
For every $j\in\mathbb{N}$, choose $\theta_j\in\{-1,1\}$ such that
$x_j=\theta_j|x_j|$, taking $\theta_j=1$ when $x_j=0$. For
$v=1,\cdots,s$, write $\textbf{u}_v=(u_{v,j})_{j\in\mathbb{N}}$ and define
$$
\textbf{w}_v\triangleq
\sum\limits_{j\in\mathbb{N}}\theta_ju_{v,j} e_j.
$$
One has
$\textbf{w}_v\in\mathcal{G}$. Now define
$$
\widetilde{\textbf{u}}_v\triangleq
\begin{cases}
	\textbf{w}_v,&\lambda_v\geqslant0,\\
	-\textbf{w}_v,&\lambda_v<0.
\end{cases}
$$
We have
$\widetilde{\textbf{u}}_v\in\mathcal{G}$. Moreover,
$$
\begin{aligned}
	\textbf{x}
	&=\sum\limits_{j\in\mathbb{N}}\theta_j|x_j|e_j=\sum\limits_{j\in\mathbb{N}}
	\theta_j
	\left(
	\sum\limits_{v=1}^{s}\lambda_vu_{v,j}
	\right)e_j\\
	&=\sum\limits_{v=1}^{s}\lambda_v
	\sum\limits_{j\in\mathbb{N}}\theta_ju_{v,j} e_j=\sum\limits_{v=1}^{s}\lambda_v \textbf{w}_v
	=\sum\limits_{v=1}^{s}|\lambda_v|\widetilde{\textbf{u}}_v.
\end{aligned}
$$
Note 
$$\frac{\textbf{x}}{1+\varepsilon} =\sum\limits_{v=1}^{s} \frac{|\lambda_{v} |}{1+\varepsilon} \widetilde{\textbf{u}}_{v} +\frac{1}{2} \left( 1-\frac{\sum\limits_{v=1}^{s} |\lambda_{v} |}{1+\varepsilon} \right) e_{1}+\frac{1}{2} \left( 1-\frac{\sum\limits_{v=1}^{s} |\lambda_{v} |}{1+\varepsilon} \right) \left( -e_{1} \right) \in \mathrm{conv}\mathcal{G}.$$
Letting $\varepsilon\to0$ proves
\begin{equation}
\textbf{x}\in\overline{\mathrm{conv}\mathcal{G}}^{\,Y},
\qquad
\textbf{x}\in\mathbb{R}^{(\mathbb{N})},\quad \|\textbf{x}\|_Y\leqslant1.
\label{shi-youxianshitubao}
\end{equation}
Finally, if $\textbf{x}\in Y\cap\mathbb{R}^{\mathbb{N}}$ and
$\|\textbf{x}\|_Y\leqslant1$, then each $P_n\textbf{x}$ satisfies \eqref{shi-youxianshitubao}. Moreover, since $P_{n}\textbf{x} \to \textbf{x}$, we have $\textbf{x} \in \overline{\mathrm{conv} \mathcal{G}}^{Y}$. This proves \eqref{shi-shiqiutubao}, and hence \itemref{xiang-shibuyuxubu}.
\end{proof}

\begin{proposition}\label{mingti-jibenhanshu}\label{mingti-jiejinyikehe}
Define
$$
\varphi(n)\triangleq
\left\|\sum\limits_{j=1}^{n}e_j\right\|_Y,
\qquad n\in\mathbb{N}.
$$
Then, for every nonempty finite set $F\subset\mathbb{N}$ and every
choice of $\omega_j \in \mathbb{C}$ with $|\omega_j|=1$, $j\in F$,
\begin{equation}
\left\|\sum\limits_{j\in F}\omega_je_j\right\|_Y=\varphi(|F|).
\label{shi-xiangweifanshu}
\end{equation}
Moreover,
\begin{equation}
\varphi(a_r)=\frac{a_r}{t_r},
\qquad r\in\mathbb{N}_0,
\label{shi-yuanzifanshu}
\end{equation}
and for all $n\geqslant a_{r-1},n,r\in\mathbb{N}$,
\begin{equation}
t_r\leqslant\frac{2r}{\varphi(r)},
\qquad
\varphi(n)\leqslant\frac{n}{t_{r-1}}.
\label{shi-canshuyujibenhanshu}
\end{equation}
For every
$n\in\mathbb{N}$ and every $\textbf{z}\in Y$ satisfying
$|\mathrm{supp}\textbf{z}|\leqslant n$, it holds that
\begin{equation}
\frac{1}{\Psi(n)}\|\textbf{z}\|_{\ell^1}
\leqslant\|\textbf{z}\|_Y\leqslant\|\textbf{z}\|_{\ell^1}.
\label{shi-youxianweibijiao}
\end{equation}
\end{proposition}

\begin{proof}

We first establish the identity
\begin{equation}
\frac{n}{\varphi(n)}
=
\sup\limits_{\substack{\textbf{z}\in Y,\ \|\textbf{z}\|_Y\leqslant1\\
\mathrm{supp}\textbf{z}\subset\{1,\cdots,n\}}}
\|\textbf{z}\|_{\ell^1}
=
\sup\limits_{k\in\mathbb{N}_0}\frac{t_k}{a_k}\min\{n,a_k\},
\qquad n\in\mathbb{N}.
\label{shi-jibenhanshugongshi}
\end{equation}

Fix $n\in\mathbb{N}$ and let $\textbf{z} \in Y$ with $\mathrm{supp}\ \textbf{z}\subset\{1,2,\cdots,n\}$. For each permutation $\pi\in S_n$, set

$$
\textbf{z}^{\pi}
\triangleq
\sum\limits_{j=1}^{n}|z_j|e_{\pi(j)}.
$$
We have
$$
\frac{1}{n!}\sum\limits_{\pi\in S_n}\textbf{z}^{\pi}
=
\frac{\|\textbf{z}\|_{\ell^1}}{n}
\sum\limits_{k=1}^{n}e_k.
$$

By \eqref{shi-zhihuanxiangwei}, for every $\pi\in S_n$,

$$
\|\textbf{z}^{\pi}\|_Y
=
\||\textbf{z}|\|_Y
=
\|\textbf{z}\|_Y.
$$

Consequently, using the triangle inequality and the definition of $\varphi(n)$,

$$
\begin{aligned}
    \frac{\varphi(n)}{n}\|\textbf{z}\|_{\ell^1}
    &=
    \left\|
    \frac{\|\textbf{z}\|_{\ell^1}}{n}
    \sum\limits_{k=1}^{n}e_k
    \right\|_Y=
    \left\|
    \frac{1}{n!}\sum\limits_{\pi\in S_n}\textbf{z}^{\pi}
    \right\|_Y\leqslant
    \frac{1}{n!}\sum\limits_{\pi\in S_n}\|\textbf{z}^{\pi}\|_Y=
    \|\textbf{z}\|_Y.
\end{aligned}
$$

Note that $$\left\| \frac{1}{\varphi (n)} \sum\limits_{j=1}^{n} e_{j} \right\|_{Y} =1,\left\| \frac{1}{\varphi (n)} \sum\limits_{j=1}^{n} e_{j} \right\|_{\ell^{1}} =\frac{n}{\varphi \left( n \right)}.$$Thus the first equality in
\eqref{shi-jibenhanshugongshi} follows. Hence

\begin{equation}
1\leqslant\frac{n}{\varphi(n)}
\leqslant\frac{n+1}{\varphi(n+1)},
\qquad n=1,2,\cdots.
\label{shi-jibenhanshudandiao}
\end{equation}
We now prove that\begin{equation}
	\frac{n}{\varphi(n)}
	=
	\sup\limits_{\textbf{u}=\left( u_{j} \right)_{j\in \mathbb{N}}\in\mathcal{G}}\left|\sum\limits_{j=1}^{n}u_j\right|.\label{shi-1}
\end{equation}
For every $\textbf{x}=(x_i)_{i\in\mathbb{N}}\in\mathrm{conv}\mathcal{G}$, write
$$
\textbf{x}=\sum\limits_{v=1}^{s}\lambda_v \textbf{u}_v,
\qquad
\textbf{u}_v\in\mathcal{G},\quad
\lambda_v\geqslant0,\quad v=1,2,\cdots,s,
\qquad
\sum\limits_{v=1}^{s}\lambda_v=1.
$$
Then
$$
\begin{aligned}
	\left|\sum\limits_{i=1}^{n}x_i\right|
	&=
	\left|\sum\limits_{i=1}^{n}\sum\limits_{v=1}^{s}\lambda_vu_{v,i}\right|
	=
	\left|\sum\limits_{v=1}^{s}\lambda_v\sum\limits_{i=1}^{n}u_{v,i}\right|\\
	&\leqslant
	\sum\limits_{v=1}^{s}\lambda_v
	\left|\sum\limits_{i=1}^{n}u_{v,i}\right|
	\leqslant
	\sup\limits_{\textbf{u}=(u_j)_{j\in\mathbb{N}}\in\mathcal{G}}
	\left|\sum\limits_{j=1}^{n}u_j\right|.
\end{aligned}
$$
Consequently,
$$
\sup\limits_{\textbf{x}=(x_j)_{j\in\mathbb{N}}\in\mathrm{conv}\mathcal{G}}
\left|\sum\limits_{j=1}^{n}x_j\right|
=
\sup\limits_{\textbf{u}=(u_j)_{j\in\mathbb{N}}\in\mathcal{G}}
\left|\sum\limits_{j=1}^{n}u_j\right|.
$$

The real linear functional
$\textbf{x}\mapsto\sum\limits_{j=1}^{n}x_j$ is continuous, since \eqref{shi-fanliqianru} gives
$$
\left|\sum\limits_{j=1}^{n}x_j\right|
\leqslant n\|\textbf{x}\|_{\ell^\infty}
\leqslant n\|\textbf{x}\|_Y.
$$
Combining this with \eqref{shi-shiqiutubao}, we obtain
$$
\sup\limits_{\substack{\textbf{x}\in Y\cap\mathbb{R}^{\mathbb{N}}\\
		\|\textbf{x}\|_Y\leqslant1}}
\left|\sum\limits_{j=1}^{n}x_j\right|=\sup\limits_{\textbf{x}=(x_j)_{j\in\mathbb{N}}\in\overline{\mathrm{conv}\mathcal{G}}^{\,Y}}
\left|\sum\limits_{j=1}^{n}x_j\right|
=
\sup\limits_{\textbf{u}=(u_j)_{j\in\mathbb{N}}\in\mathcal{G}}
\left|\sum\limits_{j=1}^{n}u_j\right|.
$$
For every $\textbf{x}\in Y\cap\mathbb{R}^{\mathbb{N}}$ satisfying $\|\textbf{x}\|_{Y}\leqslant1$, we have
$$
\left|\sum\limits_{j=1}^{n}x_j\right|
\leqslant
\sum\limits_{j=1}^{n}|x_j|
=
\left\|P_n\textbf{x}\right\|_{\ell^1}
\leqslant
\frac{n}{\varphi(n)},
$$which gives \eqref{shi-1}.

For each $k\in\mathbb{N}_0$, a sequence in $\mathcal{G}_k$ has
$a_k$ nonzero coordinates, all of absolute value $\frac{t_k}{a_k}$.
Hence
$$
\sup\limits_{\textbf{u}\in\mathcal{G}_k}\left|\sum\limits_{j=1}^{n}u_j\right|
=\frac{t_k}{a_k}\min\{n,a_k\}.
$$
Taking the supremum over $k$ proves the second
equality in \eqref{shi-jibenhanshugongshi}.

For $n=a_r$, if $k\leqslant r$, then
$$
\frac{t_k}{a_k}\min\{a_r,a_k\}=t_k\leqslant t_r.
$$
If $k>r$, then \eqref{shi-yuanzigaodu} gives
$$
\frac{t_k}{a_k}\min\{a_r,a_k\}
=\frac{t_k}{a_k}a_r\leqslant t_r.
$$
so
$\varphi(a_r)=\frac{a_r}{t_r}$, which proves
\eqref{shi-yuanzifanshu}. Formula \eqref{shi-xiangweifanshu} follows from
\eqref{shi-zhihuanxiangwei}.

Since $a_{\kappa(r)}\leqslant r$, the term $k=\kappa(r)$ in
\eqref{shi-jibenhanshugongshi} gives
$
\frac{r}{\varphi(r)}\geqslant t_{\kappa(r)}.
$
Combining this with $t_{r}\leqslant 2t_{\kappa \left( r \right)}$ yields
$t_r\leqslant\frac{2r}{\varphi(r)}$.
For every integer $n\geqslant a_{r-1}$,
\eqref{shi-jibenhanshudandiao} and
\eqref{shi-yuanzifanshu} give
$$
\frac{n}{\varphi(n)}
\geqslant\frac{a_{r-1}}{\varphi(a_{r-1})}=t_{r-1}.
$$
Thus \eqref{shi-canshuyujibenhanshu} holds.

Fix $n\in\mathbb{N}$.
Since $\left( a_{r} \right)_{r\in \mathbb{N}_0}$ is strictly increasing, $a_0=1$, and $a_r\to+\infty$,
there is a unique $r\in\mathbb{N}$ such that
$$
a_{r-1}\leqslant n<a_r.
$$
For $0\leqslant k\leqslant r-1$, we have
$$
\frac{t_k}{a_k}\min\{n,a_k\}=t_k\leqslant t_{r-1},
$$
with equality when $k=r-1$. For $k\geqslant r$,
\eqref{shi-yuanzigaodu} gives
$$
\frac{t_k}{a_k}\min\{n,a_k\}
=n\frac{t_k}{a_k}
\leqslant\frac{n}{a_r}t_r,
$$
with equality when $k=r$. Therefore
\eqref{shi-jibenhanshugongshi} yields
\begin{equation}
\frac{n}{\varphi(n)}
=\max\left\{t_{r-1},\frac{n}{a_r}t_r\right\}.
\label{shi-jibenhanshufenduan}
\end{equation}
For the choice \eqref{shi-huanmancanshu}, use $t_{r-1}\leqslant t_r$, $n<a_r$,
and the monotonicity of $\Psi$ to obtain

\begin{equation}
\frac{n}{\varphi(n)}
\leqslant t_r
\leqslant\Psi(a_{r-1})
\leqslant\Psi(n).
\label{shi-yudingjiejin}
\end{equation}

Let $\textbf{z}\in Y$ satisfy
$|\mathrm{supp}\textbf{z}|\leqslant n$.
If $\textbf{z}=\textbf{0}$, then
\eqref{shi-youxianweibijiao} is immediate.
Otherwise choose a permutation $\pi$ of $\mathbb{N}$ such that
$$
\mathrm{supp}\textbf{z}
\subset\{\pi(1),\cdots,\pi(n)\},
$$
and put
$$
\textbf{w}\triangleq\sum\limits_{j=1}^{n}z_{\pi(j)}e_j.
$$
Applying \eqref{shi-jibenhanshugongshi} to
$\frac{\textbf{w}}{\|\textbf{w}\|_Y}$ and then using
\eqref{shi-yudingjiejin} and \eqref{shi-zhihuanxiangwei}, we obtain
$$
\|\textbf{z}\|_{\ell^1}
=\|\textbf{w}\|_{\ell^1}
\leqslant\frac{n}{\varphi(n)}\|\textbf{w}\|_Y
\leqslant\Psi(n)\|\textbf{z}\|_Y.
$$
Since $\textbf{z}$ has finite support, the opposite estimate
$\|\textbf{z}\|_Y\leqslant\|\textbf{z}\|_{\ell^1}$ follows from
\eqref{shi-fanlifanshubijiao}.
This proves \eqref{shi-youxianweibijiao}.
\end{proof}

\begin{corollary}\label{tuilun-yangeqianru}\label{tuilun-jiruoshoulian}
The inclusions
$$
\ell^1\subsetneq Y\subsetneq c_0
$$
are continuous, and
\begin{equation}
\|e_j\|_Y=1,\qquad j\in\mathbb{N},
\qquad
e_j\rightharpoonup0\quad\text{in }Y.
\label{shi-jiruoshoulian}
\end{equation}
\end{corollary}

\begin{proof}

If $\textbf{z}\in\ell^1$, then
\eqref{shi-fanlifanshubijiao} implies
$$
\|P_{n}\textbf{z} -P_{m}\textbf{z}\|_{Y} =\left\|\sum\limits_{j=m+1}^n z_j e_j\right\|_Y
\leqslant\sum\limits_{j=m+1}^n|z_j|,
\qquad n>m,n,m\in \mathbb{N}_0.
$$
Hence $\textbf{z}\in Y$. It follows that $\ell^1\subset Y$ and
$$
\|\textbf{z}\|_Y\leqslant\|\textbf{z}\|_{\ell^1},
\qquad \textbf{z}\in\ell^1.
$$
Together with \eqref{shi-fanliqianru}, this proves that both
inclusions in \namedref{Corollary}{tuilun-yangeqianru} are continuous.

If $Y=\ell^1$ as sets, the inverse mapping theorem applied to the
continuous inclusion $\ell^1\hookrightarrow Y$ gives a constant
$C>0$ such that
$$
\|\textbf{z}\|_{\ell^1}\leqslant C\|\textbf{z}\|_Y,
\qquad \textbf{z}\in\ell^1.
$$
Taking $\textbf{z}=\sum\limits_{j=1}^{a_r}e_j$ and using
\eqref{shi-yuanzifanshu}, we obtain $t_r\leqslant C$ for every
$r\in\mathbb{N}_0$, contrary to $t_r\to+\infty$. Hence
$\ell^1\subsetneq Y$.

If $Y=c_0$ as sets, the inverse mapping theorem applied to the
continuous inclusion $Y\hookrightarrow c_0$ gives a constant $C>0$
such that
$$
\|\textbf{z}\|_Y\leqslant C\|\textbf{z}\|_{\ell^\infty},
\qquad \textbf{z}\in Y.
$$
Taking again $\textbf{z}=\sum\limits_{j=1}^{a_r}e_j$ gives
$\frac{a_r}{t_r}\leqslant C$. However, by \eqref{shi-yuanzigaodu}, 
$$
\lim\limits_{r\to +\infty} \frac{a_{r}}{t_{r}} \geqslant \lim\limits_{r\to +\infty} 2^{r}=+\infty.
$$
This contradiction proves $Y\subsetneq c_0$.

Let $\psi\in Y^*$. Suppose that $\psi(e_j)$ does not tend to
zero. Then there exist $\varepsilon>0$ and an infinite set
$J\subset\mathbb{N}$ such that
$|\psi(e_j)|\geqslant\varepsilon$ for every $j\in J$.
For each $r\in\mathbb{N}$, choose $F_r\subset J$ with
$|F_r|=a_r$. By \eqref{shi-xiangweifanshu} and \eqref{shi-yuanzifanshu},
$$
\left\|
\frac{t_r}{a_r}\sum\limits_{j\in F_r}
\frac{\overline{\psi(e_j)}}{|\psi(e_j)|}e_j
\right\|_Y=1.
$$
Therefore,
$$
\varepsilon t_r
\leqslant
\frac{t_r}{a_r}\sum\limits_{j\in F_r}|\psi(e_j)|
=
\left|\psi\left(
\frac{t_r}{a_r}\sum\limits_{j\in F_r}
\frac{\overline{\psi(e_j)}}{|\psi(e_j)|}e_j
\right)\right|
\leqslant\|\psi\|_{Y^*}.
$$
This contradicts $t_r\to+\infty$. Hence
$\psi(e_j)\to0$ for every $\psi\in Y^*$, and therefore
$e_j\rightharpoonup0$ in $Y$. Together with
$\|e_j\|_Y=1$ from \eqref{shi-fanlifanshubijiao}, this establishes
\eqref{shi-jiruoshoulian}.
\end{proof}

\begin{remark}\label{zhuji-butonggou}
    The space $\ell^1$ has the Schur property; that is, every weakly convergent
    sequence in $\ell^1$ converges in norm
    (see, for example, \cite[Theorem~2.3.6, p.~37]{AK16}).
    Consequently, \eqref{shi-jiruoshoulian} implies that $Y$ is not isomorphic to
    $\ell^1$.
\end{remark}

For $P\in \mathscr{P}_0\cap \mathcal{P}^{m} \left( Y\right)$, define
\begin{equation}
P^{\#}(\textbf{z})\triangleq
\sum\limits_{\substack{\alpha\in\mathbb{N}_0^{(\mathbb{N})}\\|\alpha|=m}}
|c_\alpha(P)|\textbf{z}^{\alpha}.
\label{shi-zhixishuduoxiangshi}
\end{equation}
By \eqref{shi-youxianzhankai}, only finitely many
coefficients in \eqref{shi-zhixishuduoxiangshi} are
nonzero. By \eqref{shi-fanliqianru}, $P^{\#}$ is a continuous
$m$-homogeneous polynomial on $Y$. 

\begin{lemma}\label{yinli-zhichijishu}
Let $m,d\in\mathbb{N},P\in \mathscr{P}_0\cap \mathcal{P}^{m} \left( Y\right)$ and let $\textbf{z}\in \mathbb{C}^{(\mathbb{N})}$ with $\left| \mathrm{supp} \textbf{z} \right| =d$. Then
\begin{equation}
P^{\#}(|\textbf{z}|)
\leqslant\binom{m+d-1}{m}^{\frac{1}{2}}\|P\|_{\mathcal{P}^m(Y)}\|\textbf{z}\|_Y^m.
\label{shi-zhichijishu}
\end{equation}
\end{lemma}

\begin{proof}

Without loss of generality, assume that $\textbf{z} \neq \textbf{0}$ and let $\mathrm{supp}\textbf{z}=\{j_1,\cdots,j_d\}$.  Consider $\textbf{x}=\left( x_{i} \right)_{i\in \mathbb{N}}\triangleq\frac{|\textbf{z}|}{\|\textbf{z}\|_Y}$.  Clearly, \eqref{shi-zhichijishu} is equivalent to
\begin{equation}
	P^{\#}(\textbf{x})\leqslant\binom{m+d-1}{m}^{\frac{1}{2}}\|P\|_{\mathcal{P}^m(Y)}.\label{shi-11}
\end{equation}
Consider
\begin{equation}
	\mathcal{A}\triangleq
	\{\alpha\in\mathbb{N}_0^{(\mathbb{N})}:|\alpha|=m,\ \mathrm{supp}\alpha\subset \mathrm{supp}\textbf{z}\},
	\qquad |\mathcal{A}|=\binom{m+d-1}{m}.
	\label{shi-danxiangshishu}
\end{equation}
We have
\begin{equation}
	\begin{aligned}
		&\frac{1}{(2\pi )^{d}} \int\limits_{[0,2\pi ]^{d}} \left| P\left( \sum\limits_{\ell =1}^{d} x_{j_{\ell}}e^{\mathrm{i} \theta_{\ell}}e_{j_{\ell}} \right) \right|^{2} \mathrm{d} \theta_{1} \cdots \mathrm{d} \theta_{d}\\ &\  =\frac{1}{(2\pi )^{d}} \int\limits_{[0,2\pi ]^{d}} \left| \sum\limits_{\alpha \in \mathcal{A}} c_{\alpha}(P)\textbf{x}^{\alpha} e^{\mathrm{i} \sum\limits_{\ell =1}^{d} \alpha_{j_{\ell}} \theta_{\ell}} \right|^{2} \mathrm{d} \theta_{1} \cdots \mathrm{d} \theta_{d}\\ &\  =\frac{1}{(2\pi )^{d}} \sum\limits_{\alpha ,\beta \in \mathcal{A}} c_{\alpha}(P)\overline{c_{\beta}(P)} \textbf{x}^{\alpha +\beta} \int\limits_{[0,2\pi ]^{d}} e^{\mathrm{i} \sum\limits_{\ell =1}^{d} (\alpha_{j_{\ell}} -\beta_{j_{\ell}} )\theta_{\ell}}\mathrm{d} \theta_{1} \cdots \mathrm{d} \theta_{d}\\ &\  =\frac{1}{(2\pi )^{d}} \sum\limits_{\alpha \in \mathcal{A}} \left| c_{\alpha}(P) \right|^{2} \textbf{x}^{2\alpha} \int\limits_{[0,2\pi ]^{d}} 1\mathrm{d} \theta_{1} \cdots \mathrm{d} \theta_{d} =\sum\limits_{\alpha \in \mathcal{A}} \left| c_{\alpha}(P) \right|^{2} \textbf{x}^{2\alpha}\\ &\  \leqslant \frac{1}{(2\pi )^{d}} \int\limits_{[0,2\pi ]^{d}} \|P\|_{\mathcal{P}^{m} \left( Y \right)}^{2} \cdot \left\| \sum\limits_{\ell =1}^{d} x_{j_{\ell}}e^{\mathrm{i} \theta_{\ell}}e_{j_{\ell}} \right\|_{Y}^{2m} \mathrm{d} \theta_{1} \cdots \mathrm{d} \theta_{d}\\ &\  =\frac{1}{(2\pi )^{d}} \int\limits_{[0,2\pi ]^{d}} \|P\|_{\mathcal{P}^{m} \left( Y \right)}^{2} \cdot \|\textbf{x}\|_{Y}^{2m} \mathrm{d} \theta_{1} \cdots \mathrm{d} \theta_{d} =\|P\|_{\mathcal{P}^{m} \left( Y \right)}^{2},\end{aligned}
	\label{shi-xishupingfanghe}
\end{equation}
where the penultimate equality follows from \eqref{shi-qumobaofan}.  Then the Cauchy--Schwarz inequality, \eqref{shi-danxiangshishu}, and \eqref{shi-xishupingfanghe} give
\begin{equation}
	\begin{aligned}
		P^{\#}(\textbf{x})
		&=\sum\limits_{\alpha\in\mathcal{A}}|c_\alpha(P)|\textbf{x}^\alpha
		\leqslant|\mathcal{A}|^{\frac{1}{2}}
		\left(\sum\limits_{\alpha\in\mathcal{A}}|c_\alpha(P)|^2\textbf{x}^{2\alpha}\right)^{\frac{1}{2}}\leqslant\binom{m+d-1}{m}^{\frac{1}{2}}\|P\|_{\mathcal{P}^m(Y)},
	\end{aligned}
	\label{shi-xishuhepingfangjie}
\end{equation}
which gives \eqref{shi-zhichijishu}.

\end{proof}

\begin{lemma}\label{yinli-suijixishuhua}
Let $m\in\mathbb{N}$ and let $P\in \mathscr{P}_0\cap \mathcal{P}^{m} \left( Y\right)$. 
$\textbf{x}=\left( x_{i} \right)_{i\in \mathbb{N}}\in\mathbb{C}^{(\mathbb{N})}\cap[0,+\infty)^{\mathbb{N}}$
satisfies \begin{equation}
P^{\#}(\textbf{x})
\leqslant\frac{m^m}{m!}\binom{3m-1}{m}^{\frac{1}{2}}
\left(\|\textbf{x}\|_Y+
\frac{\varphi(m)}{m}\|\textbf{x}\|_{\ell^1}\right)^m
\|P\|_{\mathcal{P}^m(Y)}.
\label{shi-xishuhuaguji}
\end{equation}
\end{lemma}

\begin{proof}

Without loss of generality, assume that $\textbf{x} \neq \textbf{0}$. Put
\begin{equation}
	I\triangleq
	\left\{j\in \mathbb{N}:x_j\geqslant\frac{\|\textbf{x}\|_{\ell^1}}{m}\right\},
	\qquad
	J\triangleq
	\left\{j\in \mathbb{N}:0<x_j<\frac{\|\textbf{x}\|_{\ell^1}}{m}\right\}.
	\label{shi-daxiaozuobiao}
\end{equation}
Then $I\cup J=\mathrm{supp}\textbf{x}$ is finite. Moreover,
$$
\frac{\|\textbf{x}\|_{\ell^1}}{m}|I|
\leqslant\sum\limits_{j\in I}x_j
\leqslant\|\textbf{x}\|_{\ell^1},
$$
and hence $|I|\leqslant m$. For
$(i_1,\cdots,i_m)\in(J\cup\{0\})^m$, define
\begin{equation}
	\begin{aligned}	\lambda_{i_1,\cdots,i_m}&\triangleq
		\left(
		1-\sum\limits_{j\in J}\frac{x_j}{\|\textbf{x}\|_{\ell^1}}
		\right)^{|\{1\leqslant \nu \leqslant m:i_\nu=0\}|}
		\prod\limits_{\substack{1\leqslant\nu\leqslant m\\ i_\nu\neq0}}
		\frac{x_{i_\nu}}{\|\textbf{x}\|_{\ell^1}},\\F_{i_1,\cdots,i_m}&	\triangleq\{i_\nu:1\leqslant\nu\leqslant m,\ i_\nu\neq0\},\end{aligned}
	\label{shi-quanzhongzhichi}
\end{equation}
and set
\begin{equation}
	\textbf{y}_{i_1,\cdots,i_m}
	\triangleq
	\sum\limits_{j\in I}x_je_j+
	\frac{\|\textbf{x}\|_{\ell^1}}{m}
	\sum\limits_{j\in F_{i_1,\cdots,i_m}}e_j.
	\label{shi-xishuxiangliang}
\end{equation}
Since
$
0\leqslant
\sum\limits_{j\in J}\frac{x_j}{\|\textbf{x}\|_{\ell^1}}
\leqslant1,
$
all $\lambda_{i_1,\cdots,i_m}$ are nonnegative, and 
\begin{equation}
	\begin{aligned}
		\sum\limits_{(i_1,\cdots,i_m)\in(J\cup\{0\})^m}
		\lambda_{i_1,\cdots,i_m}
		&=
		\left(
		1-\sum\limits_{j\in J}\frac{x_j}{\|\textbf{x}\|_{\ell^1}}
		+\sum\limits_{j\in J}\frac{x_j}{\|\textbf{x}\|_{\ell^1}}
		\right)^m
		=1.
	\end{aligned}
	\label{shi-quanzhongweiyi}
\end{equation}
For every $(i_1,\cdots,i_m)\in(J\cup\{0\})^m$, one has
$|F_{i_1,\cdots,i_m}|\leqslant m$. Thus, by \eqref{shi-fanshudandiao} and \eqref{shi-xiangweifanshu},
\begin{equation}
	\begin{aligned}
		|\mathrm{supp}\textbf{y}_{i_1,\cdots,i_m}|
		&\leqslant|I|+|F_{i_1,\cdots,i_m}|
		\leqslant2m,\\
		\|\textbf{y}_{i_1,\cdots,i_m}\|_Y
		&\leqslant
		\left\|\sum\limits_{j\in I}x_je_j\right\|_Y
		+\frac{\|\textbf{x}\|_{\ell^1}}{m}
		\left\|\sum\limits_{j\in F_{i_1,\cdots,i_m}}e_j\right\|_Y\\
		&\leqslant
		\|\textbf{x}\|_Y+
		\frac{\varphi(m)}{m}\|\textbf{x}\|_{\ell^1}.
	\end{aligned}
	\label{shi-xishufanshujie}
\end{equation}

Fix $\alpha\in\mathbb{N}_0^{(\mathbb{N})}$ with $|\alpha|=m$ and
$\textbf{x}^{\alpha}>0$, and put
$
S\triangleq J\cap\mathrm{supp}\alpha.
$
Define
$$
\mathcal I_S
\triangleq
\left\{
(i_1,\cdots,i_m)\in(J\cup\{0\})^m:
\left| \{ \nu \in \{ 1,\cdots ,m\} :i_{\nu}=j\} \right|=1
\ \text{for every }j\in S
\right\}.
$$
Then, by \eqref{shi-quanzhongzhichi},
\begin{align}
	\sum\limits_{(i_1,\cdots,i_m)\in\mathcal I_S}
	\lambda_{i_1,\cdots,i_m}
	&=
	\frac{m!}{(m-|S|)!}
	\prod\limits_{j\in S}
	\frac{x_j}{\|\textbf{x}\|_{\ell^1}}
	\nonumber\\
	&\quad\times
	\sum\limits_{\substack{
			(i_1,\cdots,i_{m-|S|})\\
			\in((J\setminus S)\cup\{0\})^{m-|S|}
	}}
	\prod\limits_{\substack{1\leqslant \nu\leqslant m-|S|\\ i_\nu=0}}
	\left(
	1-\sum\limits_{j\in J}
	\frac{x_j}{\|\textbf{x}\|_{\ell^1}}
	\right)
	\prod\limits_{\substack{1\leqslant \nu\leqslant m-|S|\\ i_\nu\neq0}}
	\frac{x_{i_\nu}}{\|\textbf{x}\|_{\ell^1}}
	\nonumber\\
	&=
	\frac{m!}{(m-|S|)!}
	\prod\limits_{j\in S}
	\frac{x_j}{\|\textbf{x}\|_{\ell^1}}
	\left(
	1-\sum\limits_{j\in J}
	\frac{x_j}{\|\textbf{x}\|_{\ell^1}}
	+
	\sum\limits_{j\in J\setminus S}
	\frac{x_j}{\|\textbf{x}\|_{\ell^1}}
	\right)^{m-|S|}
	\nonumber\\
	&=
	\frac{m!}{(m-|S|)!}
	\prod\limits_{j\in S}
	\frac{x_j}{\|\textbf{x}\|_{\ell^1}}
	\left(
	1-\sum\limits_{j\in S}
	\frac{x_j}{\|\textbf{x}\|_{\ell^1}}
	\right)^{m-|S|}.
	\label{shi-zhidingzuobiaoquan}
\end{align}
By \eqref{shi-daxiaozuobiao},
$
\sum\limits_{j\in S}\frac{x_j}{\|\textbf{x}\|_{\ell^1}}
\leqslant\frac{|S|}{m},
$
and therefore \eqref{shi-zhidingzuobiaoquan} is at least
\begin{equation}
	\begin{aligned}
		&\frac{m!}{(m-|S|)!}
		\prod\limits_{j\in S}\frac{x_j}{\|\textbf{x}\|_{\ell^1}}
		\left(\frac{m-|S|}{m}\right)^{m-|S|}\\
		&\qquad=
		\frac{m!}{m^m}
		\frac{(m-|S|)^{m-|S|}}{(m-|S|)!}
		\prod\limits_{j\in S}
		\frac{m x_j}{\|\textbf{x}\|_{\ell^1}}
		\geqslant
		\frac{m!}{m^m}
		\prod\limits_{j\in S}
		\frac{m x_j}{\|\textbf{x}\|_{\ell^1}}.
	\end{aligned}
	\label{shi-zhidingzuobiaoquanjie}
\end{equation}

It follows from \eqref{shi-zhidingzuobiaoquanjie} that
\begin{equation}
	\begin{aligned}
		&\sum\limits_{(i_1,\cdots,i_m)\in(J\cup\{0\})^m}
		\lambda_{i_1,\cdots,i_m}
		\textbf{y}_{i_1,\cdots,i_m}^{\alpha}\geqslant
		\sum\limits_{(i_1,\cdots,i_m)\in\mathcal I_S}
		\lambda_{i_1,\cdots,i_m}
		\textbf{y}_{i_1,\cdots,i_m}^{\alpha}
		\\&\qquad =
		\prod\limits_{j\in I}x_j^{\alpha_j}
		\prod\limits_{j\in S}
		\left(
		\frac{\|\textbf{x}\|_{\ell^1}}{m}
		\right)^{\alpha_j}
		\sum\limits_{(i_1,\cdots,i_m)\in\mathcal I_S}
		\lambda_{i_1,\cdots,i_m}
		\\
		&\qquad\geqslant
		\frac{m!}{m^m}
		\prod\limits_{j\in I}x_j^{\alpha_j}
		\prod\limits_{j\in S}
		\left(
		\frac{\|\textbf{x}\|_{\ell^1}}{m}
		\right)^{\alpha_j}
		\frac{m x_j}{\|\textbf{x}\|_{\ell^1}}
		\\
		&\qquad=
		\frac{m!}{m^m}
		\prod\limits_{j\in I}x_j^{\alpha_j}
		\prod\limits_{j\in S}
		\left(
		\frac{\|\textbf{x}\|_{\ell^1}}{m}
		\right)^{\alpha_j-1}x_j
		\\
		&\qquad\geqslant
		\frac{m!}{m^m}
		\prod\limits_{j\in I}x_j^{\alpha_j}
		\prod\limits_{j\in S}x_j^{\alpha_j}
		=
		\frac{m!}{m^m}\textbf{x}^{\alpha},
	\end{aligned}
	\label{shi-dandiaoxishuquanjie}
\end{equation}where the last inequality follows from $0<x_{j}<\frac{\|\textbf{x}\|_{\ell^{1}}}{m}, j\in S\subset J$. If
$\textbf{x}^{\alpha}=0$, \eqref{shi-dandiaoxishuquanjie} also holds. Multiplying
\eqref{shi-dandiaoxishuquanjie} by $|c_\alpha(P)|$ and summing over
$|\alpha|=m$ yields
\begin{equation}
	\frac{m!}{m^m}P^{\#}(\textbf{x})
	\leqslant
	\sum\limits_{(i_1,\cdots,i_m)\in(J\cup\{0\})^m}
	\lambda_{i_1,\cdots,i_m}
	P^{\#}(\textbf{y}_{i_1,\cdots,i_m}).
	\label{shi-jiaquanduoxiangshijie}
\end{equation}
By \namedref{Lemma}{yinli-zhichijishu} and \eqref{shi-xishufanshujie}, for every
$(i_1,\cdots,i_m)\in(J\cup\{0\})^m$,
$$
P^{\#}(\textbf{y}_{i_1,\cdots,i_m})
\leqslant
\binom{3m-1}{m}^{\frac{1}{2}}
\left(
\|\textbf{x}\|_Y+
\frac{\varphi(m)}{m}\|\textbf{x}\|_{\ell^1}
\right)^m
\|P\|_{\mathcal{P}^m(Y)}.
$$
Combining this estimate with \eqref{shi-quanzhongweiyi} and
\eqref{shi-jiaquanduoxiangshijie} gives
$$
P^{\#}(\textbf{x})
\leqslant
\frac{m^m}{m!}\binom{3m-1}{m}^{\frac{1}{2}}
\left(
\|\textbf{x}\|_Y+
\frac{\varphi(m)}{m}\|\textbf{x}\|_{\ell^1}
\right)^m
\|P\|_{\mathcal{P}^m(Y)},
$$
which is \eqref{shi-xishuhuaguji}.
\end{proof}

For a complex normed space $E$ and $m\in\mathbb{N}$, let
$\mathcal{L}^m(E)$ denote the space of continuous complex $m$-linear
forms on $E$, with norm
\begin{equation}
\|B\|_{\mathcal{L}^m(E)}\triangleq
\sup\limits_{\substack{\textbf{z}^{(1)},\cdots,\textbf{z}^{(m)}\in E\\
\|\textbf{z}^{(j)}\|_E\leqslant1\ (1\leqslant j\leqslant m)}}
|B(\textbf{z}^{(1)},\cdots,\textbf{z}^{(m)})|.
\label{shi-duoxianxingfanshu}
\end{equation}
When no ambiguity can arise, we write $\|B\|$ for
$\|B\|_{\mathcal{L}^m(E)}$.

A form $B\in\mathcal{L}^m(Y)$ depends on finitely many coordinates
if there is a finite nonempty set $W\subset\mathbb{N}$ such that
$$
B(\textbf{z}^{(1)},\cdots,\textbf{z}^{(m)})
=B(P_W\textbf{z}^{(1)},\cdots,P_W\textbf{z}^{(m)}).
$$
For such a $m$-linear form, put
$b_{i_1,\cdots,i_m}\triangleq B(e_{i_1},\cdots,e_{i_m})$. Then
\begin{equation}
B(\textbf{z}^{(1)},\cdots,\textbf{z}^{(m)})
=\sum\limits_{(i_1,\cdots,i_m)\in W^m}
b_{i_1,\cdots,i_m}z_{i_1}^{(1)}\cdots z_{i_m}^{(m)},
\label{shi-duoxianxingzhankai}
\end{equation}
and we define
\begin{equation}
B^{\#}(\textbf{z}^{(1)},\cdots,\textbf{z}^{(m)})
\triangleq
\sum\limits_{(i_1,\cdots,i_m)\in W^m}
|b_{i_1,\cdots,i_m}|z_{i_1}^{(1)}\cdots z_{i_m}^{(m)}.
\label{shi-zhixishuxingshi}
\end{equation}
By \eqref{shi-fanliqianru}, this is a continuous complex $m$-linear form. A real or complex $m$-linear form $B$ is symmetric if
\begin{equation}
B(\textbf{z}^{(\pi(1))},\cdots,\textbf{z}^{(\pi(m))})
=B(\textbf{z}^{(1)},\cdots,\textbf{z}^{(m)}),
\qquad \pi\in S_m.
\label{shi-duoxianxingduichen}
\end{equation}
Put $$\mathcal{G}^{+}\triangleq
\{\textbf{u}=\left( u_{j} \right)_{j\in \mathbb{N}}\in\mathcal{G}:u_j\geqslant0\text{ for every }j\in\mathbb{N}\}.$$

\begin{lemma}\label{yinli-yuanziqufan}
	Let $m\in\mathbb{N}$ and let $B\in\mathcal{L}^m(Y)$ depend on
	finitely many coordinates. Then
	\begin{equation}
		\|B^{\#}\|_{\mathcal{L}^m(Y)}
		=
		\sup\limits_{\textbf{u}^{(1)},\cdots,\textbf{u}^{(m)}\in\mathcal{G}^{+}}
		B^{\#}
		\left(
		\textbf{u}^{(1)},\cdots,\textbf{u}^{(m)}
		\right).
		\label{shi-yuanziqufan}
	\end{equation}
	If $B$ is symmetric and
	$P(\textbf{z})\triangleq B(\textbf{z},\cdots,\textbf{z})$, then
	\begin{equation}
		P^{\#}(\textbf{z})
		=
		B^{\#}(\textbf{z},\cdots,\textbf{z}),
		\qquad
		\|P\|_{\mathcal{P}^m(Y)}
		\leqslant
		\|B\|_{\mathcal{L}^m(Y)},
		\label{shi-duijiaoxishu}
	\end{equation}
	and, if
	$$
	\textbf{u}^{(1)},\cdots,\textbf{u}^{(m)}\in Y,
	\qquad
	\textbf{u}^{(j)}=|\textbf{u}^{(j)}|,
	\quad 1\leqslant j\leqslant m,
	$$
	then
	\begin{equation}
		B^{\#}
		\left(
		\textbf{u}^{(1)},\cdots,\textbf{u}^{(m)}
		\right)
		\leqslant
		\frac{m^m}{m!}
		P^{\#}
		\left(
		\frac{1}{m}
		\sum\limits_{j=1}^{m}\textbf{u}^{(j)}
		\right).
		\label{shi-duichenpingjun}
	\end{equation}
\end{lemma}

\begin{proof}
	Set
	$$
	K\triangleq
	\sup\limits_{\textbf{u}^{(1)},\cdots,\textbf{u}^{(m)}\in\mathcal{G}^{+}}
	B^{\#}\left(\textbf{u}^{(1)},\cdots,\textbf{u}^{(m)}\right).
	$$
	For $\textbf{u}^{(1)},\cdots,\textbf{u}^{(m)}\in\mathcal{G}$, by
	\eqref{shi-zhixishuxingshi},
	$$
	\left|
	B^{\#}\left(\textbf{u}^{(1)},\cdots,\textbf{u}^{(m)}\right)
	\right|
	\leqslant
	B^{\#}\left(
	|\textbf{u}^{(1)}|,\cdots,|\textbf{u}^{(m)}|
	\right)
	\leqslant K.
	$$
	If
	$$
	\textbf{x}^{(j)}
	=
	\sum\limits_{\nu=1}^{N_j}
	\lambda_{\nu}^{(j)}\textbf{u}_{\nu}^{(j)}
	\in\mathrm{conv}\mathcal{G},
	\qquad
	\lambda_{\nu}^{(j)}\geqslant0,
	\qquad
	\sum\limits_{\nu=1}^{N_j}\lambda_{\nu}^{(j)}=1,j=1,2,\cdots,m,
	$$
	then multilinearity gives
	\begin{align*}
		\left|
		B^{\#}\left(\textbf{x}^{(1)},\cdots,\textbf{x}^{(m)}\right)
		\right|
		&\leqslant
		\sum\limits_{\nu_1=1}^{N_1}\cdots
		\sum\limits_{\nu_m=1}^{N_m}
		\left(
		\prod\limits_{j=1}^{m}\lambda_{\nu_j}^{(j)}
		\right)
		\left|
		B^{\#}\left(
		\textbf{u}_{\nu_1}^{(1)},\cdots,
		\textbf{u}_{\nu_m}^{(m)}
		\right)
		\right|
		\\
		&\leqslant
		K
		\sum\limits_{\nu_1=1}^{N_1}\cdots
		\sum\limits_{\nu_m=1}^{N_m}
		\prod\limits_{j=1}^{m}\lambda_{\nu_j}^{(j)}
		=
		K.
	\end{align*}
	Thus, by \eqref{shi-shiqiutubao} and continuity,
	$$
	\left|
	B^{\#}\left(\textbf{x}^{(1)},\cdots,\textbf{x}^{(m)}\right)
	\right|
	\leqslant K
	$$
	whenever $\textbf{x}^{(j)}\in Y\cap\mathbb{R}^{\mathbb{N}}$ and
	$\|\textbf{x}^{(j)}\|_{Y}\leqslant1$, $j=1,\cdots,m$.
	
	Now let $\textbf{z}^{(j)}\in Y$ with
	$\|\textbf{z}^{(j)}\|_{Y}\leqslant1,j=1,2,\cdots,m$. By
	\eqref{shi-qumobaofan},
	$$
	|\textbf{z}^{(j)}|\in Y\cap\mathbb{R}^{\mathbb{N}},
	\qquad
	\left\||\textbf{z}^{(j)}|\right\|_{Y}
	=
	\|\textbf{z}^{(j)}\|_{Y}
	\leqslant1,
	$$
	and therefore
	$$
	\left|
	B^{\#}\left(\textbf{z}^{(1)},\cdots,\textbf{z}^{(m)}\right)
	\right|
	\leqslant
	B^{\#}\left(
	|\textbf{z}^{(1)}|,\cdots,|\textbf{z}^{(m)}|
	\right)
	\leqslant K.
	$$
	Hence
	$\|B^{\#}\|_{\mathcal{L}^{m}(Y)}\leqslant K$.
	Conversely, by \eqref{shi-shiqiutubao},
	$$
	\mathcal{G}^{+}
	\subset
	\mathcal{G}
	\subset
	\overline{\mathrm{conv}\mathcal{G}}^{\,Y}
	=
	\left\{
	\textbf{x}\in Y\cap\mathbb{R}^{\mathbb{N}}:
	\|\textbf{x}\|_{Y}\leqslant1
	\right\},
	$$
	so that $K\leqslant\|B^{\#}\|_{\mathcal{L}^{m}(Y)}$. This proves
	\eqref{shi-yuanziqufan}.
	
By the symmetry of $B$ and the definitions of $P^{\#}$ and $B^{\#}$, \eqref{shi-duijiaoxishu} follows immediately.
	
	Finally, let $\textbf{u}^{(j)}\in Y$ satisfy
	$\textbf{u}^{(j)}=|\textbf{u}^{(j)}|$, $j=1,\cdots,m$. Since
	$B^{\#}$ has nonnegative coefficients,
	$$
	B^{\#}\left(
	\textbf{u}^{(j_1)},\cdots,\textbf{u}^{(j_m)}
	\right)
	\geqslant0
	\qquad
	(1\leqslant j_1,\cdots,j_m\leqslant m).
	$$
	Hence, using symmetry,
	\begin{align*}
		&m^{m}
		P^{\#}\left(
		\frac{1}{m}\sum\limits_{j=1}^{m}\textbf{u}^{(j)}
		\right)
		=
		B^{\#}\left(
		\sum\limits_{j=1}^{m}\textbf{u}^{(j)},
		\cdots,
		\sum\limits_{j=1}^{m}\textbf{u}^{(j)}
		\right)
=
		\sum\limits_{j_1,\cdots,j_m=1}^{m}
		B^{\#}\left(
		\textbf{u}^{(j_1)},\cdots,\textbf{u}^{(j_m)}
		\right)
		\\
		&\qquad\geqslant
		\sum\limits_{\pi\in S_m}
		B^{\#}\left(
		\textbf{u}^{(\pi(1))},\cdots,\textbf{u}^{(\pi(m))}
		\right)=
		m!\,
		B^{\#}\left(
		\textbf{u}^{(1)},\cdots,\textbf{u}^{(m)}
		\right),
	\end{align*}
	which yields \eqref{shi-duichenpingjun}.
\end{proof}

\begin{lemma}\label{yinli-youxianwang}
Let $E$ be a complex normed space of dimension $d\geqslant1$, and
let $m\geqslant2$ be an integer. There is a finite set
$\mathcal{N}\subset\overline{B_E}$ such that
\begin{equation}
|\mathcal{N}|\leqslant(1+4m)^{2d},
\qquad
\|C\|_{\mathcal{L}^{m-1}(E)}
\leqslant\frac{2m}{m+1}
\max\limits_{\xi_{1} ,\cdots ,\xi_{m-1} \in \mathcal{N}} |C(\xi_{1} ,\cdots ,\xi_{m-1} )|
\label{shi-youxianwangguji}
\end{equation}
for every $C\in\mathcal{L}^{m-1}(E)$.
\end{lemma}

\begin{proof}
	\proofstep{1}
	
	For $\textbf{a}\in E$ and $r>0$, write
	$$
	B_E(\textbf{a},r)
	\triangleq
	\left\{
	\textbf{x}\in E:
	\|\textbf{x}-\textbf{a}\|_E<r
	\right\}.
	$$
	Consider the largest $q\in\mathbb{N}$ for which there exist
	$\xi^{(1)},\cdots,\xi^{(q)}\in E$ satisfying
	$$
	\|\xi^{(j)}\|_E\leqslant1,
	\qquad
	1\leqslant j\leqslant q,
	$$
	and
	$$
	\|\xi^{(i)}-\xi^{(j)}\|_E>\frac{1}{2m},
	\qquad
	1\leqslant i<j\leqslant q.
	$$
	Such a largest $q$ exists since $\overline{B_E}$ is compact.
	
	The sets $\left\{ B_{E}\left( \xi^{(j)} ,\frac{1}{4m} \right) \right\}_{j=1}^{q}$are pairwise disjoint. Moreover, if
	$\textbf{x}\in B_E\left(\xi^{(j)},\frac{1}{4m}\right)$, then
	$$
	\|\textbf{x}\|_E
	\leqslant
	\|\textbf{x}-\xi^{(j)}\|_E
	+\|\xi^{(j)}\|_E
	<
	1+\frac{1}{4m}.
	$$
	Hence
	$$
	\bigcup\limits_{j=1}^{q}
	B_E\left(\xi^{(j)},\frac{1}{4m}\right)
	\subset
	\left(1+\frac{1}{4m}\right)B_E.
	$$
	
	Let $\mathrm{vol}_{2d}$ denote Lebesgue measure on the
	underlying real $2d$-dimensional vector space of $E$. We have
	$$
	\begin{aligned}
		&q\left(\frac{1}{4m}\right)^{2d}
		\mathrm{vol}_{2d}(B_E)
		=
		\sum\limits_{j=1}^{q}
		\mathrm{vol}_{2d}
		\left(
		B_E\left(\xi^{(j)},\frac{1}{4m}\right)
		\right)
		\\
		&\quad=
		\mathrm{vol}_{2d}
		\left(
		\bigcup\limits_{j=1}^{q}
		B_E\left(\xi^{(j)},\frac{1}{4m}\right)
		\right)
		\\
		&\quad\leqslant
		\mathrm{vol}_{2d}
		\left(
		\left(1+\frac{1}{4m}\right)B_E
		\right)
		=
		\left(1+\frac{1}{4m}\right)^{2d}
		\mathrm{vol}_{2d}(B_E).
	\end{aligned}
	$$
	Therefore
$
	q
	\leqslant
	\left(
	\frac{1+\frac{1}{4m}}
	{\frac{1}{4m}}
	\right)^{2d}
	=
	(1+4m)^{2d}.
$ Put
	$$
	\mathcal{N}
	\triangleq
	\left\{
	\xi^{(1)},\cdots,\xi^{(q)}
	\right\}.
	$$
	Then
$
	|\mathcal{N}|
	\leqslant
	(1+4m)^{2d}.
$

	\proofstep{2}
	
	Let $C\in\mathcal{L}^{m-1}(E)$, and put
$
	M_C
	\triangleq
\max\limits_{\xi_{1} ,\cdots ,\xi_{m-1} \in \mathcal{N}} |C(\xi_{1} ,\cdots ,\xi_{m-1} )|.
$
	Fix
	$\textbf{x}^{(1)},\cdots,\textbf{x}^{(m-1)}
	\in\overline{B_E}$.
		By the maximality of $q$, for every $1\leqslant j\leqslant m-1$ there exists
	$\xi^{(k_{j})}\in\mathcal{N}$ such that
	$$
	\|\textbf{x}^{(j)}-\xi^{(k_{j})}\|_E
	\leqslant
	\frac{1}{2m}.
	$$
	By multilinearity,
	$$
	\begin{aligned}
		&\left|
		C\left(
		\textbf{x}^{(1)},\cdots,\textbf{x}^{(m-1)}
		\right)
		-
		C\left(
		\xi^{(k_1)},\cdots,\xi^{(k_{m-1})}
		\right)
		\right|
		\\
		&\quad\leqslant
		\sum\limits_{j=1}^{m-1}
		\left|
		C\left(
		\xi^{(k_1)},\cdots,\xi^{(k_{j-1})},
		\textbf{x}^{(j)}-\xi^{(k_j)},
		\textbf{x}^{(j+1)},\cdots,\textbf{x}^{(m-1)}
		\right)
		\right|
		\\
		&\quad\leqslant
		\sum\limits_{j=1}^{m-1}
		\|C\|_{\mathcal{L}^{m-1}(E)}
		\left\|
		\textbf{x}^{(j)}-\xi^{(k_j)}
		\right\|_E\leqslant
		\frac{m-1}{2m}
		\|C\|_{\mathcal{L}^{m-1}(E)}.
	\end{aligned}
	$$
	It follows that
	$$
\begin{aligned}\left| C\left( \textbf{x}^{(1)} ,\cdots ,\textbf{x}^{(m-1)} \right) \right|&\leqslant \left| C\left( \xi^{(k_{1})} ,\cdots ,\xi^{(k_{m-1})} \right) \right| +\frac{m-1}{2m} \| C\|_{\mathcal{L}^{m-1} (E)}\\ &\leqslant M_{C}+\frac{m-1}{2m} \| C\|_{\mathcal{L}^{m-1} (E)} .\end{aligned}
	$$
Hence
	$$
	\|C\|_{\mathcal{L}^{m-1}(E)}
	\leqslant
	M_C
	+
	\frac{m-1}{2m}
	\|C\|_{\mathcal{L}^{m-1}(E)}.
	$$
	Therefore
	$$
	\|C\|_{\mathcal{L}^{m-1}(E)}
	\leqslant
	\frac{2m}{m+1}
	\max\limits_{\xi^{(1)},\cdots,\xi^{(m-1)}\in\mathcal{N}}
	\left|
	C\left(\xi^{(1)},\cdots,\xi^{(m-1)}\right)
	\right|.
	$$
	This proves \eqref{shi-youxianwangguji}.
\end{proof}

For $m\in\mathbb{N}$, put
\begin{equation}
\theta_m\triangleq1+\frac{\varphi(m)}{m}t_{m^2-1},
\qquad K_1\triangleq1,
\qquad
K_m\triangleq
\left(\frac{m^m}{m!}\right)^2
\binom{3m-1}{m}^{\frac{1}{2}}\theta_m^m
,\quad m=2,3,\cdots.
\label{shi-guinachangshu}
\end{equation}

\begin{proposition}\label{mingti-duoxianxingguji}
For every $m\in\mathbb{N}$ and every symmetric
$B\in\mathcal{L}^m(Y)$ depending on finitely many coordinates,
\begin{equation}
\|B^{\#}\|_{\mathcal{L}^m(Y)}
\leqslant K_m\|B\|_{\mathcal{L}^m(Y)}.
\label{shi-duoxianxingguji}
\end{equation}
\end{proposition}

\begin{proof}
The argument is inspired by
\cite[Lemma~4.2 and  Theorem~4.3, pp.~59--62]{DK05}.
We first verify that
\begin{equation}
K_m\geqslant4K_{m-1},\qquad m\geqslant2.
\label{shi-changshuzengzhang}
\end{equation}
Firstly,
$$K_{2}=4\sqrt{10} \theta_{2}^{2} \geqslant 4\sqrt{10} >4K_{1},$$Then for $m\geqslant 3$, since $\left( t_{r} \right)_{r\in \mathbb{N}_{0}} ,\left( \varphi \left( n \right) \right)_{n\in \mathbb{N}}$ are increasing, we have
$$
\begin{aligned}\theta_{m}&=1+\frac{\varphi \left( m \right)}{m} t_{m^{2}-1}\geqslant 1+\frac{\varphi \left( m-1 \right)}{m} t_{\left( m-1 \right)^{2} -1}\\ &=1+\frac{\left( \theta_{m-1} -1 \right) \left( m-1 \right)}{m} =\frac{1}{m} +\frac{m-1}{m} \theta_{m-1}.\end{aligned}
$$
Hence
$$\log \theta_{m} \geqslant \log \left( \frac{1}{m} \cdot 1+\left( 1-\frac{1}{m} \right) \theta_{m-1} \right) \geqslant \frac{1}{m} \log 1+\left( 1-\frac{1}{m} \right) \log \theta_{m-1},$$
that is, $\theta_m^m\geqslant\theta_{m-1}^{m-1}$. Moreover, for $m\geqslant 3$, we have
$$
\frac{\frac{m^{m}}{m!}}{\frac{\left( m-1 \right)^{m-1}}{\left( m-1 \right) !}} =\left( 1+\frac{1}{m-1} \right)^{m-1} \geqslant 2,\frac{\binom{3m-1}{m}}{\binom{3m-4}{m-1}} =\frac{3\left( 3m-1 \right) \left( 3m-2 \right)}{2m\left( 2m-1 \right)} >1 .
$$
These inequalities prove \eqref{shi-changshuzengzhang}.

We now argue by induction on $m$. For $m=1$, put
$$
B\left( \textbf{z} \right) = \sum\limits_{i=1}^{n} B\left( e_{i} \right) z_{i},\textbf{z} =\left( z_{i} \right)_{i\in \mathbb{N}} \in Y, \omega_j\triangleq
\begin{cases}
\frac{\overline{B(e_j)}}{|B(e_j)|},&B(e_j)\neq0,\\[2pt]
1,&B(e_j)=0,
\end{cases}j=1,2,\cdots.
$$For any $\textbf{z} \in Y$, by \eqref{shi-zhihuanxiangwei} and \eqref{shi-fanshudandiao}, we have
$$\begin{aligned}\left| B^{\#}\left( \textbf{z} \right) \right|&=\left| \sum\limits_{i=1}^{n} \left| B\left( e_{i} \right) \right| z_{i} \right| =\left| \sum\limits_{i=1}^{n} \omega_{i}B\left( e_{i} \right) z_{i} \right|\\ &\leqslant \|B\| \cdot \left\| \sum\limits_{i=1}^{n} \omega_{i}z_{i}e_{i} \right\|_{Y} =\|B\| \cdot \left\| \sum\limits_{i=1}^{n} z_{i}e_{i} \right\|_{Y}\\ &\leqslant \|B\| \cdot \|\textbf{z}\|_{Y},\end{aligned}$$
and hence $\|B^{\#}\|\leqslant\|B\|$. This proves \eqref{shi-duoxianxingguji} in the case $m=1$.
Assume that $m\geqslant 2$ and that \eqref{shi-duoxianxingguji} holds with $m$ replaced by $m-1$. Fix $B$ as in the statement and set
$P(\textbf{z})\triangleq B(\textbf{z},\cdots,\textbf{z})$.
By \namedref{Lemma}{yinli-yuanziqufan} \eqref{shi-yuanziqufan},
it suffices to estimate
$B^{\#}(\textbf{u}^{(1)},\cdots,\textbf{u}^{(m)})$ for vectors of the form
\begin{equation}
	\textbf{u}^{(j)}
	=
	\frac{t_{r_j}}{a_{r_j}}
	\sum\limits_{k\in F_j}e_k,
	\qquad
	r_j\in\mathbb{N}_0,\quad
	F_j\subset\mathbb{N},\quad
	|F_j|=a_{r_j},
	\qquad
	1\leqslant j\leqslant m.
	\label{shi-guinayuanzi}
\end{equation}
By \namedref{Proposition}{mingti-jiejinyikehe}
\begin{equation}
\|\textbf{u}^{(j)}\|_Y=1,
\qquad \|\textbf{u}^{(j)}\|_{\ell^1}=t_{r_j},\qquad j=1,2,\cdots,m.
\label{shi-yuanziliangfanshu}
\end{equation}
Put $r\triangleq\max\limits_{1\leqslant j\leqslant m}r_j$.

\proofstep{1}
If $r<m^2$, set $\textbf{x}\triangleq\frac{1}{m}\sum\limits_{j=1}^{m}\textbf{u}^{(j)}$.
Then
$\textbf{x}\in\mathbb{C}^{(\mathbb{N})}\cap[0,+\infty)^{\mathbb{N}}$.
By \eqref{shi-yuanziliangfanshu}, the triangle inequality, and
monotonicity of $(t_s)_{s\in\mathbb{N}_0}$, we have
$$
\begin{aligned}
	\|\textbf{x}\|_Y
	\leqslant\frac{1}{m}\sum\limits_{j=1}^{m}\|\textbf{u}^{(j)}\|_Y=1,\qquad	\|\textbf{x}\|_{\ell^1}
	=\frac{1}{m}\sum\limits_{j=1}^{m}t_{r_j}
	\leqslant t_r\leqslant t_{m^2-1}.
\end{aligned}
$$
Applying \namedref{Lemma}{yinli-suijixishuhua} \eqref{shi-xishuhuaguji} to $\textbf{x}$ and using
the definition of $\theta_m$ in \eqref{shi-guinachangshu}, we obtain
$$
\begin{aligned}
	P^{\#}(\textbf{x})
	&\leqslant
	\frac{m^m}{m!}\binom{3m-1}{m}^{\frac{1}{2}}
	\left(
	\|\textbf{x}\|_Y+
	\frac{\varphi(m)}{m}\|\textbf{x}\|_{\ell^1}
	\right)^m\|P\|\leqslant
	\frac{m^m}{m!}\binom{3m-1}{m}^{\frac{1}{2}}
	\theta_m^m\|P\|.
\end{aligned}
$$
Consequently, \namedref{Lemma}{yinli-yuanziqufan} \eqref{shi-duichenpingjun}, \eqref{shi-duijiaoxishu},
and \eqref{shi-guinachangshu} yield
\begin{equation}
	\begin{aligned}
		B^{\#}(\textbf{u}^{(1)},\cdots,\textbf{u}^{(m)})
		&\leqslant\frac{m^m}{m!}P^{\#}(\textbf{x})\leqslant
		\left(\frac{m^m}{m!}\right)^2
		\binom{3m-1}{m}^{\frac{1}{2}}\theta_m^m\|B\|
		=K_m\|B\|.
	\end{aligned}
	\label{shi-xiaocengguji}
\end{equation}

\proofstep{2}
Suppose that $r\geqslant m^2$. Without loss of generality, assume that $r_m=r$.
Since  $r-2\geqslant2$,
the defining inequalities \eqref{shi-huanmancanshu} for $\left( t_{r} \right)_{r\in \mathbb{N}_{0}}$ give
\begin{equation}
\frac{t_r}{t_{r-2}}
=\frac{t_r}{t_{r-1}}\frac{t_{r-1}}{t_{r-2}}
\leqslant\left(1+\frac{1}{r}\right)
\left(1+\frac{1}{r-1}\right)=\frac{r+1}{r-1}.
\label{shi-liangcengbizhi}
\end{equation}
Let $J\triangleq\{j\in\{1,\cdots,m-1\}:r_j>r-2\}$.
For $j\in J$,
\begin{equation}
	\begin{aligned}
		&\frac{t_{r_j}}{t_{r-2}}\binom{a_{r_j}}{a_{r-2}}^{-1}
		\sum\limits_{\substack{D\subset F_j\\ |D|=a_{r-2}}}
		\left(\frac{t_{r-2}}{a_{r-2}}\sum\limits_{k\in D}e_k\right)\\
		&\quad=
		\frac{t_{r_j}}{a_{r-2}}\binom{a_{r_j}}{a_{r-2}}^{-1}
		\sum\limits_{k\in F_j}
		\left|\left\{D\subset F_j:|D|=a_{r-2},\ k\in D\right\}\right|e_k\\
		&\quad=
		\frac{t_{r_j}}{a_{r-2}}\binom{a_{r_j}}{a_{r-2}}^{-1}
		\binom{a_{r_j}-1}{a_{r-2}-1}\sum\limits_{k\in F_j}e_k
		=
		\frac{t_{r_j}}{a_{r_j}}\sum\limits_{k\in F_j}e_k
		=
		\textbf{u}^{(j)}.
	\end{aligned}
	\label{shi-jiangcengpingjun}
\end{equation}
For $1\leqslant j<m$, define the finite nonempty set
$$
\mathcal E_j\triangleq
\begin{cases}
\left\{\frac{t_{r-2}}{a_{r-2}}\sum\limits_{k\in D}e_k:
D\subset F_j,\ |D|=a_{r-2}\right\},&j\in J,\\[2pt]
\{\textbf{u}^{(j)}\},&j\notin J,
\end{cases}
\qquad
\Lambda\triangleq\prod\limits_{j\in J}\frac{t_{r_j}}{t_{r-2}}.
$$
Note that $$
|\mathcal E_j|
=
\begin{cases}
	\displaystyle \binom{a_{r_j}}{a_{r-2}}, & j\in J,\\[6pt]
	1, & j\notin J.
\end{cases}$$
Choose $\textbf{v}^{(j)}\in\mathcal E_j$, $1\leqslant j<m$ such that $$
B^{\#}\left(
\textbf{v}^{(1)},\cdots,\textbf{v}^{(m-1)},\textbf{u}^{(m)}
\right)
= \max\limits_{\left(\textbf{w}^{(1)},\cdots,\textbf{w}^{(m-1)}\right)
	\in\prod\limits_{j=1}^{m-1}\mathcal E_j}
B^{\#}\left(
\textbf{w}^{(1)},\cdots,\textbf{w}^{(m-1)},\textbf{u}^{(m)}
\right).
$$ Set
$\textbf{v}^{(m)}=\textbf{u}^{(m)}$. Every $\textbf{v}^{(j)}$ with $j<m$ belongs to
$\mathcal G^+\cap\bigcup\limits_{k=0}^{r-2}\mathcal G_k$.
By multilinearity and \eqref{shi-jiangcengpingjun},
\begin{equation}
	\begin{aligned}
		B^{\#}(\textbf{u}^{(1)},\cdots,\textbf{u}^{(m)})
		&=
		\left(\prod\limits_{j\in J}\frac{t_{r_j}}{t_{r-2}}\right)
		B^{\#}\left(
		\frac{1}{|\mathcal E_1|}
		\sum\limits_{\textbf{w}^{(1)}\in\mathcal E_1}\textbf{w}^{(1)},
		\cdots,
		\frac{1}{|\mathcal E_{m-1}|}
		\sum\limits_{\textbf{w}^{(m-1)}\in\mathcal E_{m-1}}\textbf{w}^{(m-1)},
		\textbf{u}^{(m)}
		\right)\\
		&=
		\frac{\Lambda}{\prod\limits_{j=1}^{m-1}|\mathcal E_j|}
		\sum\limits_{\substack{\textbf{w}^{(j)}\in\mathcal E_j\\1\leqslant j<m}}
		B^{\#}\left(
		\textbf{w}^{(1)},\cdots,\textbf{w}^{(m-1)},\textbf{u}^{(m)}
		\right)\\
		&\leqslant
		\frac{\Lambda}{\prod\limits_{j=1}^{m-1}|\mathcal E_j|}
		\left(\prod\limits_{j=1}^{m-1}|\mathcal E_j|\right)
		B^{\#}\left(
		\textbf{v}^{(1)},\cdots,\textbf{v}^{(m)}
		\right)\\
		&=
		\Lambda B^{\#}\left(
		\textbf{v}^{(1)},\cdots,\textbf{v}^{(m)}
		\right).
	\end{aligned}
	\label{shi-jiangcengguji}
\end{equation}
Monotonicity of $\left( t_{r} \right)_{r\in \mathbb{N}_{0}}$ and \eqref{shi-liangcengbizhi} imply
\begin{equation}
1\leqslant\Lambda
\leqslant\left(\frac{t_r}{t_{r-2}}\right)^{|J|}
\leqslant\left(\frac{t_r}{t_{r-2}}\right)^{m-1}
\leqslant\left(\frac{r+1}{r-1}\right)^{m-1}.
\label{shi-jiangcengxishu}
\end{equation}
Put
\begin{equation}
U\triangleq\bigcup\limits_{j=1}^{m-1}\mathrm{supp}\textbf{v}^{(j)},
\qquad |U|\leqslant(m-1)a_{r-2}.
\label{shi-guinazhichi}
\end{equation}
The space $P_UY$ has complex dimension
$|U|\geqslant1$ and carries the
norm inherited from $Y$. For $k\in F_m$, define
\begin{equation}
B_k(\textbf{z}^{(1)},\cdots,\textbf{z}^{(m-1)})
\triangleq
B\left(P_U\textbf{z}^{(1)},\cdots,
P_U\textbf{z}^{(m-1)},e_k\right).
\label{shi-gudingzuobiao}
\end{equation}
Each $B_k$ is continuous, symmetric, and depends only on the
coordinates in $U$. By \eqref{shi-touyingsuosuo}, its norm on $Y$
equals the norm of its restriction to $\left(P_UY\right)^{m-1}$. Since
$\|\textbf{v}^{(j)}\|_Y=1$ by \eqref{shi-yuanzifanshu}, the coefficient
formula \eqref{shi-zhixishuxingshi} and the induction hypothesis give
\begin{equation}
\begin{aligned}
B^{\#}(\textbf{v}^{(1)},\cdots,\textbf{v}^{(m)})
&=\frac{t_r}{a_r}\sum\limits_{k\in F_m}
B_k^{\#}(\textbf{v}^{(1)},\cdots,\textbf{v}^{(m-1)})\\
&\leqslant\frac{t_r}{a_r}\sum\limits_{k\in F_m}\|B_k^{\#}\|
\leqslant\frac{t_r}{a_r}K_{m-1}
\sum\limits_{k\in F_m}\|B_k\|.
\end{aligned}
\label{shi-diciguina}
\end{equation}

\proofstep{3}

Choose $\mathcal N\subset\overline{B_{P_UY}}$ as in
\namedref{Lemma}{yinli-youxianwang}. For each $k\in F_m$, choose
$\Theta(k)\in\mathcal N^{m-1}$ at which $|B_k|$ attains its maximum
on this finite set. Equation \eqref{shi-youxianwangguji} gives
\begin{equation}
\|B_k\|\leqslant\frac{2m}{m+1}|B_k(\Theta(k))|.
\label{shi-wangdianqufan}
\end{equation}
Enumerate the distinct values of $\Theta$ as
$\Theta_\nu=(\xi_\nu^{(1)},\cdots,\xi_\nu^{(m-1)})$,
$1\leqslant\nu\leqslant M$, and put
$D_\nu\triangleq\{k\in F_m:\Theta(k)=\Theta_\nu\}$.
These sets are nonempty and form a partition of $F_m$. By
\eqref{shi-guinazhichi} and \eqref{shi-youxianwangguji},
\begin{equation}
\begin{aligned}
M\leqslant|\mathcal N|^{m-1}
\leqslant(1+4m)^{2|U|(m-1)}\leqslant(1+4m)^{2(m-1)^2a_{r-2}},\qquad
\sum\limits_{\nu=1}^{M}|D_\nu|=a_r.
\end{aligned}
\label{shi-wangdianfenzu}
\end{equation}
For $k\in D_\nu$, choose
$$
\omega_k\triangleq
\begin{cases}
\frac{\overline{B_k(\Theta_\nu)}}{|B_k(\Theta_\nu)|},
&B_k(\Theta_\nu)\neq0,\\[2pt]
1,&B_k(\Theta_\nu)=0,
\end{cases}
\qquad
\textbf{w}_\nu\triangleq\sum\limits_{k\in D_\nu}\omega_ke_k.
$$
Each $\xi_\nu^{(j)}$ belongs to $P_UY$ and has norm at most $1$.
By \eqref{shi-xiangweifanshu},
$\|\textbf{w}_\nu\|_Y=\varphi(|D_\nu|)$. Hence
\eqref{shi-gudingzuobiao}, \eqref{shi-wangdianqufan}, and multilinearity give
\begin{equation}
\begin{aligned}
\sum\limits_{k\in D_\nu}\|B_k\|
&\leqslant\frac{2m}{m+1}
\sum\limits_{k\in D_\nu}|B_k(\Theta_\nu)|=\frac{2m}{m+1}
\left|B(\xi_\nu^{(1)},\cdots,\xi_\nu^{(m-1)},\textbf{w}_\nu)\right|\\& \leqslant\frac{2m}{m+1}\|B\|\varphi(|D_\nu|).
\end{aligned}
\label{shi-fenzuxiangwei}
\end{equation}
 \eqref{shi-canshuyujibenhanshu} gives $\varphi(n)\leqslant\frac{n}{t_{r-2}}$ for $n\geqslant a_{r-2}$.
For $1\leqslant n<a_{r-2}$, \eqref{shi-jiruoshoulian} and the triangle
inequality give $\varphi(n)\leqslant n<a_{r-2}$. Thus
\begin{equation}
	\varphi(n)\leqslant\frac{n}{t_{r-2}}+a_{r-2},
	\qquad n\in\mathbb N.
	\label{shi-jibenhanshuyuxiang}
\end{equation}
Summing \eqref{shi-fenzuxiangwei}  over $\nu=1,2,\cdots,M$ and using
\eqref{shi-wangdianfenzu} and \eqref{shi-jibenhanshuyuxiang}, we obtain
\begin{equation}
	\begin{aligned}
		\sum\limits_{k\in F_m}\|B_k\|
		&\leqslant\frac{2m}{m+1}\|B\|
		\sum\limits_{\nu=1}^{M}\varphi(|D_\nu|)\leqslant\frac{2m}{m+1}\|B\|
		\left(\frac{a_r}{t_{r-2}}+Ma_{r-2}\right).
	\end{aligned}
	\label{shi-fenzuzonghe}
\end{equation}
By \eqref{shi-cengjicanshu},
\begin{equation}
	a_{r-1}=2^{a_{r-2}+1},\qquad a_r=2^{\,2^{a_{r-2}+1}+1}.
	\label{shi-liangcengditui}
\end{equation}
Define 
\begin{equation}
	\varepsilon_{m,r}\triangleq
	\frac{t_ra_{r-2}}{a_r}(1+4m)^{2(m-1)^2a_{r-2}}
	=\frac{t_ra_{r-2}(1+4m)^{2(m-1)^2a_{r-2}}}
	{2^{\,2^{a_{r-2}+1}+1}}.
	\label{shi-cengjiwucha}
\end{equation}
Then \eqref{shi-wangdianfenzu} implies
\begin{equation}
	\begin{aligned}
		\frac{t_r}{a_r}\left(\frac{a_r}{t_{r-2}}+Ma_{r-2}\right)
		&=\frac{t_r}{t_{r-2}}+\frac{t_rMa_{r-2}}{a_r}\leqslant\frac{t_r}{t_{r-2}}+\varepsilon_{m,r}
		\leqslant\frac{r+1}{r-1}+\varepsilon_{m,r}.
	\end{aligned}
	\label{shi-cengjixishu}
\end{equation}
Combining \eqref{shi-jiangcengguji}, \eqref{shi-diciguina},
\eqref{shi-fenzuzonghe}, and \eqref{shi-cengjixishu}, we obtain
\begin{equation}
	\begin{aligned}
		B^{\#}(\textbf{u}^{(1)},\cdots,\textbf{u}^{(m)})
		&\leqslant\frac{2m}{m+1}\Lambda
		\frac{t_r}{a_r}\left(\frac{a_r}{t_{r-2}}+Ma_{r-2}\right)
		K_{m-1}\|B\|\\
		&\leqslant\frac{2m}{m+1}\Lambda
		\left(\frac{t_r}{t_{r-2}}+\varepsilon_{m,r}\right)K_{m-1}\|B\|\\
		&\leqslant\frac{2m}{m+1}\left(\frac{r+1}{r-1}\right)^{m-1}
		\left(\frac{r+1}{r-1}+\varepsilon_{m,r}\right)K_{m-1}\|B\|,
	\end{aligned}
	\label{shi-dacengguji}
\end{equation}
where the last inequality uses \eqref{shi-jiangcengxishu}.

\proofstep{4}

It is easy to see that
$$
r+1\leqslant 2r^2\leqslant a_{r-2},r=4,5,\cdots,\qquad1+4m\leqslant2^{m^2}, m=2,3,\cdots.
$$
Since $r\geqslant m^2$, we obtain
\begin{equation}
(1+4m)^{2(m-1)^2a_{r-2}}
\leqslant2^{2m^2(m-1)^2a_{r-2}}
\leqslant2^{2r^2a_{r-2}}
\leqslant2^{a_{r-2}^2}.
\label{shi-wangdianbijiao}
\end{equation}
$2^{a_{r-2}+1}\geqslant2a_{r-2}^2$ and
$2^{a_{r-2}^2+1}\geqslant2a_{r-2}^4\geqslant16a_{r-2}^2$.
Together with $t_r\leqslant r+1$,
\eqref{shi-cengjiwucha}, and \eqref{shi-wangdianbijiao}, this yields
\begin{equation}
\begin{aligned}
0\leqslant\varepsilon_{m,r}
&\leqslant(r+1)a_{r-2}\,2^{a_{r-2}^2-2^{a_{r-2}+1}-1}\leqslant\frac{a_{r-2}^2}{2^{a_{r-2}^2+1}}
\leqslant\frac{1}{16}.
\end{aligned}
\label{shi-wuchajiexian}
\end{equation}
It remains to compare the coefficient in \eqref{shi-dacengguji}
with $\frac{K_m}{K_{m-1}}$. If $m=2$, then $r\geqslant4$ and
$\frac{r+1}{r-1}\leqslant\frac{5}{3}$, so
\begin{equation}
\frac{2m}{m+1}\left(\frac{r+1}{r-1}\right)^{m-1}\left(\frac{r+1}{r-1}+\varepsilon_{m,r}\right)
\leqslant\frac{4}{3}\frac{5}{3}
\left(\frac{5}{3}+\frac{1}{16}\right)
=\frac{415}{108}<4.
\label{shi-erciguinabijiao}
\end{equation}
Suppose $m\geqslant3$. For $0\leqslant x<1$,
$$
e^x=\sum\limits_{j=0}^{\infty}\frac{x^j}{j!}
\leqslant\sum\limits_{j=0}^{\infty}x^j
=\frac{1}{1-x}.
$$
Together with $\log(1+x)\leqslant x$ for $x\geqslant0$ and
$r\geqslant m^2$, this gives
\begin{equation}
\begin{aligned}
\left(\frac{r+1}{r-1}\right)^{m-1}
&=\left(1+\frac{2}{r-1}\right)^{m-1}
\leqslant e^{\frac{2(m-1)}{r-1}}\leqslant e^{\frac{2}{m+1}}
\leqslant\frac{m+1}{m-1},\\
\frac{2m}{m+1}\left(\frac{r+1}{r-1}\right)^{m-1}
&\leqslant\frac{2m}{m-1}\leqslant3,
\qquad
\frac{r+1}{r-1}\leqslant\frac{m^2+1}{m^2-1}\leqslant\frac{5}{4}.
\end{aligned}
\label{shi-gaociguinabijiao}
\end{equation}
Equations \eqref{shi-wuchajiexian} and \eqref{shi-gaociguinabijiao} imply
$$
\frac{2m}{m+1}\left(\frac{r+1}{r-1}\right)^{m-1}\left(\frac{r+1}{r-1}+\varepsilon_{m,r}\right)
\leqslant3\left(\frac{5}{4}+\frac{1}{16}\right)
=\frac{63}{16}<4.
$$
Together with \eqref{shi-erciguinabijiao} and
\eqref{shi-changshuzengzhang}, this shows that the coefficient in
\eqref{shi-dacengguji} is at most $\frac{K_m}{K_{m-1}}$. Consequently,
$B^{\#}(\textbf{u}^{(1)},\cdots,\textbf{u}^{(m)})\leqslant K_m\|B\|$ when
$r\geqslant m^2$. Combining this with \eqref{shi-xiaocengguji} and
taking the supremum in \eqref{shi-yuanziqufan} proves
\eqref{shi-duoxianxingguji} in degree $m$, completing the induction.
\end{proof}

Define
\begin{equation}
C_0\triangleq1,\qquad C_1\triangleq1,\qquad
C_m\triangleq\frac{m^m}{m!}K_m
=\left(\frac{m^m}{m!}\right)^3
\binom{3m-1}{m}^{\frac{1}{2}}\theta_m^m
\quad(m\geqslant2).
\label{shi-duoxiangshichangshu}
\end{equation}

\begin{proposition}\label{mingti-duoxiangshiguji}
For every $m\in\mathbb{N}_0$, every $P\in\mathcal{P}^m(Y)$, and every
$\textbf{z}\in Y$, one has
\begin{equation}
\sum\limits_{\substack{\alpha\in\mathbb{N}_0^{(\mathbb{N})}\\|\alpha|=m}}
|c_\alpha(P)\textbf{z}^{\alpha}|
\leqslant C_m\|P\|_{\mathcal{P}^m(Y)}\|\textbf{z}\|_Y^m.
\label{shi-fanlixishuguji}
\end{equation}
If $P$ depends on finitely many coordinates, then
\begin{equation}
\|P^{\#}\|_{\mathcal{P}^m(Y)}\leqslant C_m\|P\|_{\mathcal{P}^m(Y)}.
\label{shi-youxianxishuguji}
\end{equation}
\end{proposition}

\begin{proof}
For $m=0$, the polynomial $P$ is constant and both assertions follow
from $P^{\#}=|P|$. Let $m\geqslant1$, suppose first that $P$ depends
on finitely many coordinates.
The polarization formula in \cite[Lemma~7.13\,(1), pp.~84--85]{KM97}
gives
\begin{equation}
B(\textbf{x}^{(1)},\cdots,\textbf{x}^{(m)})
\triangleq\frac{1}{2^m m!}
\sum\limits_{\epsilon\in\{-1,1\}^m}
\left(\prod\limits_{j=1}^{m}\epsilon_j\right)
P\left(\sum\limits_{j=1}^{m}\epsilon_j\textbf{x}^{(j)}\right),P\left( \textbf{x} \right) =B\left( \textbf{x} ,\cdots ,\textbf{x} \right)
\label{shi-jihuagongshi}
\end{equation}
For $\|\textbf{x}^{(j)}\|_Y\leqslant1,j=1,2,\cdots,m$,
\eqref{shi-jihuagongshi} gives
$$
\begin{aligned}
|B(\textbf{x}^{(1)},\cdots,\textbf{x}^{(m)})|
&\leqslant\frac{\|P\|_{\mathcal{P}^m(Y)}}{2^m m!}
\sum\limits_{\epsilon\in\{-1,1\}^m}
\left\|\sum\limits_{j=1}^{m}\epsilon_j\textbf{x}^{(j)}\right\|_Y^m\leqslant\frac{m^m}{m!}\|P\|_{\mathcal{P}^m(Y)}.
\end{aligned}
$$
Hence
\begin{equation}
\|B\|_{\mathcal{L}^m(Y)}
\leqslant\frac{m^m}{m!}\|P\|_{\mathcal{P}^m(Y)}.
\label{shi-jihuafanshu}
\end{equation}
By \eqref{shi-duijiaoxishu}, \eqref{shi-duoxianxingguji}, and \eqref{shi-jihuafanshu},
$$
\begin{aligned}
\|P^{\#}\|_{\mathcal{P}^m(Y)}
&\leqslant\|B^{\#}\|_{\mathcal{L}^m(Y)}
\leqslant K_m\|B\|_{\mathcal{L}^m(Y)}\leqslant\frac{m^m}{m!}K_m\|P\|_{\mathcal{P}^m(Y)}
=C_m\|P\|_{\mathcal{P}^m(Y)}.
\end{aligned}
$$
This proves \eqref{shi-youxianxishuguji}. If $P\in\mathcal{P}^m(Y)$ , by \eqref{shi-youxianxishuguji} and $\|P\circ P_n\|\leqslant\|P\|$, it holds that
$$
\begin{aligned}
&\sum\limits_{\substack{|\alpha|=m\\
\mathrm{supp}\alpha\subset\{1,\cdots,n\}}}
|c_\alpha(P)\textbf{z}^\alpha|=
\sum\limits_{|\alpha|=m}
|c_\alpha(P\circ P_n)\textbf{z}^\alpha|=\left( P\circ P_{n} \right)^{\#} \left( \left| \textbf{z} \right| \right)
\leqslant C_m\|P\|\|\textbf{z}\|_Y^m.
\end{aligned}
$$
Letting $n\to\infty$ proves \eqref{shi-fanlixishuguji} with the same
explicit coefficient $C_m$.
\end{proof}
\begin{remark}
The coefficients $C_m$ in \eqref{shi-duoxiangshichangshu} have at most
exponential growth. For $m\geqslant2$, $a_{m} > m^{2}-1$
gives $\kappa(m^2-1)\leqslant m-1$. By \eqref{shi-huanmancanshu},
\eqref{shi-huanmancanshujie}, and \eqref{shi-canshuyujibenhanshu},
$$
t_{m^2-1}\leqslant2t_{\kappa(m^2-1)}\leqslant2t_m
\leqslant\frac{4m}{\varphi(m)},
\qquad 1\leqslant \theta_{m} =1+\frac{\varphi \left( m \right)}{m} t_{m^{2}-1}\leqslant 1+\frac{\varphi \left( m \right)}{m} \frac{4m}{\varphi \left( m \right)} =5.
$$
Stirling's formula gives
\begin{equation}
	\begin{aligned}
		&\limsup\limits_{m\to \infty} C_{m}^{\frac{1}{m}}=\limsup\limits_{m\to \infty} \left[ \left( \frac{m^{m}}{m!} \right)^{\frac{3}{m}} \binom{3m-1}{m}^{\frac{1}{2m}} \theta_{m} \right]\\ &\  \leqslant 5\lim\limits_{m\to \infty} \left[ \left( \frac{m^{m}}{m!} \right)^{\frac{3}{m}} \binom{3m-1}{m}^{\frac{1}{2m}} \right] =\frac{15\sqrt{3}}{2} e^{3}\end{aligned}
	\label{shi-cishigenjie}
\end{equation}
\end{remark}

Combining \namedref{Proposition}{mingti-jishubiaoshi}, \namedref{Proposition}{mingti-jubuhua}, \namedref{Theorem}{dingli-shoulianpanju}, \namedref{Proposition}{mingti-duoxiangshiguji}, and \eqref{shi-cishigenjie}, we obtain one of the main results of this paper:

\begin{theorem}\label{dingli-fanliquanshoulian}
Let $Y$ be the complex Banach sequence space in \namedref{Definition}{dingyi-fanlikongjian-Y}.
Then

\begin{equation}
\mathrm{dom}H(O)=O
\qquad\text{for every Reinhardt domain }O\subset Y.
\label{shi-fanliquanyushoulian}
\end{equation}
For every $f\in H(O)$ and every $\textbf{z}\in O$,
\eqref{shi-juduishoulian} and \eqref{shi-hanshubiaoshi} hold.
In particular, $Y$ is a complex Banach sequence space
satisfying  \eqref{shi-fanliquanyushoulian}
which is not isomorphic to $\ell^1$.
\end{theorem}

\section[Exact monomial convergence domains in classical sequence spaces]
{Exact monomial convergence domains\texorpdfstring{\\}{ }in classical sequence spaces}
\label{jie-jingdiankongjian}

For $\textbf{z}\in c_0$, let $\textbf{z}^*=(z_j^*)_{j\geqslant1}$ be
the nonincreasing rearrangement of $(|z_j|)_{j\geqslant1}$.

For holomorphic functions on Reinhardt domains in $\ell^p$,
upper and lower inclusions for their monomial convergence
domains were obtained in
\cite[Theorem~4.7 and Example~4.9(a), pp.~32--33]{DMP09}.

\begin{theorem}\label{dingli-jingdianshoulian}
One has
\begin{equation}
\mathrm{dom}H(\ell^1)=\ell^1.
\label{shi-yikeheshoulian}
\end{equation}
For $1<p\leqslant2$,
\begin{equation}
\mathrm{dom}H(\ell^p)
=
\left\{\textbf{z}\in \ell^p:
\lim\limits_{n\to\infty}
\frac{\sum\limits_{j=1}^{n}z_j^*}
{(\log(n+1))^{1-\frac{1}{p}}}=0\right\}.
\label{shi-xiaozhishushoulian}
\end{equation}
For $2<p<\infty$,
\begin{equation}
\mathrm{dom}H(\ell^p)
=
\left\{\textbf{z}\in \ell^p:
\lim\limits_{n\to\infty}
\frac{\sum\limits_{j=1}^{n}(z_j^*)^{\frac{2p}{p+2}}}
{(\log(n+1))^{\frac{p}{p+2}}}=0\right\}.
\label{shi-dazhishushoulian}
\end{equation}
Moreover,
\begin{equation}
\mathrm{dom}H(c_0)
=
\left\{\textbf{z}\in c_0:
\lim\limits_{n\to\infty}
\frac{\sum\limits_{j=1}^{n}(z_j^*)^2}{\log(n+1)}=0\right\}.
\label{shi-lingjixianshoulian}
\end{equation}
\end{theorem}
\subsection{Auxiliary estimates for monomial convergence}

For homogeneous polynomials we use the norm
\eqref{shi-duoxiangshifanshu}, written as $\|G\|$
when its ambient space is clear.
For the auxiliary results below, fix $1<p\leqslant\infty$, take
$X=\ell^p$ when $p<\infty$ and $X=c_0$ when $p=\infty$, and put
\begin{equation}
\sigma\triangleq1-\frac{1}{\min\{p,2\}},
\qquad
\frac{1}{q}\triangleq\sigma+\frac{1}{p},
\qquad \frac{1}{\infty}=0.
\label{shi-fenleizhishu}
\end{equation}
Thus
$$
q=
\begin{cases}
1,&1<p\leqslant2,\\
\frac{2p}{p+2},&2<p<\infty,\\
2,&p=\infty.
\end{cases}
$$
For $\textbf{z}\in c_0$, define the possibly infinite quantity
\begin{equation}
M(\textbf{z})\triangleq
\sup\limits_{n\in\mathbb{N}}
\frac{\left(\sum\limits_{j=1}^{n}(z_j^*)^q\right)^{\frac{1}{q}}}
     {(\log(n+1))^\sigma}.
\label{shi-chongpaoliang}
\end{equation}
For $\textbf{z}\in c_0$ and $n\in\mathbb{N}_0$, we use the same
coordinate formulas as in \eqref{shi-shouweitouying},
$$
P_n\textbf{z}\triangleq\sum\limits_{j=1}^{n}z_je_j,
\qquad
Q_n\textbf{z}\triangleq\textbf{z}-P_n\textbf{z},
$$
with $P_0=0$ and $Q_0=I_{c_0}$.

\begin{lemma}\label{yinli-chongpaixingzhi}
    For every $\textbf{z}\in c_0$ with $M(\textbf{z})<\infty$, one has
    $\textbf{z}\in X$ and
    \begin{equation}
        \|\textbf{z}\|_X\leqslant C M(\textbf{z}),
        \label{shi-chongpaiqianru}
    \end{equation}
    where $C\geqslant1$ depends only on $p$.
    Moreover, whenever $\textbf{z}\in c_0$ and $M(\textbf{z})<\infty$,
\begin{equation}
        M(\lambda\textbf{z})=|\lambda|M(\textbf{z}),\qquad
    M\bigl((z_{\pi(j)})_{j\in\mathbb{N}}\bigr)=M(\textbf{z}),\qquad
    M\bigl((\varepsilon_jz_j)_{j\in\mathbb{N}}\bigr)=M(\textbf{z}),\label{shi-chongpaiduichen}
\end{equation}
    for every $\lambda\in\mathbb{C}$, every bijection
    $\pi:\mathbb{N}\to\mathbb{N}$, and every
    $(\varepsilon_j)_{j\in\mathbb{N}}\subset\mathbb{C}$ with
    $|\varepsilon_j|=1$ for all $j\in\mathbb{N}$. If
    $\textbf{w}\in c_0$ and $|\textbf{w}|\leqslant|\textbf{z}|$, then
    $M(\textbf{w})\leqslant M(\textbf{z})$.

    Furthermore, for every $\textbf{z}\in c_0$,
    \begin{equation}
        \lim\limits_{n\to\infty}
        \frac{\left(\sum\limits_{j=1}^{n}(z_j^*)^q\right)^{\frac{1}{q}}}
        {(\log(n+1))^\sigma}=0
        \quad\Longleftrightarrow\quad
        \lim\limits_{N\to\infty}M(Q_N\textbf{z})=0.
        \label{shi-weibudengjia}
    \end{equation}
    Finally, for every $\textbf{z}\in c_0$, $N\in\mathbb{N}_0$, and
    $n\in\mathbb{N}$,
    \begin{equation}
        0\leqslant
        \left(\sum\limits_{j=1}^{n}(z_j^*)^q\right)^{\frac{1}{q}}
        -\left(\sum\limits_{j=1}^{n}
        \bigl((Q_N\textbf{z})_j^*\bigr)^q\right)^{\frac{1}{q}}
        \leqslant\|P_N\textbf{z}\|_{\ell^q}.
        \label{shi-shoubumodong}
    \end{equation}
\end{lemma}

\begin{proof}
    For every $\textbf{z}\in c_0$ and $n\in\mathbb{N}$,
    \begin{equation}
        \left(\sum\limits_{j=1}^{n}(z_j^*)^q\right)^{\frac{1}{q}}
        =
        \sup\limits_{\substack{F\subset\mathbb{N}\\|F|=n}}
        \left(\sum\limits_{j\in F}|z_j|^q\right)^{\frac{1}{q}}.
        \label{shi-chongpaishangquejie}
    \end{equation}
    Hence \eqref{shi-chongpaishangquejie} yields \eqref{shi-chongpaiduichen}
    and
$
    |\textbf{w}|\leqslant|\textbf{z}|
    \Longrightarrow
    M(\textbf{w})\leqslant M(\textbf{z}).
$
    In particular,
    \begin{equation}
        M(P_N\textbf{z}),\ M(Q_N\textbf{z})
        \leqslant M(\textbf{z}).
        \label{shi-chongpaitouying}
    \end{equation}

    If $M(\textbf{z})<\infty$, then, for every $j\in\mathbb{N}$,
    $$
    j^{\frac{1}{q}}z_j^*
    \leqslant
    \left(\sum\limits_{k=1}^{j}(z_k^*)^q\right)^{\frac{1}{q}}
    \leqslant
    M(\textbf{z})(\log(j+1))^\sigma,
    $$
    and therefore
    \begin{equation}
        z_j^*\leqslant
        M(\textbf{z})
        \frac{(\log(j+1))^\sigma}{j^{\frac{1}{q}}}.
        \label{shi-chongpaidianjie}
    \end{equation}
    If $p<\infty$, then \eqref{shi-chongpaidianjie}
    and \eqref{shi-fenleizhishu} give
    $$
    \|\textbf{z}\|_X^p
    =
    \sum\limits_{j=1}^{\infty}(z_j^*)^p
    \leqslant
    M(\textbf{z})^p
    \sum\limits_{j=1}^{\infty}
    \frac{(\log(j+1))^{\sigma p}}{j^{\frac{p}{q}}}
    =
    M(\textbf{z})^p
    \sum\limits_{j=1}^{\infty}
    \frac{(\log(j+1))^{\sigma p}}{j^{1+\sigma p}}
    <\infty.
    $$
    Consequently,
    $$
    \|\textbf{z}\|_X
    \leqslant
    \left(
    \sum\limits_{j=1}^{\infty}
    \frac{(\log(j+1))^{\sigma p}}{j^{1+\sigma p}}
    \right)^{\frac{1}{p}}
    M(\textbf{z}).
    $$
    If $p=\infty$, then \eqref{shi-chongpaidianjie}
    gives
    $$
    \|\textbf{z}\|_X
    =z_1^*
    \leqslant
    (\log2)^\sigma M(\textbf{z})
    \leqslant M(\textbf{z}).
    $$
    Thus \eqref{shi-chongpaiqianru} follows.

    Let $N\in\mathbb{N}_0$ and $n\in\mathbb{N}$. For every
    $F\subset\mathbb{N}$ with $|F|=n$, Minkowski's inequality gives
    $$
    \left(\sum\limits_{j\in F}|z_j|^q\right)^{\frac{1}{q}}
    \leqslant
    \left(\sum\limits_{j\in F}|(P_N\textbf{z})_j|^q\right)^{\frac{1}{q}}
    +
    \left(\sum\limits_{j\in F}|(Q_N\textbf{z})_j|^q\right)^{\frac{1}{q}}
    $$
    and hence
    $$
    \left(\sum\limits_{j\in F}|z_j|^q\right)^{\frac{1}{q}}
    \leqslant
    \|P_N\textbf{z}\|_{\ell^q}
    +
    \left(\sum\limits_{j\in F}|(Q_N\textbf{z})_j|^q\right)^{\frac{1}{q}}.
    $$
    Taking the supremum over all such $F$ and using
    \eqref{shi-chongpaishangquejie}, we obtain
    $$
    \left(\sum\limits_{j=1}^{n}(z_j^*)^q\right)^{\frac{1}{q}}
    \leqslant
    \|P_N\textbf{z}\|_{\ell^q}
    +
    \left(\sum\limits_{j=1}^{n}
    \bigl((Q_N\textbf{z})_j^*\bigr)^q\right)^{\frac{1}{q}}.
    $$
    Moreover,
$
    |Q_N\textbf{z}|\leqslant|\textbf{z}|,
    $
    so \eqref{shi-chongpaishangquejie} gives
    $$
    \left(\sum\limits_{j=1}^{n}
    \bigl((Q_N\textbf{z})_j^*\bigr)^q\right)^{\frac{1}{q}}
    \leqslant
    \left(\sum\limits_{j=1}^{n}(z_j^*)^q\right)^{\frac{1}{q}}.
    $$
    Thus \eqref{shi-shoubumodong} follows.

    Suppose
$
    \lim\limits_{n\to\infty}
    \frac{\left(\sum\limits_{j=1}^{n}(z_j^*)^q\right)^{\frac{1}{q}}}
    {(\log(n+1))^\sigma}
    =0.
$
    For every $\varepsilon>0$, there exists $n_0\geqslant2,n_0\in \mathbb{N}$ such that
    $$
    \frac{\left(\sum\limits_{j=1}^{n}(z_j^*)^q\right)^{\frac{1}{q}}}
    {(\log(n+1))^\sigma}
    \leqslant\varepsilon,
    \qquad n\geqslant n_0.
    $$
    Since $|Q_N\textbf{z}|\leqslant|\textbf{z}|$,
    \eqref{shi-chongpaishangquejie} gives
    $$
    \frac{\left(\sum\limits_{j=1}^{n}
        \bigl((Q_N\textbf{z})_j^*\bigr)^q\right)^{\frac{1}{q}}}
    {(\log(n+1))^\sigma}
    \leqslant\varepsilon,
    \qquad n\geqslant n_0.
    $$
    For $1\leqslant n<n_0$,
    $$
    \frac{\left(\sum\limits_{j=1}^{n}
        \bigl((Q_N\textbf{z})_j^*\bigr)^q\right)^{\frac{1}{q}}}
    {(\log(n+1))^\sigma}
    \leqslant
    \frac{n^{\frac{1}{q}}}{(\log(n+1))^\sigma}
    \|Q_N\textbf{z}\|_{\ell^\infty}
    \leqslant
    \frac{n_0^{\frac{1}{q}}}{(\log2)^\sigma}
    \|Q_N\textbf{z}\|_{\ell^\infty}.
    $$
    Therefore
    $$
    M(Q_N\textbf{z})
    \leqslant
    \max\left\{
    \varepsilon,
    \frac{n_0^{\frac{1}{q}}}{(\log2)^\sigma}
    \|Q_N\textbf{z}\|_{\ell^\infty}
    \right\}.
    $$
    Since $\textbf{z}\in c_0$,
$
    \lim\limits_{N\to\infty}
    \|Q_N\textbf{z}\|_{\ell^\infty}=0,
$
    and hence
$
    \limsup\limits_{N\to\infty}M(Q_N\textbf{z})
    \leqslant\varepsilon.
$
    Letting $\varepsilon\to0$ gives
$
    \lim\limits_{N\to\infty}M(Q_N\textbf{z})=0.
$

Conversely, \eqref{shi-shoubumodong} gives, for every
    $N\in\mathbb{N}_0$ and $n\in\mathbb{N}$,
    $$
    \frac{\left(\sum\limits_{j=1}^{n}(z_j^*)^q\right)^{\frac{1}{q}}}
    {(\log(n+1))^\sigma}
    \leqslant
    \frac{\|P_N\textbf{z}\|_{\ell^q}}{(\log(n+1))^\sigma}
    +
    \frac{\left(\sum\limits_{j=1}^{n}
        \bigl((Q_N\textbf{z})_j^*\bigr)^q\right)^{\frac{1}{q}}}
    {(\log(n+1))^\sigma}.
    $$
    By \eqref{shi-chongpaoliang},
    $$
    \frac{\left(\sum\limits_{j=1}^{n}(z_j^*)^q\right)^{\frac{1}{q}}}
    {(\log(n+1))^\sigma}
    \leqslant
    \frac{\|P_N\textbf{z}\|_{\ell^q}}{(\log(n+1))^\sigma}
    +
    M(Q_N\textbf{z}).
    $$
Therefore
    $$
    \limsup\limits_{n\to\infty}
    \frac{\left(\sum\limits_{j=1}^{n}(z_j^*)^q\right)^{\frac{1}{q}}}
    {(\log(n+1))^\sigma}
    \leqslant M(Q_N\textbf{z}).
    $$
    Thus
$
    \lim\limits_{N\to\infty}M(Q_N\textbf{z})=0
$
    implies
    $$
    \lim\limits_{n\to\infty}
    \frac{\left(\sum\limits_{j=1}^{n}(z_j^*)^q\right)^{\frac{1}{q}}}
    {(\log(n+1))^\sigma}=0.
    $$
    This proves \eqref{shi-weibudengjia}.
\end{proof}

\begin{lemma}\label{yinli-fenkuaiqiuhe}
For $X,p,q,\sigma$ as in \eqref{shi-fenleizhishu}, let
$C_0\geqslant1$ satisfy
\begin{equation}
\begin{aligned}
2^{\frac{1}{q}}(\log2)^\sigma&\leqslant C_0,
& e^2 2^{4\sigma+\frac{1}{p}}&\leqslant C_0,\\
u^\sigma(1+u)^{1-\frac{1}{q}}2^{-\frac{\sigma u}{2}}
&\leqslant C_0\quad(u\geqslant0),
&2^\sigma\max\left\{1,
\frac{2^{-\frac{\sigma}{2}}}{1-2^{-\frac{\sigma}{2}}}\right\}
&\leqslant C_0.
\end{aligned}
\label{shi-fuzhuchangshu}
\end{equation}
Let $1\leqslant k_0\leqslant K$ be integers and let
$\omega_{k_0},\cdots,\omega_K\geqslant0$ satisfy
\begin{equation}
\left(\sum\limits_{k=k_0}^{L}\omega_k^q\right)^{\frac{1}{q}}
\leqslant C_0L^\sigma,
\qquad k_0\leqslant L\leqslant K.
\label{shi-kuaibufenhe}
\end{equation}
Then, for every $d\in\mathbb{N}$ and $k_0\leqslant L\leqslant K$,
\begin{equation}
d^{\frac{d}{p}}
\sum\limits_{\substack{\beta_{k_0},\cdots,\beta_L\in\mathbb{N}_0\\
\beta_{k_0}+\cdots+\beta_L=d}}
\prod\limits_{k=k_0}^{L}\frac{\omega_k^{\beta_k}}{(\beta_k!)^{\frac{1}{q}}}
\leqslant(3C_0)^d\left(\frac{L}{d}\right)^{\sigma d}
\left(1+\frac{L}{d}\right)^{(1-\frac{1}{q})d}.
\label{shi-duocishuqiuhe}
\end{equation}
Moreover, for every $d\in\mathbb{N}$,
\begin{equation}
d^\sigma\sum\limits_{L=1}^{\infty}2^{-\sigma L}
\left(\frac{L}{d}\right)^{\sigma d}
\left(1+\frac{L}{d}\right)^{(1-\frac{1}{q})d}
\leqslant C_0^{2d}.
\label{shi-zhishishuaijian}
\end{equation}
\end{lemma}

\begin{proof}
    Fix $d\in\mathbb{N}$ and $k_0\leqslant L\leqslant K$. By \eqref{shi-fenleizhishu}, $1\leqslant q\leqslant2$ and $\frac{1}{p}-\frac{1}{q}=-\sigma<0$. We have
    \begin{equation}
        \#\left\{
        (\beta_{k_0},\cdots,\beta_L)\in\mathbb{N}_0^{L-k_0+1}:
        \sum\limits_{k=k_0}^{L}\beta_k=d
        \right\}
        =\binom{L-k_0+d}{d}\leqslant\binom{L+d}{d}.
        \label{shi-duocishujishu}
    \end{equation}
    By \eqref{shi-duocishujishu}, H\"older's inequality, and the multinomial theorem,
    \begin{equation}
        \begin{aligned}
            \sum\limits_{\substack{\beta_{k_0},\cdots,\beta_L\in\mathbb{N}_0\\
                    \beta_{k_0}+\cdots+\beta_L=d}}
            \prod\limits_{k=k_0}^{L}\frac{\omega_k^{\beta_k}}{(\beta_k!)^{\frac{1}{q}}}
            &\quad\leqslant
            \binom{L+d}{d}^{1-\frac{1}{q}}
            \left(
            \sum\limits_{\substack{\beta_{k_0},\cdots,\beta_L\in\mathbb{N}_0\\
                    \beta_{k_0}+\cdots+\beta_L=d}}
            \prod\limits_{k=k_0}^{L}\frac{\omega_k^{q\beta_k}}{\beta_k!}
            \right)^{\frac{1}{q}}\\
            &\quad=
            \binom{L+d}{d}^{1-\frac{1}{q}}
            \frac{\left(\sum\limits_{k=k_0}^{L}\omega_k^q\right)^{\frac{d}{q}}}
            {(d!)^{\frac{1}{q}}}.
        \end{aligned}
        \label{shi-duochongheguji}
    \end{equation}
    By \begin{equation}
  \begin{aligned}d!&=e^{\sum\limits_{k=1}^{d} \log k}\geqslant e^{\sum\limits_{k=2}^{d} \int\limits_{k-1}^{k} \log x\mathrm{d} x}=e^{d\log d-d+1}\geqslant e^{d\log d-\log 3\cdot d}=\left( \frac{d}{3} \right)^{d} ,\end{aligned}
    \end{equation}we have
    \begin{equation}
\binom{L+d}{d}\leqslant\frac{(L+d)^d}{d!}
\leqslant\left(3\left(1+\frac{L}{d}\right)\right)^d.
        \
        \label{shi-fenkuaijicheng}
    \end{equation}
    By \eqref{shi-duochongheguji}, \eqref{shi-fenkuaijicheng}, \eqref{shi-kuaibufenhe}, and \eqref{shi-fenleizhishu},
    $$
    \begin{aligned}
        &d^{\frac{d}{p}}
        \sum\limits_{\substack{\beta_{k_0},\cdots,\beta_L\in\mathbb{N}_0\\
                \beta_{k_0}+\cdots+\beta_L=d}}
        \prod\limits_{k=k_0}^{L}\frac{\omega_k^{\beta_k}}{(\beta_k!)^{\frac{1}{q}}}\leqslant
        d^{\frac{d}{p}}
        \left(3\left(1+\frac{L}{d}\right)\right)^{(1-\frac{1}{q})d}
        (C_0L^\sigma)^d\left(\frac{3}{d}\right)^{\frac{d}{q}}\\
        &\quad=
        (3C_0)^dL^{\sigma d}d^{(\frac{1}{p}-\frac{1}{q})d}
        \left(1+\frac{L}{d}\right)^{(1-\frac{1}{q})d}
        =
        (3C_0)^d\left(\frac{L}{d}\right)^{\sigma d}
        \left(1+\frac{L}{d}\right)^{(1-\frac{1}{q})d}.
    \end{aligned}
    $$
    This proves \eqref{shi-duocishuqiuhe}.

    Finally, $d\leqslant2^d$ and \eqref{shi-fuzhuchangshu} give
    $$
    \begin{aligned}
        &d^\sigma\sum\limits_{L=1}^{\infty}
        2^{-\sigma L}
        \left(\frac{L}{d}\right)^{\sigma d}
        \left(1+\frac{L}{d}\right)^{(1-\frac{1}{q})d}=
        d^\sigma\sum\limits_{L=1}^{\infty}2^{-\frac{\sigma L}{2}}
        \left[
        \left(\frac{L}{d}\right)^\sigma
        \left(1+\frac{L}{d}\right)^{1-\frac{1}{q}}
        2^{-\frac{\sigma L}{2d}}
        \right]^d\\
        &\quad\leqslant
        (2^\sigma C_0)^d
        \sum\limits_{L=1}^{\infty}2^{-\frac{\sigma L}{2}}=
        (2^\sigma C_0)^d
        \frac{2^{-\frac{\sigma}{2}}}{1-2^{-\frac{\sigma}{2}}}
\leqslant
        C_0^d\left[
        2^\sigma\max\left\{1,
        \frac{2^{-\frac{\sigma}{2}}}{1-2^{-\frac{\sigma}{2}}}
        \right\}
        \right]^d
        \leqslant C_0^{2d}.
    \end{aligned}
    $$
    This proves \eqref{shi-zhishishuaijian}.
\end{proof}
\begin{lemma}\label{yinli-chongpaixishuguji}
	For $X,p,q,\sigma$ as in \eqref{shi-fenleizhishu}, there exists
	a constant $C\geqslant1$, depending only on $p$, such that, for every
	$m\in\mathbb{N}$, every $G\in\mathcal{P}^m(X)$, and every
	$\textbf{z}\in c_0$ with $M(\textbf{z})<\infty$,
	\begin{equation}
		\sum\limits_{\substack{\alpha\in\mathbb{N}_0^{(\mathbb{N})}\\
				|\alpha|=m}}
		|c_\alpha(G)\textbf{z}^\alpha|
		\leqslant
		C^m M(\textbf{z})^m\|G\|.
		\label{shi-chongpaixishuguji}
	\end{equation}
\end{lemma}
\begin{proof}
    Fix $p$ and choose $C_0\geqslant1$, depending only on $p$, such that \eqref{shi-fuzhuchangshu} holds and \eqref{shi-chongpaiqianru} holds with $C=C_0$. We first assume $G=G\circ P_n$. 

    \proofstep{1}
For $m\geqslant2$ and $\emptyset\neq A\subset\{\alpha\in\mathbb{N}_0^{(\mathbb{N})}:|\alpha|=m,\ \mathrm{supp}\alpha\subset\{1,\cdots,n\}\}$, put
    \begin{equation}
        A^-\triangleq\{\alpha-e_{\max\mathrm{supp}\alpha}:\alpha\in A\}.
        \label{shi-shanjianzhiji}
    \end{equation}
    By \eqref{shi-duoxiangshifanshu} and \eqref{shi-shouweisuosuo}, $\|G\|=\sup\limits_{\|\textbf{u}\|_{\ell^p}\leqslant1}|G(\sum\limits_{j=1}^{n}u_je_j)|$, where $\textbf{u}\in\mathbb{C}^n$. Moreover,
    \begin{equation}
        \begin{gathered}
            \alpha=\sum\limits_{\nu=1}^{m}e_{j_\nu}
            \quad\Longrightarrow\quad
            \alpha-e_{\max\mathrm{supp}\alpha}=\sum\limits_{\nu=1}^{m-1}e_{j_\nu},
            \qquad 1\leqslant j_1\leqslant\cdots\leqslant j_m\leqslant n,\\
            em\max\left\{e^{\frac{m-1}{p}},2^{\frac{m-1}{2}}\right\}
            \leqslant me^m\leqslant e^{2m}.
        \end{gathered}
        \label{shi-zhijiduiying}
    \end{equation}
    Thus \cite[Theorem~3.2]{BDS19} and \eqref{shi-zhijiduiying} give
    \begin{equation}
        \sum\limits_{\alpha\in A}|c_\alpha(G)\textbf{u}^\alpha|
        \leqslant e^{2m}|A^-|^\sigma\|\textbf{u}\|_X^m\|G\|,
        \qquad \textbf{u}\in X.
        \label{shi-zhijixishuguji}
    \end{equation}

    Take $\pi\in S_n$ with $|z_{\pi(1)}|\geqslant\cdots\geqslant|z_{\pi(n)}|$ and define
    \begin{equation}
        \begin{aligned}
            \varepsilon_j&\triangleq
            \begin{cases}
                \displaystyle\frac{\overline{z_{\pi(j)}}}{|z_{\pi(j)}|},&z_{\pi(j)}\neq0,\\
                1,&z_{\pi(j)}=0,
            \end{cases}
            \qquad 1\leqslant j\leqslant n,\\
            T\textbf{x}&\triangleq\sum\limits_{j=1}^{n}\varepsilon_jx_{\pi(j)}e_j+Q_n\textbf{x},
            \qquad \textbf{x}\in X.
        \end{aligned}
        \label{shi-chongpaibianhuan}
    \end{equation}
    Then \eqref{shi-chongpaibianhuan} yields
    \begin{equation}
        \begin{gathered}
            T^{-1}\textbf{x}=\sum\limits_{j=1}^{n}\overline{\varepsilon_j}x_je_{\pi(j)}+Q_n\textbf{x},
            \qquad \|T\textbf{x}\|_X=\|\textbf{x}\|_X,\\
            TP_n=P_nT,
            \qquad TP_n\textbf{z}=\sum\limits_{j=1}^{n}|z_{\pi(j)}|e_j=(P_n\textbf{z})^*.
        \end{gathered}
        \label{shi-chongpaidengju}
    \end{equation}
    By \eqref{shi-chongpaidengju}, \eqref{shi-youxianzhankai}, \eqref{shi-chongpaiduichen}, and \eqref{shi-chongpaitouying},
    \begin{equation}
        \begin{gathered}
            \|G\circ T^{-1}\|=\|G\|,
            \qquad M(TP_n\textbf{z})=M(P_n\textbf{z})\leqslant M(\textbf{z}),\\
            \sum\limits_{|\alpha|=m}|c_\alpha(G)\textbf{z}^\alpha|
            =\sum\limits_{|\alpha|=m}|c_\alpha(G\circ T^{-1})(TP_n\textbf{z})^\alpha|.
        \end{gathered}
        \label{shi-xishuhebudong}
    \end{equation}
For proving \eqref{shi-chongpaixishuguji}, the cases
$G=0$ and $P_n\textbf{z}=\textbf{0}$ are immediate.
By \eqref{shi-xishuhebudong} and homogeneity,
it therefore suffices to assume
    \begin{equation}
        \|G\|=M(\textbf{z})=1,
        \qquad z_1\geqslant\cdots\geqslant z_n\geqslant0,
        \qquad z_j=0\quad(j>n).
        \label{shi-jingdianguifan}
    \end{equation}
    For $m=1$, \eqref{shi-youxianzhankai} and \eqref{shi-chongpaiqianru} give
    \begin{equation}
        \sum\limits_{j=1}^{n}|c_{e_j}(G)|z_j
        =\sup\limits_{|w_j|\leqslant z_j\ (j\in\mathbb{N})}|G(\textbf{w})|
        \leqslant\|G\|\|\textbf{z}\|_X\leqslant C_0.
        \label{shi-yicixishuguji}
    \end{equation}

    \proofstep{2}
Henceforth $m\geqslant2$. Put
    \begin{gather}
        k_0\triangleq\left\lceil\frac{\log m}{\log2}\right\rceil,
        \quad I_0\triangleq\{1,\cdots,2^{k_0}-1\},
        \quad I_k\triangleq\{2^k,\cdots,2^{k+1}-1\}\quad(k\geqslant k_0).
        \notag\\
        1\leqslant|I_0|<2m,
        \qquad |I_k|=2^k\geqslant m\quad(k\geqslant k_0).
        \label{shi-shuangjinfenkuai}
    \end{gather}
    Increasing $n$ if necessary, take $n=2^{K+1}-1$ with $K\geqslant k_0$. By \eqref{shi-fenleizhishu} and \eqref{shi-jingdianguifan},
    \begin{equation}
        \omega_k\triangleq2^{\sigma k}\|P_{I_k}\textbf{z}\|_X
        \leqslant2^{k(\sigma+\frac{1}{p})}z_{2^k}
        =2^{\frac{k}{q}}z_{2^k},
        \qquad k_0\leqslant k\leqslant K.
        \label{shi-kuaiquanzhong}
    \end{equation}
    For $k_0\leqslant L\leqslant K$, \eqref{shi-jingdianguifan}, \eqref{shi-kuaiquanzhong} and \eqref{shi-chongpaoliang} yield
    \begin{equation}
        \begin{aligned}
            \sum\limits_{k=k_0}^{L}\omega_k^q
            &\leqslant\sum\limits_{k=k_0}^{L}2^kz_{2^k}^q
            \leqslant2\sum\limits_{k=k_0}^{L}\sum\limits_{j=2^{k-1}}^{2^k-1}z_j^q\\
            &=2\sum\limits_{j=2^{k_0-1}}^{2^L-1}z_j^q
            \leqslant2\sum\limits_{j=1}^{2^L-1}z_j^q
            \leqslant2(L\log2)^{\sigma q}.
        \end{aligned}
        \label{shi-quanzhongbufenhe}
    \end{equation}
    By \eqref{shi-chongpaiqianru}, \eqref{shi-quanzhongbufenhe}, and \eqref{shi-fuzhuchangshu},
    \begin{equation}
        \|P_{I_0}\textbf{z}\|_X\leqslant C_0,
        \qquad
        \left(\sum\limits_{k=k_0}^{L}\omega_k^q\right)^{\frac{1}{q}}
        \leqslant2^{\frac{1}{q}}(\log2)^\sigma L^\sigma
        \leqslant C_0L^\sigma.
        \label{shi-kuaitongyijie}
    \end{equation}

    \proofstep{3}
For $\beta=(\beta_{k_0},\cdots,\beta_K)\in\mathbb{N}_0^{K-k_0+1}$, write $|\beta|\triangleq\sum\limits_{k=k_0}^{K}\beta_k$. For $h\in\mathbb{N}_0$ with $h+|\beta|=m$, set
    \begin{equation}
        A_{h,\beta}\triangleq\left\{\alpha\in\mathbb{N}_0^{(\mathbb{N})}:
        \begin{gathered}
            |\alpha|=m,\quad\mathrm{supp}\alpha\subset\{1,\cdots,n\},\\
            \sum\limits_{j\in I_0}\alpha_j=h,\quad
            \sum\limits_{j\in I_k}\alpha_j=\beta_k\quad(k_0\leqslant k\leqslant K)
        \end{gathered}\right\}.
        \label{shi-fenkuaizhiji}
    \end{equation}
    By \eqref{shi-youxianzhankai} and \eqref{shi-fenkuaizhiji},
    \begin{equation}
        \sum\limits_{|\alpha|=m}|c_\alpha(G)\textbf{z}^\alpha|
        =\sum\limits_{h+|\beta|=m}\sum\limits_{\alpha\in A_{h,\beta}}
        |c_\alpha(G)\textbf{z}^\alpha|.
        \label{shi-xishuhefenjie}
    \end{equation}
    For $h<m$, put
    $$
        d\triangleq m-h=|\beta|\geqslant1,
        \qquad L\triangleq\max\{k:\beta_k>0\},
        \qquad t\triangleq\max I_L.
    $$
    For $\eta\in\mathbb{N}_0^{(\mathbb{N})}$, \eqref{shi-shanjianzhiji} and \eqref{shi-fenkuaizhiji} give
    \begin{equation}
        \begin{gathered}
            \alpha\in A_{h,\beta},\quad j=\max\mathrm{supp}\alpha
            \quad\Longrightarrow\quad j\in I_L,\quad\alpha-e_j+e_t\in A_{h,\beta},\\
            \eta+e_t\in A_{h,\beta}
            \quad\Longrightarrow\quad\max\mathrm{supp}(\eta+e_t)=t,\quad
            \eta=(\eta+e_t)-e_t\in A_{h,\beta}^-.
        \end{gathered}
        \label{shi-zuidazhibiao}
    \end{equation}
    Consequently, \eqref{shi-zuidazhibiao} gives
    \begin{equation}
        \begin{aligned}
            \eta\in A_{h,\beta}^-
            &\Longleftrightarrow\eta+e_t\in A_{h,\beta}\\
            &\Longleftrightarrow
            \begin{cases}
                \mathrm{supp}\eta\subset\{1,\cdots,t\},\\
                \sum\limits_{j\in I_0}\eta_j=h,\quad
                \sum\limits_{j\in I_L}\eta_j=\beta_L-1,\\
                \sum\limits_{j\in I_k}\eta_j=\beta_k\quad(k_0\leqslant k<L).
            \end{cases}
        \end{aligned}
        \label{shi-shanjianjikehua}
    \end{equation}
   By \eqref{shi-shanjianjikehua}, the coordinates of $\eta$ in the different blocks may be chosen independently. The numbers of nonnegative integer solutions on $I_0$, $I_k$ $(k_0\leqslant k<L)$, and $I_L$ are respectively
   $$
   \binom{|I_0|+h-1}{h},\qquad
   \binom{2^k+\beta_k-1}{\beta_k},\qquad
   \binom{2^L+\beta_L-2}{\beta_L-1}.
   $$
   Therefore,
   \begin{equation}
   	|A_{h,\beta}^-|
   	=\binom{|I_0|+h-1}{h}
   	\binom{2^L+\beta_L-2}{\beta_L-1}
   	\prod\limits_{k=k_0}^{L-1}
   	\binom{2^k+\beta_k-1}{\beta_k}.
   	\label{shi-shanjianjijishu}
   \end{equation}
    By \eqref{shi-shuangjinfenkuai},
    \begin{equation}
        \begin{aligned}
            \binom{|I_0|+h-1}{h}
            &\leqslant2^{|I_0|+h-1}\leqslant2^{3m},
            &&0\leqslant h\leqslant m,\\
            \binom{2^k+v-1}{v}
            &=\frac{\prod\limits_{\nu=0}^{v-1}(2^k+\nu)}{v!}
            \leqslant\frac{2^{(k+1)v}}{v!},
            &&k\geqslant k_0,\quad 0\leqslant v\leqslant m.
        \end{aligned}
        \label{shi-kuaizuhejie}
    \end{equation}
    Combining \eqref{shi-shanjianjijishu} and \eqref{shi-kuaizuhejie}, we obtain
    \begin{equation}
        \begin{aligned}
            |A_{h,\beta}^-|
            &\leqslant2^{3m}\frac{2^{(L+1)(\beta_L-1)}}{(\beta_L-1)!}
            \prod\limits_{k=k_0}^{L-1}\frac{2^{(k+1)\beta_k}}{\beta_k!}\\
            &=2^{3m+d-1}\frac{\beta_L}{2^L}
            \prod\limits_{k=k_0}^{L}\frac{2^{k\beta_k}}{\beta_k!}
            \leqslant2^{4m}\frac{\beta_L}{2^L}
            \prod\limits_{k=k_0}^{L}\frac{2^{k\beta_k}}{\beta_k!}.
        \end{aligned}
        \label{shi-shanjianjishangjie}
    \end{equation}

    \proofstep{4}
By \eqref{shi-fenkuaizhiji},
    \begin{equation}
        \|P_{I_0}\textbf{z}\|_X^h
        \prod\limits_{k=k_0}^{K}\|P_{I_k}\textbf{z}\|_X^{\beta_k}=0
        \quad\Longrightarrow\quad
        \sum\limits_{\alpha\in A_{h,\beta}}|c_\alpha(G)\textbf{z}^\alpha|=0.
        \label{shi-lingkuaichuli}
    \end{equation}
    Otherwise, define $\textbf{u}$ by
    \begin{equation}
        \begin{aligned}
            P_{I_0}\textbf{u}
            &\triangleq\left(\frac{h}{m}\right)^{\frac{1}{p}}
            \frac{P_{I_0}\textbf{z}}{\|P_{I_0}\textbf{z}\|_X}
            &&\text{if }h>0,\\
            P_{I_k}\textbf{u}
            &\triangleq\left(\frac{\beta_k}{m}\right)^{\frac{1}{p}}
            \frac{P_{I_k}\textbf{z}}{\|P_{I_k}\textbf{z}\|_X}
            &&\text{if }\beta_k>0,\quad k_0\leqslant k\leqslant K,
        \end{aligned}
        \label{shi-kuaiguifanhua}
    \end{equation}
    and set all remaining coordinates of $\textbf{u}$ equal to zero. By \eqref{shi-kuaiguifanhua},
    \begin{equation}
        \|\textbf{u}\|_X^p=\frac{h}{m}+\sum\limits_{k=k_0}^{K}\frac{\beta_k}{m}=1
        \quad(p<\infty),\qquad
        \|\textbf{u}\|_X=1\quad(p=\infty).
        \label{shi-kuaixiangliangfanshu}
    \end{equation}
    For $\alpha\in A_{h,\beta}$, \eqref{shi-fenkuaizhiji} and \eqref{shi-kuaiguifanhua} give
    \begin{equation}
        \textbf{z}^\alpha
        =\frac{m^{\frac{m}{p}}}{h^{\frac{h}{p}}}\|P_{I_0}\textbf{z}\|_X^h
        \prod\limits_{k=k_0}^{K}\frac{\|P_{I_k}\textbf{z}\|_X^{\beta_k}}{\beta_k^{\frac{\beta_k}{p}}}
        \,\textbf{u}^\alpha.
        \label{shi-kuaidanxiangshi}
    \end{equation}
    By \eqref{shi-zhijixishuguji}, \eqref{shi-kuaixiangliangfanshu}, and \eqref{shi-kuaidanxiangshi},
    \begin{equation}
        \begin{aligned}
            \sum\limits_{\alpha\in A_{h,\beta}}|c_\alpha(G)\textbf{z}^\alpha|
            &\leqslant e^{2m}|A_{h,\beta}^-|^\sigma
            \frac{m^{\frac{m}{p}}}{h^{\frac{h}{p}}}\|P_{I_0}\textbf{z}\|_X^h
            \prod\limits_{k=k_0}^{K}\frac{\|P_{I_k}\textbf{z}\|_X^{\beta_k}}{\beta_k^{\frac{\beta_k}{p}}}.
        \end{aligned}
        \label{shi-dankuaixishuguji}
    \end{equation}
    Since $h+d=m$, the identity $\max\limits_{0\leqslant u\leqslant1}\{-u\log u-(1-u)\log(1-u)\}=\log2$ gives
    \begin{equation}
        \frac{m^m}{h^hd^d}
        =e^{m\left(-\frac{h}{m}\log\frac{h}{m}-\frac{d}{m}\log\frac{d}{m}\right)}
        \leqslant2^m.
        \label{shi-eryuanshang}
    \end{equation}
    Also, $v!\leqslant v^v$ and \eqref{shi-fenleizhishu} give
    \begin{equation}
        \frac{1}{(v!)^\sigma v^{\frac{v}{p}}}
        \leqslant\frac{1}{(v!)^{\sigma+\frac{1}{p}}}
        =\frac{1}{(v!)^{\frac{1}{q}}},
        \qquad v\in\mathbb{N}_0.
        \label{shi-jichenghebing}
    \end{equation}
    Substitute \eqref{shi-shanjianjishangjie} and \eqref{shi-kuaiquanzhong} into \eqref{shi-dankuaixishuguji}. Using $\beta_L\leqslant d$, \eqref{shi-eryuanshang}, and \eqref{shi-jichenghebing}, we obtain
    \begin{equation}
        \begin{aligned}
            \sum\limits_{\alpha\in A_{h,\beta}}|c_\alpha(G)\textbf{z}^\alpha|
            &\leqslant e^{2m}2^{4\sigma m}\beta_L^\sigma2^{-\sigma L}
            \|P_{I_0}\textbf{z}\|_X^h
            \frac{m^{\frac{m}{p}}}{h^{\frac{h}{p}}}
            \prod\limits_{k=k_0}^{L}
            \frac{\omega_k^{\beta_k}}{(\beta_k!)^\sigma\beta_k^{\frac{\beta_k}{p}}}\\
            &\leqslant\left(e^2 2^{4\sigma+\frac{1}{p}}\right)^m
            d^\sigma2^{-\sigma L}\|P_{I_0}\textbf{z}\|_X^{m-d}d^{\frac{d}{p}}
            \prod\limits_{k=k_0}^{L}\frac{\omega_k^{\beta_k}}{(\beta_k!)^{\frac{1}{q}}}\\
            &\leqslant C_0^m d^\sigma2^{-\sigma L}
            \|P_{I_0}\textbf{z}\|_X^{m-d}d^{\frac{d}{p}}
            \prod\limits_{k=k_0}^{L}\frac{\omega_k^{\beta_k}}{(\beta_k!)^{\frac{1}{q}}},
        \end{aligned}
        \label{shi-kuaishuaijianguji}
    \end{equation}
    where the last inequality follows from \eqref{shi-fuzhuchangshu}. By \eqref{shi-lingkuaichuli}, \eqref{shi-kuaishuaijianguji} also holds if $h>0$ and $\| P_{I_0}\textbf z\|_X=0$, or if $\beta_k>0$ and $\| P_{I_k}\textbf z\|_X=0$ for some $k_0\leqslant k\leqslant L$.

    \proofstep{5}
By \eqref{shi-kuaitongyijie}, \namedref{Lemma}{yinli-fenkuaiqiuhe} applies to $(\omega_k)_{k=k_0}^{K}$. For $1\leqslant d\leqslant m$, \eqref{shi-kuaishuaijianguji} gives
    \begin{equation}
        \begin{aligned}
            \sum\limits_{|\beta|=d}\sum\limits_{\alpha\in A_{m-d,\beta}}
            |c_\alpha(G)\textbf{z}^\alpha|
            &=\sum\limits_{L=k_0}^{K}
            \sum\limits_{\substack{|\beta|=d\\\max\{k:\beta_k>0\}=L}}
            \sum\limits_{\alpha\in A_{m-d,\beta}}|c_\alpha(G)\textbf{z}^\alpha|\\
            &\leqslant C_0^m\|P_{I_0}\textbf{z}\|_X^{m-d}d^\sigma
            \sum\limits_{L=k_0}^{K}2^{-\sigma L}d^{\frac{d}{p}}
            \sum\limits_{\substack{\beta_{k_0},\cdots,\beta_L\in\mathbb{N}_0\\
                    \beta_{k_0}+\cdots+\beta_L=d}}
            \prod\limits_{k=k_0}^{L}\frac{\omega_k^{\beta_k}}{(\beta_k!)^{\frac{1}{q}}}\\
            &\leqslant C_0^m\|P_{I_0}\textbf{z}\|_X^{m-d}(3C_0)^d d^\sigma
            \sum\limits_{L=k_0}^{K}2^{-\sigma L}
            \left(\frac{L}{d}\right)^{\sigma d}
            \left(1+\frac{L}{d}\right)^{(1-\frac{1}{q})d}\\
            &\leqslant3^dC_0^{2m}d^\sigma\sum\limits_{L=1}^{\infty}2^{-\sigma L}
            \left(\frac{L}{d}\right)^{\sigma d}
            \left(1+\frac{L}{d}\right)^{(1-\frac{1}{q})d}\\
            &\leqslant3^dC_0^{2m+2d},
        \end{aligned}
        \label{shi-tongciweikuaiqiuhe}
    \end{equation}
    by \eqref{shi-duocishuqiuhe}, \eqref{shi-kuaitongyijie}, and \eqref{shi-zhishishuaijian}.

    For $d=0$, $\beta=0$. By \eqref{shi-shanjianzhiji}, \eqref{shi-fenkuaizhiji}, and \eqref{shi-shuangjinfenkuai},
    \begin{equation}
        \begin{aligned}
            A_{m,0}^-&=\{\eta\in\mathbb{N}_0^{(\mathbb{N})}:|\eta|=m-1,\ \mathrm{supp}\eta\subset I_0\},\\
            |A_{m,0}^-|&=\binom{|I_0|+m-2}{m-1}
            \leqslant2^{|I_0|+m-2}\leqslant2^{3m}.
        \end{aligned}
        \label{shi-shoukuaizhijishu}
    \end{equation}
    Apply \eqref{shi-zhijixishuguji} with $\textbf{u}=P_{I_0}\textbf{z}$. By \eqref{shi-shoukuaizhijishu}, \eqref{shi-kuaitongyijie}, and \eqref{shi-fuzhuchangshu},
    \begin{equation}
        \begin{aligned}
            \sum\limits_{\alpha\in A_{m,0}}|c_\alpha(G)\textbf{z}^\alpha|
            &\leqslant e^{2m}|A_{m,0}^-|^\sigma\|P_{I_0}\textbf{z}\|_X^m\|G\|\leqslant e^{2m}2^{3\sigma m}C_0^m\leqslant C_0^{2m}.
        \end{aligned}
        \label{shi-shoukuaixishuguji}
    \end{equation}
    Combining \eqref{shi-xishuhefenjie}, \eqref{shi-tongciweikuaiqiuhe}, and \eqref{shi-shoukuaixishuguji}, we obtain
    \begin{equation}
        \begin{aligned}
            \sum\limits_{|\alpha|=m}|c_\alpha(G)\textbf{z}^\alpha|
            &\leqslant C_0^{2m}+\sum\limits_{d=1}^{m}3^dC_0^{2m+2d}\leqslant(m+1)3^mC_0^{4m}\leqslant(6C_0^4)^m.
        \end{aligned}
        \label{shi-jingdianqiuhe}
    \end{equation}
    Put $C\triangleq6C_0^4$. By \eqref{shi-xishuhebudong}, \eqref{shi-yicixishuguji}, \eqref{shi-jingdianqiuhe}, and homogeneity, for every $m,n\in\mathbb{N}$ and $G=G\circ P_n$,
    \begin{equation}
        \sum\limits_{|\alpha|=m}|c_\alpha(G)\textbf{z}^\alpha|
        \leqslant C^mM(P_n\textbf{z})^m\|G\|
        \leqslant C^mM(\textbf{z})^m\|G\|.
        \label{shi-youxianchongpaiguji}
    \end{equation}
    Finally, for arbitrary $G\in\mathcal{P}^m(X)$, \eqref{shi-shouweisuosuo}, and \eqref{shi-youxianchongpaiguji} yield
    $$
    \begin{aligned}
        \sum\limits_{|\alpha|=m}|c_\alpha(G)\textbf{z}^\alpha|
        &=\lim\limits_{n\to\infty}
        \sum\limits_{\substack{|\alpha|=m\\\mathrm{supp}\alpha\subset\{1,\cdots,n\}}}
        |c_\alpha(G)\textbf{z}^\alpha|=\lim\limits_{n\to\infty}
        \sum\limits_{|\alpha|=m}|c_\alpha(G\circ P_n)\textbf{z}^\alpha|\\
        &\leqslant C^mM(\textbf{z})^m\sup\limits_{n\in\mathbb{N}}\|G\circ P_n\|
        \leqslant C^mM(\textbf{z})^m\|G\|.
    \end{aligned}
    $$
    This proves \eqref{shi-chongpaixishuguji}.
\end{proof}

\begin{lemma}\label{yinli-duoxiangshiwang}
    Let $n,m\in\mathbb{N}$ with $m\geqslant2$. There exists a finite set
    $\mathcal{N}\subset\{\textbf{u}\in\mathbb{C}^n:
    \|\textbf{u}\|_{\ell^2}\leqslant1\}$ with
    $|\mathcal{N}|\leqslant(1+12m)^{2n}$ such that
    $$
    \sup\limits_{\|\textbf{u}\|_{\ell^2}\leqslant1}|P(\textbf{u})|
    \leqslant2\max\limits_{\textbf{u}\in\mathcal{N}}|P(\textbf{u})|
    $$
    for every $m$-homogeneous polynomial $P$ on $\mathbb{C}^n$.
\end{lemma}

\begin{proof}
        \proofstep{1}

    Set $\delta=\frac{1}{6m}$. For $\textbf{a}\in\mathbb{C}^n$
    and $r>0$, write
    $$
    B(\textbf{a},r)
    \triangleq
    \{\textbf{w}\in\mathbb{C}^n:
    \|\textbf{w}-\textbf{a}\|_{\ell^2}<r\}.
    $$
    Consider the largest $q\in\mathbb{N}$ for which there exist
    $\textbf{u}^{(1)},\cdots,\textbf{u}^{(q)}\in\mathbb{C}^n$
    such that
    $$
    \|\textbf{u}^{(j)}\|_{\ell^2}\leqslant1,
    \qquad 1\leqslant j\leqslant q,
    $$
    and
    $$
    \|\textbf{u}^{(i)}-\textbf{u}^{(j)}\|_{\ell^2}>\delta,
    \qquad 1\leqslant i<j\leqslant q.
    $$
    The balls
    $B(\textbf{u}^{(j)},\frac{\delta}{2})$,
    $1\leqslant j\leqslant q$, are pairwise disjoint. For
    $\textbf{w}\in B(\textbf{u}^{(j)},\frac{\delta}{2})$,
    $$
    \|\textbf{w}\|_{\ell^2}
    \leqslant
    \|\textbf{w}-\textbf{u}^{(j)}\|_{\ell^2}
    +\|\textbf{u}^{(j)}\|_{\ell^2}
    <\frac{\delta}{2}+1.
    $$
    Thus
    $$
    \bigcup\limits_{j=1}^q
    B\left(\textbf{u}^{(j)},\frac{\delta}{2}\right)
    \subset
    B\left(\textbf{0},1+\frac{\delta}{2}\right).
    $$

    Let $\mathrm{vol}_{2n}$ denote real $2n$-dimensional
    Lebesgue measure on $\mathbb{C}^n$. We have
    $$
    \begin{aligned}
        &q\left(\frac{\delta}{2}\right)^{2n}
        \mathrm{vol}_{2n}(B(\textbf{0},1))=
        \sum\limits_{j=1}^q
        \mathrm{vol}_{2n}
        \left(B\left(\textbf{u}^{(j)},\frac{\delta}{2}\right)\right)\\
        &\quad=
        \mathrm{vol}_{2n}
        \left(
        \bigcup\limits_{j=1}^q
        B\left(\textbf{u}^{(j)},\frac{\delta}{2}\right)
        \right)\leqslant
        \mathrm{vol}_{2n}
        \left(B\left(\textbf{0},1+\frac{\delta}{2}\right)\right)\\
        &\quad=
        \left(1+\frac{\delta}{2}\right)^{2n}
        \mathrm{vol}_{2n}(B(\textbf{0},1)).
    \end{aligned}
    $$

    We conclude that
    $$
    q
    \leqslant
    \left(
    \frac{1+\frac{\delta}{2}}{\frac{\delta}{2}}
    \right)^{2n}
    =\left(1+\frac{2}{\delta}\right)^{2n}
    =(1+12m)^{2n}.
    $$

    Let $$\mathcal{N} \triangleq \left\{ \textbf{u}^{(1)} ,\cdots ,\textbf{u}^{(q)} \right\}.$$In particular,
    $$
    |\mathcal{N}|\leqslant(1+12m)^{2n}.
    $$

    \proofstep{2}

    Let $P$ be any $m$-homogeneous polynomial on $\mathbb{C}^n$, and put
$
    K\triangleq
    \sup\limits_{\|\textbf{u}\|_{\ell^2}\leqslant1}|P(\textbf{u})|<\infty.
$   Fix $\textbf{u},\textbf{v}\in\mathbb{C}^n$ with
    $\|\textbf{u}\|_{\ell^2}\leqslant1$ and
    $\|\textbf{v}\|_{\ell^2}=1$. For every
    $\lambda\in\mathbb{C}$ with $|\lambda|=\frac{1}{m}$,
    $$
    \left\|
    \frac{\textbf{u}+\lambda\textbf{v}}{1+\frac{1}{m}}
    \right\|_{\ell^2}
    \leqslant
    \frac{\|\textbf{u}\|_{\ell^2}
        +|\lambda|\|\textbf{v}\|_{\ell^2}}
    {1+\frac{1}{m}}
    \leqslant1.
    $$
    Consequently, the homogeneity of $P$ gives
    $$
    \begin{aligned}
        |P(\textbf{u}+\lambda\textbf{v})|
        &=\left(1+\frac{1}{m}\right)^m
        \left|
        P\left(
        \frac{\textbf{u}+\lambda\textbf{v}}{1+\frac{1}{m}}
        \right)
        \right|\leqslant\left(1+\frac{1}{m}\right)^mK.
    \end{aligned}
    $$
    Applying the Cauchy formula to the entire function
    $\lambda\mapsto P(\textbf{u}+\lambda\textbf{v})$, we obtain
    $$
    DP(\textbf{u})[\textbf{v}]
    =\frac{1}{2\pi \mathrm{i}}
    \int\limits_{|\lambda|=\frac{1}{m}}
    \frac{P(\textbf{u}+\lambda\textbf{v})}{\lambda^2}
    \,\mathrm{d}\lambda,
    $$
 Hence
    $$
    \begin{aligned}
    &|DP(\textbf{u})[\textbf{v}]|
        \leqslant
        \frac{1}{2\pi}
        \int\limits_{|\lambda|=\frac{1}{m}}
        \frac{|P(\textbf{u}+\lambda\textbf{v})|}{|\lambda|^2}
        \,|\mathrm{d}\lambda|\\
        &\quad \leqslant
        \frac{1}{2\pi}\frac{2\pi}{m}\,
        m^2\left(1+\frac{1}{m}\right)^mK=m\left(1+\frac{1}{m}\right)^mK
        \leqslant3mK.
    \end{aligned}
    $$
 By the linearity of
    $DP(\textbf{u})$, for every nonzero $\textbf{h}\in\mathbb{C}^n$,
    $$
    \begin{aligned}
        |DP(\textbf{u})[\textbf{h}]|
        &=\|\textbf{h}\|_{\ell^2}
        \left|
        DP(\textbf{u})
        \left[
        \frac{\textbf{h}}{\|\textbf{h}\|_{\ell^2}}
        \right]
        \right|\leqslant3mK\|\textbf{h}\|_{\ell^2}.
    \end{aligned}
    $$
    The same inequality holds for $\textbf{h}=\textbf{0}$.

    Now let $\textbf{u},\textbf{v}\in\mathbb{C}^n$ satisfy
    $\|\textbf{u}\|_{\ell^2},\|\textbf{v}\|_{\ell^2}\leqslant1$, we have
    \begin{equation}
        \begin{aligned}
            |P(\textbf{u})-P(\textbf{v})|
            &\leqslant 3mK\|\textbf{u}-\textbf{v}\|_{\ell^2}.
        \end{aligned}
        \label{shi-duoxiangshicha}
    \end{equation}

    Hence, for every
$\textbf{u}\in\mathbb{C}^n$ with
$\|\textbf{u}\|_{\ell^2}\leqslant1$, there exists
$\textbf{v}\in\mathcal{N}$ such that
$\|\textbf{u}-\textbf{v}\|_{\ell^2}\leqslant\delta$.
    \eqref{shi-duoxiangshicha} yields
    $$
    \begin{aligned}
        |P(\textbf{u})|
        &\leqslant
        |P(\textbf{v})|+|P(\textbf{u})-P(\textbf{v})|\\
        &\leqslant
        \max\limits_{\textbf{w}\in\mathcal{N}}|P(\textbf{w})|
        +3mK\|\textbf{u}-\textbf{v}\|_{\ell^2}\\
        &\leqslant
        \max\limits_{\textbf{w}\in\mathcal{N}}|P(\textbf{w})|
        +3mK\delta\\
        &=
        \max\limits_{\textbf{w}\in\mathcal{N}}|P(\textbf{w})|
        +\frac{K}{2}.
    \end{aligned}
    $$
    Taking the supremum over
    $\|\textbf{u}\|_{\ell^2}\leqslant1$, we obtain
    $$
    K\leqslant
    \max\limits_{\textbf{w}\in\mathcal{N}}|P(\textbf{w})|
    +\frac{K}{2},
    $$
    and therefore
    \begin{equation}
        K\leqslant
        2\max\limits_{\textbf{u}\in\mathcal{N}}|P(\textbf{u})|.
        \label{shi-duoxiangshiwangjie}
    \qedhere
    \end{equation}
\end{proof}
The following proof is inspired by
\cite[Section~5.2, pp.~329--334]{Boas00} and
\cite[Theorem~3.1 and Corollary~3.2]{Bayart12}
\begin{lemma}\label{yinli-suijiduoxiangshi}
Let $n,m\in\mathbb{N}$ with $m\geqslant2$. There are 
$\left( \varepsilon_{\alpha} \right)_{\alpha \in \mathbb{N}_{0}^{n} ,\left| \alpha \right| =m} \in \left\{ -1,1 \right\}^{\left| \left\{ \alpha \in \mathbb{N}_{0}^{n}:\left| \alpha \right| =m \right\} \right|}$ such that the polynomial
$$
R(\textbf{u})\triangleq
\sum\limits_{\substack{\alpha\in\mathbb{N}_0^n\\|\alpha|=m}}
\varepsilon_\alpha\frac{m!}{\alpha!}\textbf{u}^\alpha
$$
satisfies, simultaneously for $1\leqslant r\leqslant2$,
\begin{equation}
\sup\limits_{\|\textbf{u}\|_{\ell^r}\leqslant1}|R(\textbf{u})|
\leqslant16^{2(1-\frac{1}{r})}
(nm!\log m)^{1-\frac{1}{r}}.
\label{shi-suijiduoxiangshiguji}
\end{equation}
\end{lemma}

\begin{proof}

Choose $\mathcal{N}$ as in
\namedref{Lemma}{yinli-duoxiangshiwang}.

\proofstep{1}

Let $(\varepsilon_\alpha)_{\alpha\in\mathbb{N}_0^n,\ |\alpha|=m}$
be a family of independent Rademacher random variables; that is,
$$
\mathbb{P}(\varepsilon_\alpha=1)
=
\mathbb{P}(\varepsilon_\alpha=-1)
=
\frac{1}{2}
$$
for every $\alpha\in\mathbb{N}_0^n$ with $|\alpha|=m$. For a fixed
$\|\textbf{u}\|_{\ell^2}\leqslant1$, put
$b_\alpha=\frac{m!}{\alpha!}\textbf{u}^\alpha$.
Since $\frac{m!}{\alpha!}\leqslant m!$,
\begin{equation}
\sum\limits_{|\alpha|=m}|b_\alpha|^2
\leqslant m!\sum\limits_{|\alpha|=m}
\frac{m!}{\alpha!}|\textbf{u}|^{2\alpha}
=m!\left(\sum\limits_{j=1}^{n}|u_j|^2\right)^m\leqslant m!.
\label{shi-suijifangcha}
\end{equation}
Let $a_\alpha\in\mathbb{R}$, $|\alpha|=m$. For $|\alpha|=m$, set
$
X_\alpha\triangleq\varepsilon_\alpha a_\alpha.
$
Then the random variables $X_\alpha$ are independent and satisfy
$$
\mathbb{E}X_\alpha=0,
\qquad
-|a_\alpha|\leqslant X_\alpha\leqslant|a_\alpha|.
$$
Since there are only finitely many multi-indices $\alpha$ with
$|\alpha|=m$, \cite[Theorem~2, p.~16]{Hoeffding63} applies and gives,
for every $t>0$,
$$
\begin{aligned}
    \mathbb{P}\left(
    \sum\limits_{|\alpha|=m}\varepsilon_\alpha a_\alpha>t
    \right)
    &\leqslant
    e^{
    -\frac{2t^2}
    {\sum\limits_{|\alpha|=m}(2|a_\alpha|)^2}
    }=
    e^{
    -\frac{t^2}
    {2\sum\limits_{|\alpha|=m}a_\alpha^2}
    }.
\end{aligned}
$$
Applying the same estimate to the independent random variables
$-X_\alpha$ yields
$$
\mathbb{P}\left(
\sum\limits_{|\alpha|=m}\varepsilon_\alpha a_\alpha<-t
\right)
\leqslant
e^{
-\frac{t^2}
{2\sum\limits_{|\alpha|=m}a_\alpha^2}
}.
$$
Hence
\begin{equation}
    \mathbb{P}\left(
    \left|\sum\limits_{|\alpha|=m}
    \varepsilon_\alpha a_\alpha\right|>t
    \right)
    \leqslant
    2e^{
    -\frac{t^2}
    {2\sum\limits_{|\alpha|=m}a_\alpha^2}
    }.
    \label{shi-suijiweigailv}
\end{equation}
Applying \eqref{shi-suijiweigailv} with
$a_\alpha=\mathrm{Re}b_\alpha$ and with
$a_\alpha=\mathrm{Im}b_\alpha$, respectively, and using
\eqref{shi-suijifangcha}, yields
\begin{equation}
	\begin{aligned}
		\mathbb{P}\bigl(|R(\textbf{u})|>t\bigr)
		&\leqslant
		\mathbb{P}\left(
		\left|\mathrm{Re}R(\textbf{u})\right|
		>\frac{t}{\sqrt{2}}
		\right)
		+
		\mathbb{P}\left(
		\left|\mathrm{Im}R(\textbf{u})\right|
		>\frac{t}{\sqrt{2}}
		\right)
		\\
		&\leqslant
		2e^{
			-\frac{(\frac{t}{\sqrt{2}})^2}
			{2\sum\limits_{\left| \alpha \right| =m}(\mathrm{Re}b_\alpha)^2}
		}
		+
		2e^{
			-\frac{(\frac{t}{\sqrt{2}})^2}
			{2\sum\limits_{\left| \alpha \right| =m}(\mathrm{Im}b_\alpha)^2}
		}
		\\
		&\leqslant
		2e^{
			-\frac{t^2}{4\sum\limits_{\left| \alpha \right| =m}|b_\alpha|^2}
		}
		+
		2e^{
			-\frac{t^2}{4\sum\limits_{\left| \alpha \right| =m}|b_\alpha|^2}
		}
		\\
		&\leqslant
		4e^{-\frac{t^2}{4m!}}.
	\end{aligned}\label{shi-cheche}
\end{equation}

Set
$$
t=2\sqrt{m!\bigl(\log8+2n\log(1+12m)\bigr)}.
$$
$\left| \mathcal{N} \right| \leqslant \left( 1+12m \right)^{2n}$ shows that
$$
\mathbb{P}\left(\max\limits_{\textbf{u}\in\mathcal{N}}|R(\textbf{u})|>t\right)
\leqslant4(1+12m)^{2n}e^{-\frac{t^2}{4m!}}
=\frac{1}{2}.
$$
Consequently, there is a choice of signs for which
\eqref{shi-duoxiangshiwangjie} gives
\begin{equation}
    \sup\limits_{\|\textbf{u}\|_{\ell^2}\leqslant1}|R(\textbf{u})|
    \leqslant4\sqrt{m!\bigl(\log8+2n\log(1+12m)\bigr)}
    \leqslant B\triangleq 16\sqrt{nm!\log m}.\label{shi-suijierfanshu}
\end{equation}
For the last inequality, $m\geqslant2$ gives
$1+12m\leqslant m^5$ and $\log8\leqslant3n\log m$;
thus the quantity inside the square root is at most $13nm!\log m$,
and $4\sqrt{13}<16$.

\proofstep{2}

For the same choice of signs, the multinomial theorem gives
$$
|R(\textbf{u})|
\leqslant
\sum\limits_{|\alpha|=m}\frac{m!}{\alpha!}|\textbf{u}^{\alpha}|
=
\left(\sum\limits_{j=1}^{n}|u_j|\right)^m,
$$
and hence
$$
\sup\limits_{\|\textbf{u}\|_{\ell^1}\leqslant1}|R(\textbf{u})|
\leqslant1.
$$
Together with the estimate \eqref{shi-suijierfanshu}, this proves
\eqref{shi-suijiduoxiangshiguji} for $r=1$ and $r=2$.

Now let $1<r<2$ and set
$
\theta=2\left(1-\frac{1}{r}\right)\in(0,1).
$
Fix $\textbf{u}\in\mathbb{C}^n$ with $\|\textbf{u}\|_{\ell^r}=1$, and define
$$
\textbf{v}(w)=\bigl(v_1(w),\cdots,v_n(w)\bigr),
\qquad
0\leqslant\mathrm{Re}w\leqslant1,
$$
where
$$
v_j(w)=
\begin{cases}
    \displaystyle
    \frac{u_j}{|u_j|}
    e^{r\left(1-\frac{w}{2}\right)\log|u_j|},
    & u_j\neq0,\\[1ex]
    0,&u_j=0.
\end{cases}
$$
Since $r(1-\frac{\theta}{2})=1$, we have
$\textbf{v}(\theta)=\textbf{u}$. Moreover, for every $t\in\mathbb{R}$,
$$
\sum\limits_{j=1}^{n}|v_j(\mathrm{i}t)|
=
\sum\limits_{j=1}^{n}|u_j|^r
=
1,
\qquad
\sum\limits_{j=1}^{n}|v_j(1+\mathrm{i}t)|^2
=
\sum\limits_{j=1}^{n}|u_j|^r
=
1.
$$
By the cases $r=1$ and $r=2$ of
\eqref{shi-suijiduoxiangshiguji}, it follows that
$$
\sup\limits_{t\in\mathbb{R}}
|R(\textbf{v}(\mathrm{i}t))|
\leqslant1,
\qquad
\sup\limits_{t\in\mathbb{R}}
|R(\textbf{v}(1+\mathrm{i}t))|
\leqslant B.
$$
Since $\|\textbf{u}\|_{\ell^r}=1$, we have $|u_j|\leqslant1$ for
$1\leqslant j\leqslant n$, and therefore
$|v_j(w)|\leqslant1$ whenever
$0\leqslant\mathrm{Re}w\leqslant1$. Thus the function
$w\mapsto R(\textbf{v}(w))$ is bounded and continuous on the closed
strip $0\leqslant\mathrm{Re}w\leqslant1$ and holomorphic in its
interior. By Hadamard's three-lines theorem
\cite[Theorem~12.8, p.~257]{Rudin87},
$$
|R(\textbf{u})|
=
|R(\textbf{v}(\theta))|
\leqslant
B^\theta
=
16^{2(1-\frac{1}{r})}
(nm!\log m)^{1-\frac{1}{r}}.
$$
Since $\textbf{u}$ was arbitrary on the $\ell^r$ unit sphere, the
$m$-homogeneity of $R$ yields
$$
\sup\limits_{\|\textbf{u}\|_{\ell^r}\leqslant1}|R(\textbf{u})|
\leqslant
16^{2(1-\frac{1}{r})}
(nm!\log m)^{1-\frac{1}{r}}.
$$
Hence \eqref{shi-suijiduoxiangshiguji} holds for every
$1\leqslant r\leqslant2$.
\end{proof}

The following proof is inspired by
\cite[proof of Theorem~2.2\,(b)]{BDFMS17} and
\cite[Section~4A, pp.~953--954]{GMMS21}.
For the case $2<p<\infty$, see also
\cite[Section~6.3.1, pp.~99--101]{Mansilla19}.
\begin{lemma}\label{yinli-shangjixianbijiao}
Fix $1<p\leqslant\infty$, $\textbf{z}\in X$, and
$N\in\mathbb{N}_0$. Suppose there exist $A,R>0$ such that
\begin{equation}
\sum\limits_{|\alpha|=m}|c_\alpha(G)\textbf{z}^\alpha|
\leqslant AR^m\|G\|
\label{shi-weibuzhishijiading}
\end{equation}
for every $m\geqslant2$ and every $G\in\mathcal{P}^m(X)$ depending only on finitely many coordinates $x_j$ with $j>N$.
Then
\begin{equation}
\limsup\limits_{n\to\infty}
\frac{\left(\sum\limits_{j=1}^{n}
\bigl((Q_N\textbf{z})_j^*\bigr)^{q}\right)^{\frac{1}{q}}}
{(\log(n+1))^{\sigma}}
\leqslant R.
\label{shi-weibushangjixian}
\end{equation}
\end{lemma}

\begin{proof}
	Put $C\triangleq16$, $r\triangleq\min\{p,2\}$, and $\frac{1}{s}\triangleq\frac{1}{r}-\frac{1}{p}$, with $s=\infty$ when $p=r$. By \eqref{shi-fenleizhishu}, $0<\sigma=1-\frac{1}{r}\leqslant\frac{1}{2}$, $\frac{1}{r}=\frac{1}{s}+\frac{1}{p}$, and $\frac{1}{q}+\frac{1}{s}=1$. Fix $n,m\in\mathbb{N}$ with $m\geqslant2$. By \eqref{shi-suijiduoxiangshiguji} in \namedref{Lemma}{yinli-suijiduoxiangshi}, there exists a family
$
	(\varepsilon_\alpha)_{\substack{\alpha\in\mathbb{N}_0^n\\|\alpha|=m}}
	\in
	\{-1,1\}^{\{\alpha\in\mathbb{N}_0^n:\,|\alpha|=m\}}
$
	such that
	\begin{equation}
		\sup\limits_{\|\textbf{u}\|_{\ell^r}\leqslant1}
		\left|\sum\limits_{\substack{\alpha\in\mathbb{N}_0^n\\|\alpha|=m}}
		\varepsilon_\alpha\frac{m!}{\alpha!}\textbf{u}^\alpha\right|
		\leqslant16^{2\sigma}(nm!\log m)^\sigma
		\leqslant C(nm!\log m)^\sigma.
		\label{shi-suijicanshuguji}
	\end{equation}
	Choose $N<j_1<\cdots<j_n$, put $J\triangleq\{j_1,\cdots,j_n\}$, and let $\textbf{a}=(a_1,\cdots,a_n)\in\mathbb{C}^n$ satisfy $\|\textbf{a}\|_{\ell^s}\leqslant1$. H\"older's inequality, including the cases $s=\infty$ and $p=\infty$, gives
	$$
	\left\|(a_kx_{j_k})_{k=1}^n\right\|_{\ell^r}
	\leqslant\|\textbf{a}\|_{\ell^s}
	\left\|(x_{j_k})_{k=1}^n\right\|_{\ell^p}
	\leqslant\|\textbf{x}\|_X,
	\qquad \textbf{x}\in X.
	$$
	Define
	$$
	G_{\textbf{a}}(\textbf{x})\triangleq
	\sum\limits_{\substack{\alpha\in\mathbb{N}_0^n\\|\alpha|=m}}
	\varepsilon_\alpha\frac{m!}{\alpha!}
	\prod\limits_{k=1}^n(a_kx_{j_k})^{\alpha_k},
	\qquad \textbf{x}\in X.
	$$
	Then $G_{\textbf{a}}\in\mathcal{P}^m(X)$ depends only on the coordinates indexed by $J$, and \eqref{shi-suijicanshuguji} gives $\|G_{\textbf{a}}\|\leqslant C(nm!\log m)^\sigma$. For every $\beta\in\mathbb{N}_0^{(\mathbb{N})}$ with $|\beta|=m$,
	$$
	c_\beta(G_{\textbf{a}})=
	\begin{cases}
		\displaystyle
		\varepsilon_{(\beta_{j_1},\cdots,\beta_{j_n})}
		\frac{m!}{\beta!}\prod\limits_{k=1}^n a_k^{\beta_{j_k}},
		&\mathrm{supp}\beta\subset J,\\[5pt]
		0,&\mathrm{supp}\beta\not\subset J.
	\end{cases}
	$$
	The multinomial theorem and \eqref{shi-weibuzhishijiading} therefore yield
	\begin{equation}
		\begin{aligned}
		&\left(\sum\limits_{k=1}^n|a_kz_{j_k}|\right)^m
			=\sum\limits_{\substack{\alpha\in\mathbb{N}_0^n\\|\alpha|=m}}
			\frac{m!}{\alpha!}\prod\limits_{k=1}^n|a_kz_{j_k}|^{\alpha_k}\\
			&\quad =\sum\limits_{\substack{\beta\in\mathbb{N}_0^{(\mathbb{N})}\\|\beta|=m}}
			|c_\beta(G_{\textbf{a}})\textbf{z}^\beta| \leqslant AR^m\|G_{\textbf{a}}\|
			\\&\quad\leqslant AC R^m(nm!\log m)^\sigma.
		\end{aligned}
		\label{shi-duijiaoceshiguji}
	\end{equation}
	
	Put $S\triangleq\left\|(z_{j_k})_{k=1}^n\right\|_{\ell^q}$. By H\"older's inequality, $$\sum\limits_{k=1}^n|a_kz_{j_k}|\leqslant\|\textbf{a}\|_{\ell^s}S\leqslant S.$$If $q=1$, the choice $\textbf{a}=(1,\cdots,1)$ gives $\|\textbf{a}\|_{\ell^\infty}=1$ and $\sum\limits_{k=1}^n|a_kz_{j_k}|=S$. If $q>1$ and $S>0$, choose $a_k=\frac{|z_{j_k}|^{q-1}}{S^{q-1}}$, $1\leqslant k\leqslant n$. Since $s=\frac{q}{q-1}$,
	$$
	\|\textbf{a}\|_{\ell^s}^s
	=\frac{\sum\limits_{k=1}^n|z_{j_k}|^q}{S^q}=1,
	\qquad
	\sum\limits_{k=1}^n|a_kz_{j_k}|
	=\frac{\sum\limits_{k=1}^n|z_{j_k}|^q}{S^{q-1}}=S.
	$$
	For $S=0$, both sides of the following equality vanish. Thus \eqref{shi-duijiaoceshiguji} implies
	$$
	\begin{aligned}
		\left(\sum\limits_{k=1}^n|z_{j_k}|^q\right)^{\frac{1}{q}}
		&=\sup\limits_{\substack{\textbf{a}\in\mathbb{C}^n\\\|\textbf{a}\|_{\ell^s}\leqslant1}}
		\sum\limits_{k=1}^n|a_kz_{j_k}|\leqslant R(AC)^{\frac{1}{m}}n^{\frac{\sigma}{m}}
		(m!)^{\frac{\sigma}{m}}(\log m)^{\frac{\sigma}{m}}.
	\end{aligned}
	$$
	
	For every $F\subset\mathbb{N}$ with $|F|=n$, choose $H\subset\{N+1,N+2,\cdots\}\setminus F$ with $|H|=|F\cap\{1,\cdots,N\}|$, and put $\widetilde F\triangleq(F\setminus\{1,\cdots,N\})\cup H$. Then $\widetilde F\subset\{N+1,N+2,\cdots\}$, $|\widetilde F|=n$, and
	$$
	\sum\limits_{j\in F}|(Q_N\textbf{z})_j|^q
	=\sum\limits_{j\in F\setminus\{1,\cdots,N\}}|z_j|^q
	\leqslant\sum\limits_{j\in\widetilde F}|z_j|^q.
	$$
	Since $\textbf{z}\in X\subset c_0$, \eqref{shi-chongpaishangquejie} yields
	\begin{equation}
		\begin{aligned}
			&\left(\sum\limits_{j=1}^n\bigl((Q_N\textbf{z})_j^*\bigr)^q\right)^{\frac{1}{q}}
			=\sup\limits_{\substack{F\subset\mathbb{N}\\|F|=n}}
			\left(\sum\limits_{j\in F}|(Q_N\textbf{z})_j|^q\right)^{\frac{1}{q}}\\
			&\quad=\sup\limits_{N<j_1<\cdots<j_n}
			\left(\sum\limits_{k=1}^n|z_{j_k}|^q\right)^{\frac{1}{q}}\leqslant R(AC)^{\frac{1}{m}}n^{\frac{\sigma}{m}}
			(m!)^{\frac{\sigma}{m}}(\log m)^{\frac{\sigma}{m}}.
		\end{aligned}
		\label{shi-weibuchongpaijie}
	\end{equation}
Finally, set $m=\max\{2,\lceil\log(n+1)\rceil\}$. By the Stolz--Ces\`aro theorem,
$$
\lim\limits_{k\to\infty}
\frac{\log(k!)-k\log k}{k}
=-\lim\limits_{k\to\infty}
k\log\left(1+\frac{1}{k}\right)=-1.
$$
Hence $\frac{(k!)^{\frac{1}{k}}}{k}\longrightarrow e^{-1}$. Since $\frac{m}{\log(n+1)}\longrightarrow1$, we obtain
$$
n^{\frac{1}{m}}\longrightarrow e,
\qquad
\frac{(m!)^{\frac{1}{m}}}{\log(n+1)}
\longrightarrow e^{-1},
\qquad
(AC)^{\frac{1}{m}}(\log m)^{\frac{\sigma}{m}}
\longrightarrow1.
$$
Dividing \eqref{shi-weibuchongpaijie} by $(\log(n+1))^\sigma$ and taking the upper limit as $n\to\infty$ therefore gives
$$
\limsup\limits_{n\to\infty}
\frac{\left(\sum\limits_{j=1}^{n}
	\bigl((Q_N\textbf{z})_j^*\bigr)^q\right)^{\frac{1}{q}}}
{(\log(n+1))^\sigma}
\leqslant R\left(e\cdot e^{-1}\right)^\sigma=R.
$$
This establishes \eqref{shi-weibushangjixian}.
\end{proof}

	For $1<p\leqslant2$, the characterization of
$\mathrm{dom}H_b(\ell^p)$ was established in
\cite[Theorem~4.1, p.~953]{GMMS21}.
For $2<p<\infty$, the inclusion
$$
\mathrm{dom}H_b(\ell^p)
\subseteq
\{\textbf{z}\in\ell^p:M(\textbf{z})<\infty\}
$$
follows from \cite[Theorem~6.3.2, p.~98]{Mansilla19}. We give a unified proof below.
\subsection{Monomial convergence for functions of bounded type}
\begin{proposition}\label{mingti-youjiexingshoulian}
For $X,p,q,\sigma$ as in \eqref{shi-fenleizhishu}, one has
\begin{equation}
\mathrm{dom}H_b(X)
=\{\textbf{z}\in X:M(\textbf{z})<\infty\}.
\label{shi-youjiexingshoulian}
\end{equation}
\end{proposition}

\begin{proof}

Let $C$ be the constant in
\namedref{Lemma}{yinli-chongpaixishuguji}.
We first record the estimate for a bounded holomorphic function.
Let $r>0$, $h\in H^\infty(rB_X)$, and let $H_m$ be its homogeneous
Taylor polynomials from
\eqref{shi-qicitaile}.
By \eqref{shi-qicixishuyizhi} and
\eqref{shi-qicikexiguji},
$c_\alpha(H_m)=c_\alpha(h)$ for $|\alpha|=m$ and
$\|H_m\|\leqslant\frac{\|h\|_\infty}{r^m}$.
Thus \eqref{shi-chongpaixishuguji} gives, whenever
$\textbf{v}\in X$ and $CM(\textbf{v})<r$,
\begin{equation}
\begin{aligned}
\sum\limits_{\alpha\in\mathbb{N}_0^{(\mathbb{N})}}
|c_\alpha(h)\textbf{v}^\alpha|
&=|h(\textbf{0})|+
\sum\limits_{m=1}^{\infty}\sum\limits_{|\alpha|=m}
|c_\alpha(H_m)\textbf{v}^\alpha|\\
&\leqslant\|h\|_\infty
\sum\limits_{m=0}^{\infty}
\left(\frac{CM(\textbf{v})}{r}\right)^m
=\frac{\|h\|_\infty}{1-\frac{CM(\textbf{v})}{r}}<\infty.
\end{aligned}
\label{shi-youjiehanshuqiuhe}
\end{equation}
For $\textbf{z}\in X$ with $M(\textbf{z})<\infty$ and
$f\in H_b(X)$, choose $r>CM(\textbf{z})$.
By \eqref{shi-youjiexing},
$f|_{rB_X}$ is bounded. Applying
\eqref{shi-youjiehanshuqiuhe} at $\textbf{z}$
proves $\textbf{z}\in\mathrm{dom}H_b(X)$.

Conversely, let $\textbf{z}\in\mathrm{dom}H_b(X)$.
By \eqref{shi-dianchuzhishijie} in
\namedref{Lemma}{yinli-bituguji}, there exists $C\geqslant1$ such that
$$
\sum\limits_{|\alpha|=m}|c_\alpha(G)\textbf{z}^\alpha|
\leqslant C^m\|G\|,
\qquad
G\in\mathcal{P}^m(X),\quad m\in\mathbb{N}.
$$
Applying \namedref{Lemma}{yinli-shangjixianbijiao}  gives
$$
\limsup\limits_{n\to\infty}
\frac{\left(\sum\limits_{j=1}^{n}(z_j^*)^q\right)^{\frac{1}{q}}}
{(\log(n+1))^\sigma}
\leqslant C<\infty.
$$
Hence
$$
M(\textbf{z})
=
\sup\limits_{n\in\mathbb{N}}
\frac{\left(\sum\limits_{j=1}^{n}(z_j^*)^q\right)^{\frac{1}{q}}}
{(\log(n+1))^\sigma}
<\infty,
$$
which proves the reverse inclusion in \eqref{shi-youjiexingshoulian}.
\end{proof}

\subsection{Monomial convergence for holomorphic functions}

\begin{proof}[Proof of \namedref{Theorem}{dingli-jingdianshoulian}]\par\medskip\noindent

By \cite[Theorem~3.1, p.~19]{DMP09}, \eqref{shi-yikeheshoulian} follows directly.
We therefore fix $1<p\leqslant\infty$ and take $X,p,q,\sigma$ as in
\eqref{shi-fenleizhishu}. Then we prove
\begin{equation}
	\mathrm{dom}H(X)
	=\left\{\textbf{z}\in X:
	\lim\limits_{n\to\infty}
	\frac{\left(\sum\limits_{j=1}^{n}(z_j^*)^q\right)^{\frac{1}{q}}}
	{(\log(n+1))^\sigma}=0\right\}.
	\label{shi-jingdiantongyikehua}
\end{equation}
The assertion is immediate for $\textbf{z}=\textbf{0}$.
Let $\textbf{0}\neq\textbf{z}\in\mathrm{dom}H(X)$ and fix
$r>0$. By \eqref{shi-weibupanju} in
\namedref{Proposition}{mingti-jubuhua}, there exist $N\in\mathbb{N}$
and $A>0$ such that
$$
\sum\limits_\alpha|c_\alpha(G)\textbf{z}^\alpha|
\leqslant A\sup\limits_{\textbf{x}\in rB_X}|G(\textbf{x})|,
\qquad G\in\mathscr{P}_N.
$$
Hence, for every $m\in\mathbb{N}$ and
$G\in\mathcal{P}^m(X)\cap\mathscr{P}_N$,
$$
\sum\limits_{|\alpha|=m}|c_\alpha(G)\textbf{z}^\alpha|
\leqslant Ar^m\|G\|.
$$
Since $\overline{\mathbb{C}^{(\mathbb{N})}}^{\,X} =X$,
\namedref{Lemma}{yinli-shangjixianbijiao}, with $R=r$, gives
\eqref{shi-weibushangjixian}. By \eqref{shi-shoubumodong},
$$
\limsup\limits_{n\to\infty}
\frac{\left(\sum\limits_{j=1}^{n}(z_j^*)^q\right)^{\frac{1}{q}}}
{(\log(n+1))^\sigma}
\leqslant r.
$$
Since $r>0$ is arbitrary,
$$
\lim\limits_{n\to\infty}
\frac{\left(\sum\limits_{j=1}^{n}(z_j^*)^q\right)^{\frac{1}{q}}}
{(\log(n+1))^\sigma}
=0.
$$

Conversely, let $\textbf{z}\in X$ satisfy
$$
\lim\limits_{n\to\infty}
\frac{\left(\sum\limits_{j=1}^{n}(z_j^*)^q\right)^{\frac{1}{q}}}
{(\log(n+1))^\sigma}
=0.
$$
Fix $r>0$, and let $C$ be the constant in
\namedref{Lemma}{yinli-chongpaixishuguji}. By \eqref{shi-weibudengjia}, choose
$N\in\mathbb{N}$ such that
$
CM(Q_N\textbf{z})<\frac{r}{2}.
$
For $G\in\mathscr{P}_N$, write
$$
G=\sum\limits_{m=0}^{d}H_m,
\qquad
H_m\in\mathcal{P}^m(X)\cap\mathscr{P}_N,
\quad 0\leqslant m\leqslant d,
$$
for its homogeneous decomposition. By \eqref{shi-qicixishuyizhi},
\eqref{shi-qicikexiguji}, and \eqref{shi-chongpaixishuguji},
$$
\begin{aligned}
	&\sum\limits_\alpha|c_\alpha(G)\textbf{z}^\alpha|
	=\sum\limits_{m=0}^{d}
	\sum\limits_{|\alpha|=m}
	|c_\alpha(H_m)(Q_N\textbf{z})^\alpha|\\
	&\leqslant
	\|G\|_r
	\sum\limits_{m=0}^{d}
	\left(\frac{CM(Q_N\textbf{z})}{r}\right)^m\leqslant
	\|G\|_r\sum\limits_{m=0}^{\infty}\frac{1}{2^m}
	=2\|G\|_r.
\end{aligned}
$$
Thus \eqref{shi-weibupanju} holds for every $r>0$, and
\namedref{Proposition}{mingti-jubuhua}, with $O=X$, yields
$
\textbf{z}\in\mathrm{dom}H(X).
$
Therefore \eqref{shi-jingdiantongyikehua} holds, and
\eqref{shi-xiaozhishushoulian}, \eqref{shi-dazhishushoulian}, and \eqref{shi-lingjixianshoulian} follow from
\eqref{shi-fenleizhishu}.
\end{proof}
The characterization of $\mathrm{dom}H_b(\ell^\infty)$
was established in
\cite[Corollary~6.3.3, p.~98]{Mansilla19};
see also \cite[p.~947]{GMMS21}.
Moreover, the equality
$\mathrm{dom}H_b(c_0)=\mathrm{dom}H_b(\ell^\infty)$
was recorded in
\cite[Appendix~A, (A.8), p.~135]{Mansilla19}.
The following theorem shows that this common domain
also coincides with $\mathrm{dom}H(\ell^\infty)$.
\begin{theorem}\label{dingli-youjiexulieshoulian}
One has
\begin{equation}
\mathrm{dom}H(\ell^\infty)
=
\left\{\textbf{z}\in c_0:
\sup\limits_{n\geqslant1}
\frac{\sum\limits_{j=1}^{n}(z_j^*)^2}{\log(n+1)}<\infty\right\}.
\label{shi-youjiexulieshoulian}
\end{equation}
Moreover,
$$
\mathrm{dom}H(\ell^\infty)
=\mathrm{dom}H_b(c_0).
$$
\end{theorem}

\begin{proof}

Aron and Berner \cite[Proposition~1.1, p.~8]{AB78} prove that \begin{equation}
\{F|_{c_0}:F\in H(\ell^\infty)\}=H_b(c_0).
\label{shi-zhenghanshuxianzhi}
\end{equation}

Since
$\overline{\mathbb{C}^{(\mathbb{N})}}^{\,\ell^\infty}=c_0$,
\eqref{shi-bibaobaohan} gives
\begin{equation}
\mathrm{dom}H(\ell^\infty)\subset c_0.
\label{shi-shoulianlingjixian}
\end{equation}
Restriction from $\ell^\infty$ to $c_0$ leaves each $c_\alpha$
unchanged, because the finite-dimensional restrictions defining it
coincide. Together with \eqref{shi-shoulianlingjixian} and
\eqref{shi-youjiexingshoulian} for $p=\infty$,
this gives \eqref{shi-youjiexulieshoulian}.
\end{proof}

\begin{corollary}\label{tuilun-youjiexuliedengjia}
One has
\begin{equation}
\mathrm{dom}H_b(\ell^\infty)
=\mathrm{dom}H(\ell^\infty)
=\mathrm{dom}H_b(c_0).
\label{shi-sangeshoulianyudengjia}
\end{equation}
In particular,
$\mathrm{dom}H_b(\ell^\infty)$ is the set on the right-hand side of
\eqref{shi-youjiexulieshoulian}.
\end{corollary}

\begin{proof}
Since $H_b(\ell^\infty)\subset H(\ell^\infty)$, the definition
\eqref{shi-shoulianyu} gives
$$
\mathrm{dom}H(\ell^\infty)
\subseteq\mathrm{dom}H_b(\ell^\infty).
$$
To prove the reverse inclusion, fix
$\textbf{z}\in\mathrm{dom}H_b(\ell^\infty)$.
Every continuous quadratic polynomial on $\ell^\infty$ belongs to
$H_b(\ell^\infty)$, so \namedref{Theorem}{dingli-ercizhangai} implies
$
\textbf{z}\in
\overline{\mathbb{C}^{(\mathbb{N})}}^{\,\ell^\infty}=c_0.
$
By \namedref{Lemma}{yinli-bituguji},  there is $K\geqslant1$ such that
\begin{equation}
\sum\limits_{|\alpha|=m}|c_\alpha(P)\textbf{z}^\alpha|
\leqslant K^m\|P\|_{\mathcal{P}^m(\ell^\infty)},
\qquad P\in\mathcal{P}^m(\ell^\infty),\quad m\in\mathbb{N}.
\label{shi-youjiexuliexishuguji}
\end{equation}
The constant $K$ is independent of $m$ and $P$.

For $m,n\in\mathbb{N}$ and $G\in\mathcal{P}^m(c_0)$ satisfying
$G=G\circ P_n$, define
$$
\widetilde G(\textbf{x})\triangleq G(P_n\textbf{x}),
\qquad\textbf{x}\in\ell^\infty.
$$
Then $\widetilde G\in\mathcal{P}^m(\ell^\infty)$ and
\begin{equation}
\|\widetilde G\|_{\mathcal{P}^m(\ell^\infty)}
=\|G\|_{\mathcal{P}^m(c_0)},
\qquad
c_\alpha(\widetilde G)=c_\alpha(G),\qquad \alpha\in\mathbb{N}_0^{(\mathbb{N})}.
\label{shi-youxianzuobiaoyantuo}
\end{equation}
Applying
\eqref{shi-youjiexuliexishuguji} to $\widetilde G$
and using \eqref{shi-youxianzuobiaoyantuo}, we obtain
$$
\sum\limits_{|\alpha|=m}|c_\alpha(G)\textbf{z}^\alpha|
\leqslant K^m\|G\|_{\mathcal{P}^m(c_0)}
$$
for every continuous homogeneous polynomial on $c_0$ depending on
finitely many coordinates. Thus \namedref{Lemma}{yinli-shangjixianbijiao} yields
$$
\limsup\limits_{n\to\infty}
\frac{\left(\sum\limits_{j=1}^{n}(z_j^*)^2\right)^{\frac{1}{2}}}
{(\log(n+1))^{\frac{1}{2}}}
\leqslant K.$$
Together with $\textbf{z}\in c_0$,
\namedref{Theorem}{dingli-youjiexulieshoulian} therefore gives
$\textbf{z}\in\mathrm{dom}H(\ell^\infty)$.
This proves the reverse inclusion
$
\mathrm{dom}H_b(\ell^\infty)
\subseteq
\mathrm{dom}H(\ell^\infty).
$
The remaining equality in \eqref{shi-sangeshoulianyudengjia}
follows from \namedref{Theorem}{dingli-youjiexulieshoulian}.
\end{proof}

Combining \namedref{Proposition}{mingti-jubuhua},
\namedref{Theorem}{dingli-jingdianshoulian}, and
\namedref{Proposition}{mingti-jishubiaoshi}, we obtain the following corollary.
\begin{corollary}\label{tuilun-jingdianquanyu}
Let $X=\ell^p$ with $1\leqslant p<\infty$, or let $X=c_0$,
and let $O\subset X$ be a Reinhardt domain. If $X=\ell^1$, then
\begin{equation}
\mathrm{dom}H(O)=O.
\label{shi-yikehequanyu}
\end{equation}
If $X=\ell^p$ with $1<p\leqslant2$, then
\begin{equation}
\mathrm{dom}H(O)
=
\left\{\textbf{z}\in O:
\lim\limits_{n\to\infty}
\frac{\sum\limits_{j=1}^{n}z_j^*}
{(\log(n+1))^{1-\frac{1}{p}}}=0\right\}.
\label{shi-xiaozhishuquanyu}
\end{equation}
If $X=\ell^p$ with $2<p<\infty$, then
\begin{equation}
\mathrm{dom}H(O)
=
\left\{\textbf{z}\in O:
\lim\limits_{n\to\infty}
\frac{\sum\limits_{j=1}^{n}(z_j^*)^{\frac{2p}{p+2}}}
{(\log(n+1))^{\frac{p}{p+2}}}=0\right\}.
\label{shi-dazhishuquanyu}
\end{equation}
If $X=c_0$, then
\begin{equation}
\mathrm{dom}H(O)
=
\left\{\textbf{z}\in O:
\lim\limits_{n\to\infty}
\frac{\sum\limits_{j=1}^{n}(z_j^*)^2}{\log(n+1)}=0\right\}.
\label{shi-lingjixianquanyu}
\end{equation}
In each case, for every $f\in H(O)$ and
$\textbf{z}\in\mathrm{dom}H(O)$,
\eqref{shi-juduishoulian} and \eqref{shi-hanshubiaoshi} hold.
\end{corollary}

\subsection{Holomorphic functions on the unit ball of \texorpdfstring{$\ell^\infty$}{ell-infinity}}

\begin{proposition}\label{mingti-youjiexulieqiu}
One has
\begin{equation}
\mathrm{dom}H(B_{\ell^\infty})
=
\left\{\textbf{z}\in B_{c_0}:
\limsup\limits_{n\to\infty}
\frac{\sum\limits_{j=1}^{n}(z_j^*)^2}{\log(n+1)}<1\right\}.
\label{shi-youjiexulieqiu}
\end{equation}
Moreover, for every $f\in H(B_{\ell^\infty})$ and every
$\textbf{z}$ in the set on the right of
\eqref{shi-youjiexulieqiu},
\eqref{shi-juduishoulian} and \eqref{shi-hanshubiaoshi} hold.
\end{proposition}

\begin{proof}

\proofstep{1}

Since
$\overline{\mathbb{C}^{(\mathbb{N})}}^{\,\ell^\infty}=c_0$,
\eqref{shi-bibaobaohan}, applied to
$O=B_{\ell^\infty}$, gives
\begin{equation}
\mathrm{dom}H(B_{\ell^\infty})\subset B_{c_0}.
\label{shi-qiushoulianlingjixian}
\end{equation}
For every $f\in H(B_{\ell^\infty})$, we have
\begin{equation}
c_\alpha(f|_{B_{c_0}})=c_\alpha(f),
\qquad \alpha\in\mathbb{N}_0^{(\mathbb{N})}.
\label{shi-xianzhixishuyizhi}
\end{equation}
Consequently, whenever the monomial series of $f$ is absolutely
convergent at $\textbf{z}\in B_{c_0}$,
\namedref{Proposition}{mingti-jishubiaoshi} identifies its sum with $f(\textbf{z})$
by \eqref{shi-hanshubiaoshi}. 

By \cite[Theorem~2.2\,(a)]{BDFMS17}, together with
\cite[(14) and (16)]{BDFMS17}, we have, for
$X\in\{c_0,\ell^\infty\}$ and $\textbf{w}\in B_{c_0}$,
\begin{equation}
	\limsup\limits_{n\to\infty}
	\frac{\sum\limits_{j=1}^{n}(w_j^*)^2}{\log(n+1)}<1
	\quad\Longrightarrow\quad
	\textbf{w}\in\mathrm{dom}H^\infty(B_X).
	\label{shi-yuzhichongfentiaojian}
\end{equation}

\proofstep{2}

Let $\textbf{z}\in\mathrm{dom}H(B_{\ell^\infty})$. By \eqref{shi-qiushoulianlingjixian}, $\textbf{z}\in B_{c_0}$. Since $H_b(B_{\ell^\infty})\subset H(B_{\ell^\infty})$, \namedref{Lemma}{yinli-bituguji} gives $A>0$ and $0<\rho<1$ such that
\begin{equation}
\sum\limits_{\alpha\in\mathbb{N}_0^{(\mathbb{N})}}
|c_\alpha(f)\textbf{z}^\alpha|
\leqslant A\|f\|_\rho,\qquad f\in H_b(B_{\ell^\infty}).
\label{shi-qiubituguji}
\end{equation}
Fix $m\geqslant2$, $n\in\mathbb{N}$, and $G\in\mathcal{P}^m(c_0)$ with $G=G\circ P_n$. Put $\widetilde G(\textbf{x})\triangleq G(P_n\textbf{x})$ for $\textbf{x}\in\ell^\infty$. Then $\widetilde G|_{B_{\ell^\infty}}\in H_b(B_{\ell^\infty})$. By \eqref{shi-youxianzuobiaoyantuo}, \eqref{shi-qiubituguji}, and homogeneity,
\begin{equation}
\begin{aligned}
\sum\limits_{|\alpha|=m}|c_\alpha(G)\textbf{z}^\alpha|
&=\sum\limits_{|\alpha|=m}|c_\alpha(\widetilde G)\textbf{z}^\alpha|
\leqslant A\|\widetilde G\|_\rho\\
&=A\rho^m\|\widetilde G\|_{\mathcal{P}^m(\ell^\infty)}
=A\rho^m\|G\|_{\mathcal{P}^m(c_0)}.
\end{aligned}
\label{shi-qiucishujie}
\end{equation}
By \eqref{shi-qiucishujie} and \namedref{Lemma}{yinli-shangjixianbijiao},
$$
\limsup\limits_{n\to\infty}
\frac{\sum\limits_{j=1}^{n}(z_j^*)^2}{\log(n+1)}
\leqslant\rho^2<1.
$$
This proves the necessary condition in \eqref{shi-youjiexulieqiu}.

\proofstep{3}

By \cite[p.~455]{GMZ12}, a
function in $H(B_{c_0})$ extends holomorphically to $B_{\ell^\infty}$
if and only if it belongs to $H_b(B_{c_0})$. Therefore, for each $f\in H(B_{\ell^\infty})$ and $0<r<1$, we have
\begin{equation}
\sup\limits_{\textbf{w}\in rB_{c_0}}|f(\textbf{w})|<\infty.
\label{shi-xianzhijubuyoujie}
\end{equation}
Let $\textbf{z}\in B_{c_0}$ satisfy the upper-limit condition in
\eqref{shi-youjiexulieqiu}. Choose $r$ such that
\begin{equation}
\|\textbf{z}\|_{\ell^\infty}<r<1,
\qquad
\limsup\limits_{n\to\infty}
\frac{\sum\limits_{j=1}^{n}(z_j^*)^2}{\log(n+1)}<r^2.
\label{shi-suofangbanjing}
\end{equation}
Fix $f\in H(B_{\ell^\infty})$, and define
$$
h(\textbf{w})\triangleq f(r\textbf{w}),\qquad\textbf{w}\in B_{c_0}.
$$
By \eqref{shi-xianzhijubuyoujie},
$h\in H^\infty(B_{c_0})$.
It is easy to see that
\begin{equation}
c_\alpha(h)=r^{|\alpha|}c_\alpha(f),
\qquad\alpha\in\mathbb{N}_0^{(\mathbb{N})}.
\label{shi-suofangxishu}
\end{equation}
Also, \eqref{shi-suofangbanjing} implies
$\frac{\textbf{z}}{r}\in B_{c_0}$ and
$$
\limsup\limits_{n\to\infty}
\frac{\sum\limits_{j=1}^{n}\left(\frac{z_j^*}{r}\right)^2}
{\log(n+1)}
=\frac{1}{r^2}\limsup\limits_{n\to\infty}
\frac{\sum\limits_{j=1}^{n}(z_j^*)^2}{\log(n+1)}<1.
$$
By \eqref{shi-yuzhichongfentiaojian},
$\frac{\textbf{z}}{r}\in\mathrm{dom}H^\infty(B_{c_0})$.
Consequently, \eqref{shi-suofangxishu} yields
$$
\sum\limits_{\alpha\in\mathbb{N}_0^{(\mathbb{N})}}
|c_\alpha(f)\textbf{z}^\alpha|
=\sum\limits_{\alpha\in\mathbb{N}_0^{(\mathbb{N})}}
\left|c_\alpha(h)\left(\frac{\textbf{z}}{r}\right)^\alpha\right|
<\infty.
$$
Since $f$ was arbitrary, this proves the reverse inclusion in
\eqref{shi-youjiexulieqiu}. Equation
\eqref{shi-xianzhixishuyizhi} and
\namedref{Proposition}{mingti-jishubiaoshi}, applied to
$f|_{B_{c_0}}$ at $\textbf{z}$, now give
\eqref{shi-hanshubiaoshi}.
\end{proof}

\begin{remark}\label{zhuji-linjieqingxing}
The strict inequality in \eqref{shi-youjiexulieqiu}
characterizes $\mathrm{dom}H(B_{\ell^\infty})$.
For bounded holomorphic functions, however, a complete characterization
is still not known. The available result is due to
\cite[Theorem~2.2\,(a)--(b)]{BDFMS17}:
$$
	\begin{aligned}
		\left\{\textbf{z}\in B_{c_0}:
		\limsup\limits_{n\to\infty}
		\frac{\sum\limits_{j=1}^{n}(z_j^*)^2}{\log(n+1)}<1\right\}
		&\subseteq\mathrm{dom}H^\infty(B_{\ell^\infty})\\
		&\subsetneq
		\left\{\textbf{z}\in B_{c_0}:
		\limsup\limits_{n\to\infty}
		\frac{\sum\limits_{j=1}^{n}(z_j^*)^2}{\log(n+1)}\leqslant1\right\}.
	\end{aligned}
$$
\end{remark}

\begin{remark}\label{zhuji-jubuhuashixiao}
It is easy to see that
$$
\mathrm{dom}H(O)
\subset O\cap\mathrm{dom}H(\ell^\infty).
$$
However, the equality \eqref{shi-jubuhua} cannot be
applied with $X=\ell^\infty$, since its hypothesis fails:
$\overline{\mathbb{C}^{(\mathbb{N})}}^{\,\ell^\infty}
=c_0\neq\ell^\infty$. Equations~\eqref{shi-youjiexulieshoulian} and~\eqref{shi-youjiexulieqiu} also provide a concrete example showing that \eqref{shi-jubuhua} cannot be applied when $X=\ell^\infty$.
\end{remark}

\subsection{Sharp comparison of convergence radii on \texorpdfstring{$\ell^1$}{ell-1}}

For $m\in\mathbb{N}_0$ and $P\in\mathcal{P}^m(\ell^1)$,
we extend the notation in \eqref{shi-zhixishuduoxiangshi} by writing
$$
P^{\#}(\textbf{z})\triangleq
\sum\limits_{\substack{\alpha\in\mathbb{N}_0^{(\mathbb{N})}\\|\alpha|=m}}
|c_\alpha(P)|\textbf{z}^{\alpha},\qquad\|P^{\#}\|=
\sup\limits_{\substack{\textbf{x}\in\ell^1\cap[0,+\infty)^{\mathbb{N}}\\
		\|\textbf{x}\|_{\ell^1}\leqslant1}}
\sum\limits_{|\alpha|=m}|c_\alpha(P)|\textbf{x}^{\alpha}.
$$By \eqref{shi-yikeheshoulian}, $P^{\#}(\textbf{z})$ is well defined
for every $\textbf{z}\in\ell^1$.

Polynomial norms are those of \eqref{shi-duoxiangshifanshu},
with ambient space $\ell^1$. A polynomial on $\mathbb{C}^n$
is identified with the corresponding polynomial on $\ell^1$
depending on its first $n$ coordinates.

If $f$ is holomorphic in a neighbourhood of $\textbf{0}$ in
$\ell^1$, let $H_m$ be its homogeneous Taylor polynomials, as in
\eqref{shi-qicitaile}:
$$
H_0(\textbf{z})\triangleq f(\textbf{0}),
\qquad
H_m(\textbf{z})\triangleq
\frac{1}{m!}D^mf(\textbf{0})[\textbf{z},\cdots,\textbf{z}],
\qquad m\geqslant1.
$$
Following \cite[Section~3]{BRS25}, set
\begin{equation}
	\begin{aligned}
		r(f,\textbf{0})
		&\triangleq
		\left(\limsup\limits_{m\to\infty}\|H_m\|^{\frac{1}{m}}\right)^{-1},\\
		|r|(f,\textbf{0})
		&\triangleq
		\left(\limsup\limits_{m\to\infty}\|H_m^{\#}\|^{\frac{1}{m}}\right)^{-1}.
	\end{aligned}
	\label{shi-l1-liangzhongbanjing}
\end{equation}

As a further application of our methods, we determine the optimal constant in the comparison between the radii of regular convergence and convergence of homogeneous Taylor expansions on $\ell^1$. In \cite[p.~5919]{BRS25}, the authors established
$$
e^{-1}r(f,\textbf{0})\leqslant |r|(f,\textbf{0})\leqslant r(f,\textbf{0})
$$
and left open whether the constant $e^{-1}$ is optimal. We resolve this question by determining the sharp lower bound.
\begin{theorem}\label{dingli-l1-zuijiabanjing}
	Let $t\in\left(\frac{1}{2},1\right)$ be the unique solution of
	\begin{equation}
		-\log(1-t)=2t.
		\label{shi-l1-zuijiacanshu}
	\end{equation}
	Every function $f$ holomorphic in a neighbourhood of $\textbf{0}$
	in  $\ell^1$ satisfies
	\begin{equation}
		\sqrt{t(1-t)}\,r(f,\textbf{0})
		\leqslant |r|(f,\textbf{0})
		\leqslant r(f,\textbf{0}).
		\label{shi-l1-zuijiabanjing}
	\end{equation}
	Both constants are optimal. 
\end{theorem}

\begin{proof}
By $-\log(1-t)=2t$ and
$$
\sup\limits_{s>0}\frac{\log(1+s^2)}{2s}
=\frac{-\log(1-t)}{2\sqrt{\frac{t}{1-t}}}
=\sqrt{t(1-t)},
$$
we have
	\begin{equation}
		1+s^2\leqslant e^{2s\sqrt{t(1-t)}},
		\qquad s\geqslant0.
		\label{shi-l1-biaoliangzhishijie}
	\end{equation}
	We shall also use
	\begin{equation}
b^{2a}\geqslant a!b^{a}\geqslant a!b\left( b-1 \right) \cdots \left( b-a+1 \right) =\left( a! \right)^{2} \binom{b}{a},
		\qquad a,b\in\mathbb{N}_0,b\geqslant a.
		\label{shi-l1-zhengshujuguji}
	\end{equation}
	Finally, Stirling's formula and
	\eqref{shi-l1-zuijiacanshu} give
	\begin{equation}
\begin{aligned}
	&\begin{aligned}\lim\limits_{m\to \infty} \frac{\binom{m^{2}}{m}^{\frac{1}{m}}}{4^{\frac{1}{m}}e\, m\sqrt{t(1-t)} \left( \log 8+2m^{2}\log (1+24m) \right)^{\frac{1}{2m}}}&\end{aligned}\\ &\  =\lim\limits_{m\to \infty} \frac{\left\lceil \frac{m}{t} \right\rceil}{\left( m! \right)^{\frac{1}{m}} \binom{\left\lceil \frac{m}{t} \right\rceil}{m}^{\frac{1}{2m}}} =\frac{1}{\sqrt{t\left( 1-t \right)}}.\end{aligned}
		\label{shi-l1-caiyanggenjixian}
	\end{equation}

	\proofstep{1}
	
	We prove that, for all integers $k\geqslant m\geqslant2$,
	\begin{equation}
		\|P\|\leqslant\|P^{\#}\|
		\leqslant
		\frac{k^m}{m!\binom{k}{m}^{\frac{1}{2}}}\|P\|,
		\qquad P\in\mathcal{P}^m(\ell^1).
		\label{shi-l1-jingquexishushangjie}
	\end{equation}
The lower bound in \eqref{shi-l1-jingquexishushangjie} is immediate from
$$
|P(\textbf{z})|
\leqslant P^{\#}(|\textbf{z}|),
\qquad \textbf{z}\in B_{\ell^1}.
$$
To prove the upper bound, by the same truncation argument as in
\namedref{Lemma}{yinli-gujituiguang}, it suffices to consider
$P\in\mathcal{P}^m(\ell^1)$ satisfying $P=P\circ P_n$ for some
$n\in\mathbb{N}$. Fix
	$\textbf{x}=(x_1,\cdots,x_n)\in[0,+\infty)^n$ with
	$\sum\limits_{j=1}^n x_j=1$.
	For every $\beta\in\mathbb{N}_0^n$ with $|\beta|=k$, 
\begin{equation}
	\begin{aligned}
		&\frac{1}{k^{2m}}
		\sum\limits_{\substack{\alpha\in\mathbb{N}_0^n\\|\alpha|=m}}
		|c_\alpha(P)|^2
		\prod\limits_{j=1}^n\beta_j^{2\alpha_j}=
		\sum\limits_{\substack{\alpha\in\mathbb{N}_0^n\\|\alpha|=m}}
		\left|
		c_\alpha(P)
		\prod\limits_{j=1}^n
		\left(\frac{\beta_j}{k}\right)^{\alpha_j}
		\right|^2\\
		&\quad=
		\frac{1}{(2\pi)^n}
		\int\limits_{[0,2\pi]^n}
		\left|
		\sum\limits_{\substack{\alpha\in\mathbb{N}_0^n\\|\alpha|=m}}
		c_\alpha(P)
		\prod\limits_{j=1}^n
		\left(
		\frac{\beta_j}{k}e^{\mathrm{i}\theta_j}
		\right)^{\alpha_j}
		\right|^2
		\,\mathrm{d}\theta_1\cdots\mathrm{d}\theta_n\\
		&\quad=
		\frac{1}{(2\pi)^n}
		\int\limits_{[0,2\pi]^n}
		\left|
		P\left(
		\sum\limits_{j=1}^n
		\frac{\beta_j}{k}e^{\mathrm{i}\theta_j}e_j
		\right)
		\right|^2
		\,\mathrm{d}\theta_1\cdots\mathrm{d}\theta_n\\
		&\quad\leqslant
		\frac{1}{(2\pi)^n}
		\int\limits_{[0,2\pi]^n}
		\|P\|^2
		\left\|
		\sum\limits_{j=1}^n
		\frac{\beta_j}{k}e^{\mathrm{i}\theta_j}e_j
		\right\|_{\ell^1}^{2m}
		\,\mathrm{d}\theta_1\cdots\mathrm{d}\theta_n\\
		&\quad=
		\frac{1}{(2\pi)^n}
		\int\limits_{[0,2\pi]^n}
		\|P\|^2
		\left(
		\sum\limits_{j=1}^n\frac{\beta_j}{k}
		\right)^{2m}
		\,\mathrm{d}\theta_1\cdots\mathrm{d}\theta_n\\
		&\quad=
		\frac{1}{(2\pi)^n}
		\int\limits_{[0,2\pi]^n}
		\|P\|^2
		\left(\frac{|\beta|}{k}\right)^{2m}
		\,\mathrm{d}\theta_1\cdots\mathrm{d}\theta_n
		=\|P\|^2,
	\end{aligned}\label{shi-714}
\end{equation}
	where the second equality follows from Parseval's identity.	For $|\alpha|=m$, inequality
	\eqref{shi-l1-zhengshujuguji} implies
\begin{equation}
	\begin{aligned}
		&\sum\limits_{\substack{\beta\in\mathbb{N}_0^n\\|\beta|=k}}
		\frac{k!}{\beta!}\textbf{x}^{\beta}
		\prod\limits_{j=1}^n\beta_j^{2\alpha_j}\geqslant
		(\alpha!)^2
		\sum\limits_{\substack{\beta\in\mathbb{N}_0^n\\
				|\beta|=k,\ \beta\geqslant\alpha}}
		\frac{k!}{\beta!}\textbf{x}^{\beta}
		\prod\limits_{j=1}^n\binom{\beta_j}{\alpha_j}\\
		&\quad=
		\alpha!k!\textbf{x}^{\alpha}
		\sum\limits_{\substack{\gamma\in\mathbb{N}_0^n\\|\gamma|=k-m}}
		\frac{\textbf{x}^{\gamma}}{\gamma!}
		=
		\alpha!\frac{k!}{(k-m)!}\textbf{x}^{\alpha}.\label{shi-713}
	\end{aligned}
\end{equation}
	Here $\beta\geqslant\alpha$ is understood $\beta_{j} \geqslant \alpha_{j} ,1\leqslant j\leqslant n$.
	
By \eqref{shi-713}, \eqref{shi-714} and $\sum\limits_{\substack{\beta\in\mathbb{N}_0^n\\|\beta|=k}}
\frac{k!}{\beta!}\textbf{x}^{\beta}=\| \textbf{x} \|_{\ell^{{}1}}^{k} =1$, we have
	$$
\begin{aligned}
	\|P\|^2
	&=
	\|P\|^2
	\sum\limits_{\substack{\beta\in\mathbb{N}_0^n\\|\beta|=k}}
	\frac{k!}{\beta!}\textbf{x}^{\beta}\geqslant
	\frac{1}{k^{2m}}
	\sum\limits_{\substack{\beta\in\mathbb{N}_0^n\\|\beta|=k}}
	\frac{k!}{\beta!}\textbf{x}^{\beta}
	\sum\limits_{\substack{\alpha\in\mathbb{N}_0^n\\|\alpha|=m}}
	|c_\alpha(P)|^2
	\prod\limits_{j=1}^n\beta_j^{2\alpha_j}\\
	&=
	\frac{1}{k^{2m}}
	\sum\limits_{\substack{\alpha\in\mathbb{N}_0^n\\|\alpha|=m}}
	|c_\alpha(P)|^2
	\sum\limits_{\substack{\beta\in\mathbb{N}_0^n\\|\beta|=k}}
	\frac{k!}{\beta!}\textbf{x}^{\beta}
	\prod\limits_{j=1}^n\beta_j^{2\alpha_j}\geqslant
	\frac{k!}{k^{2m}(k-m)!}
	\sum\limits_{\substack{\alpha\in\mathbb{N}_0^n\\|\alpha|=m}}
	\alpha!|c_\alpha(P)|^2\textbf{x}^{\alpha},
\end{aligned}
	$$
Then the Cauchy--Schwarz inequality gives
	$$
	\begin{aligned}
		\bigl(P^{\#}(\textbf{x})\bigr)^2
		&\leqslant
		\left(
		\sum\limits_{\substack{\alpha\in\mathbb{N}_0^n\\|\alpha|=m}}
		\alpha!|c_\alpha(P)|^2\textbf{x}^{\alpha}
		\right)
		\left(
		\sum\limits_{\substack{\alpha\in\mathbb{N}_0^n\\|\alpha|=m}}
		\frac{\textbf{x}^{\alpha}}{\alpha!}
		\right)\leqslant
		\frac{k^{2m}(k-m)!}{m!k!}\|P\|^2,
	\end{aligned}
	$$
	which gives the upper
	bound in \eqref{shi-l1-jingquexishushangjie}.

	\proofstep{2}
	
	Let $f$ be holomorphic in a neighbourhood of $\textbf{0}$,
	and let $H_m$ be its homogeneous Taylor polynomials.
	By local boundedness of $f$ and \eqref{shi-qicikexiguji},
	$$
	\limsup\limits_{m\to\infty}\|H_m\|^{\frac{1}{m}}<+\infty.
	$$
	Apply \eqref{shi-l1-jingquexishushangjie} to $H_m$ with
	$k=\left\lceil\frac{m}{t}\right\rceil$ and using
	\eqref{shi-l1-caiyanggenjixian}, gives
	$$
	\begin{aligned}
		\limsup\limits_{m\to\infty}\|H_m\|^{\frac{1}{m}}
		&\leqslant
		\limsup\limits_{m\to\infty}\|H_m^{\#}\|^{\frac{1}{m}}\leqslant
		\frac{1}{\sqrt{t(1-t)}}
		\limsup\limits_{m\to\infty}\|H_m\|^{\frac{1}{m}}.
	\end{aligned}
	$$
Thus we have completed the proof of \eqref{shi-l1-zuijiabanjing}.
	
	\proofstep{3}
	
To prove the optimality of the lower bound, fix $ m\geqslant2$.
	For $\textbf{u}=\left( u_{i} \right)_{1\leqslant i\leqslant m^{2}}\in\mathbb{C}^{m^{2}}$ with
	$\|\textbf{u}\|_{\ell^1}\leqslant1$,
	\eqref{shi-l1-biaoliangzhishijie} give
	$$
	\begin{aligned}
		&\left(\frac{m^2}{t(1-t)}\right)^m
		\sum\limits_{\substack{F\subset\{1,\cdots,m^{2}\}\\|F|=m}}
		\prod\limits_{j\in F}|u_j|^2\leqslant
		\prod\limits_{j=1}^{m^{2}}
		\left(1+\frac{m^2|u_j|^2}{t(1-t)}\right)\leqslant
		e^{2m\sum\limits_{j=1}^{m^{2}}|u_j|}
		\leqslant e^{2m}.
	\end{aligned}
	$$
	Consequently,
	\begin{equation}
		(m!)^2
		\sum\limits_{\substack{F\subset\{1,\cdots,m^{2}\}\\|F|=m}}
		\prod\limits_{j\in F}|u_j|^2
		\leqslant
		(m!)^2
		\left(\frac{e\sqrt{t(1-t)}}{m}\right)^{2m}.
		\label{shi-l1-suijifangcha}
	\end{equation}
	
	Apply \namedref{Lemma}{yinli-duoxiangshiwang} and let $\mathcal{N}$ be the resulting finite subset
	of the closed $\ell^2$ unit ball in $\mathbb{C}^{m^{2}}$.
	For every $m$-homogeneous polynomial $P$ on $\mathbb{C}^{m^{2}}$,
	the polynomial
	$\textbf{u}\mapsto P(u_1^2,\cdots,u_{m^{2}}^2)$
	is $2m$-homogeneous. Moreover,
	$$
	\left\{
	(u_1^2,\cdots,u_{m^{2}}^2):
	\textbf{u}\in\mathbb{C}^{m^{2}},\
	\|\textbf{u}\|_{\ell^2}\leqslant1
	\right\}
	=
	\left\{
	\textbf{v}\in\mathbb{C}^{m^{2}}:
	\|\textbf{v}\|_{\ell^1}\leqslant1
	\right\}.
	$$

	The cited lemma therefore gives
	\begin{equation}
		\begin{gathered}
			|\mathcal{N}|\leqslant(1+24m)^{2m^{2}},\qquad	\|P\|
			=
			\sup\limits_{\|\textbf{u}\|_{\ell^2}\leqslant1}
			|P(u_1^2,\cdots,u_{m^{2}}^2)|
			\leqslant
			2\max\limits_{\textbf{u}\in\mathcal{N}}
			|P(u_1^2,\cdots,u_{m^{2}}^2)|.
		\end{gathered}
		\label{shi-l1-danchunxingwang}
	\end{equation}
	
	Let
	$
	(\varepsilon_F)_{\substack{F\subset\{1,\cdots,m^{2}\}\\|F|=m}}
	$
	be independent random variables satisfying
	$$
	\mathbb{P}(\varepsilon_F=1)
	=
	\mathbb{P}(\varepsilon_F=-1)
	=
	\frac{1}{2},
	\qquad
	F\subset\{1,\cdots,m^{2}\},\quad |F|=m.
	$$and put
	$$
	R_m(\textbf{u})\triangleq
	m!\sum\limits_{\substack{F\subset\{1,\cdots,m^{2}\}\\|F|=m}}
	\varepsilon_F\prod\limits_{j\in F}u_j.
	$$
	By \eqref{shi-suijiweigailv} and 
	\eqref{shi-l1-suijifangcha}, for each $\|\textbf{u}\|_{\ell^1}\leqslant1,\lambda>0$,
	$$
	\begin{aligned}
		\mathbb{P}\bigl(|R_m(\textbf{u})|>\lambda\bigr)
		&\leqslant
		\mathbb{P}\left(
		|\mathrm{Re}\,R_m(\textbf{u})|>\frac{\lambda}{\sqrt{2}}
		\right)
		+
		\mathbb{P}\left(
		|\mathrm{Im}\,R_m(\textbf{u})|>\frac{\lambda}{\sqrt{2}}
		\right)\\
		&\leqslant
		2e^{
			-\frac{(\frac{\lambda}{\sqrt{2}})^2}
			{2	\sum\limits_{\substack{F\subset\{1,\cdots,m^{2}\}\\|F|=m}}
	\left(
				\mathrm{Re}\left(
				m!\prod\limits_{j\in F}u_j
				\right)
				\right)^2}
		}
		+
		2e^{
			-\frac{(\frac{\lambda}{\sqrt{2}})^2}
			{2	\sum\limits_{\substack{F\subset\{1,\cdots,m^{2}\}\\|F|=m}}
				\left(
				\mathrm{Im}\left(
				m!\prod\limits_{j\in F}u_j
				\right)
				\right)^2}
		}\\
		&\leqslant
		4e^{-\frac{\lambda^2}
			{4(m!)^2
				\left(
				\frac{e\sqrt{t(1-t)}}{m}
				\right)^{2m}}}.
	\end{aligned}
	$$
	
	Take
	$$
	\lambda\triangleq
	2m!\left(\frac{e\sqrt{t(1-t)}}{m}\right)^m
	\sqrt{\log8+2m^{2}\log(1+24m)}.
	$$
Hence, by \eqref{shi-l1-danchunxingwang}, 
$$
\begin{aligned}\mathbb{P} \left( \frac{\| R_m\|}{2} >\lambda \right)&\leqslant \mathbb{P} \left( \max\limits_{\textbf{u} \in \mathcal{N}} |R_m(u_{1}^{2},\cdots ,u_{m^{2}}^{2})|>\lambda \right)\\ &\leqslant 4|\mathcal{N} |e^{-\log 8-2m^{2}\log (1+24m)}\\ &\leqslant 4(1+24m)^{2m^{2}}e^{-\log 8-2m^{2}\log (1+24m)}\\ &=\frac{1}{2} ,\end{aligned}
$$
	Thus there is a $	
	(\varepsilon_F)_{\substack{F\subset\{1,\cdots,m^{2}\}\\|F|=m}}
	$ such that
	$$
	\|R_m\|
	\leqslant
	4m!\left(\frac{e\sqrt{t(1-t)}}{m}\right)^m
	\sqrt{\log8+2m^{2}\log(1+24m)}.
	$$
	Note that
	$$
\| R_{m}^{\#}\| \geqslant R_{m}^{\#}\left( \left( \frac{1}{m^{2}} \right)_{1\leqslant i\leqslant m^{2}} \right) =\frac{m!\binom{m^{2}}{m}}{(m^{2})^{m}}.
	$$
	Together with \eqref{shi-l1-jingquexishushangjie}, this gives
	\begin{equation}
	\begin{aligned}\sqrt[m]{\frac{\frac{m!\binom{m^{2}}{m}}{(m^{2})^{m}}}{4m!\left( \frac{e\sqrt{t(1-t)}}{m} \right)^{m} \sqrt{\log 8+2m^{2}\log (1+24m)}}}&\leqslant \frac{\| R_{m}^{\#}\|^{\frac{1}{m}}}{\| R_{m}\|^{\frac{1}{m}}} \leqslant \frac{\left\lceil \frac{m}{t} \right\rceil}{\left( m! \right)^{\frac{1}{m}} \binom{\left\lceil \frac{m}{t} \right\rceil}{m}^{\frac{1}{2m}}} .\end{aligned}
		\label{shi-l1-jingqueliangbian}
	\end{equation}
Applying
	\eqref{shi-l1-caiyanggenjixian}, we conclude that
	\begin{equation}
		\lim\limits_{m\to\infty}
		\left(
		\frac{\|R_m^{\#}\|}{\|R_m\|}
		\right)^{\frac{1}{m}}
		=
		\frac{1}{\sqrt{t(1-t)}}.
		\label{shi-l1-jingquezengzhang}
	\end{equation}
	
	\proofstep{4}
	
	Put $n_1=0$ and $n_m=n_{m-1}+m^2$ for $m\geqslant2$,
	and define
	$$
	H_0=H_1=0,
	\qquad
	H_m(\textbf{z})\triangleq
	\frac{R_m(z_{n_{m-1}+1},\cdots,z_{n_m})}{\|R_m\|},
	\qquad m\geqslant2.
	$$
Note that
	$$
	\|H_m\|=1,
	\qquad
	\|H_m^{\#}\|=\frac{\|R_m^{\#}\|}{\|R_m\|},
	\qquad
	|H_m(\textbf{z})|
	\leqslant
	\|(P_{n_m}-P_{n_{m-1}})\textbf{z}\|_{\ell^1}^m.
	$$
	Set
	$$
	f(\textbf{z})\triangleq
	\sum\limits_{m=2}^{\infty}H_m(\textbf{z}).
	$$
	
	Fix $\textbf{x}\in\ell^1$. By \eqref{shi-weibushoulian},
	choose $M\geqslant2$ such that
	$\|Q_{n_{M-1}}\textbf{x}\|_{\ell^1}\leqslant\frac{1}{4}$.
	For $\|\textbf{y}-\textbf{x}\|_{\ell^1}<\frac{1}{4}$
	and $m\geqslant M$, 
$$
\begin{aligned}
	\|(P_{n_m}-P_{n_{m-1}})\textbf{y}\|_{\ell^1}
	&\leqslant
	\|(P_{n_m}-P_{n_{m-1}})\textbf{x}\|_{\ell^1}
	+\|(P_{n_m}-P_{n_{m-1}})(\textbf{y}-\textbf{x})\|_{\ell^1}\\
	&\leqslant
	\|Q_{n_{m-1}}\textbf{x}\|_{\ell^1}
	+\|\textbf{y}-\textbf{x}\|_{\ell^1}\leqslant
	\|Q_{n_{M-1}}\textbf{x}\|_{\ell^1}
	+\|\textbf{y}-\textbf{x}\|_{\ell^1}\leqslant\frac{1}{2}.
\end{aligned}
$$
	Hence $\sum\limits_{m=M}^{\infty} \left| H_{m}(\textbf{y} ) \right|\leqslant \sum\limits_{m=M}^{\infty} \frac{1}{2^{m}} <\infty$
	and \eqref{shi-quanchunjixian} gives $f\in H(\ell^1)$. One readily checks that  $H_m$ is the $m$-homogeneous Taylor polynomial of $f$ at $\textbf{0}$. By \eqref{shi-l1-jingquezengzhang},
	$$
	r(f,\textbf{0})=1,
	\qquad
	|r|(f,\textbf{0})=\sqrt{t(1-t)}.
	$$
	Thus the lower constant in \eqref{shi-l1-zuijiabanjing}
	cannot be increased.
	
		\proofstep{5}
		
To prove the optimality of the upper bound, consider
$$
g(\textbf{z})\triangleq
\sum\limits_{m=2}^{\infty}z_m^m\in H(\ell^1),
\qquad
\sup\limits_{\|\textbf{z}\|_{\ell^1}\leqslant1}
|g(\textbf{z})|\leqslant1.
$$
For $m\geqslant2$, the $m$-homogeneous Taylor polynomial of $g$
at $\textbf{0}$ is $H_m(\textbf{z})=z_m^m$. Hence
$
\|H_m\|=\|H_m^{\#}\|=1,
$
which gives
$
r(g,\textbf{0})=|r|(g,\textbf{0})=1.
$
This completes the proof of
\namedref{Theorem}{dingli-l1-zuijiabanjing}.
\end{proof}
\subsection{Monomial convergence of homogeneous polynomials}
As a further application of the methods developed above, we obtain a simple and complete characterization of the monomial convergence domains of continuous $m$-homogeneous polynomials on the classical sequence spaces, for each fixed $m\in\mathbb{N}$.

For $1<q<\infty$ and $1\leqslant r<\infty$, write
$$
\begin{aligned}
	\ell_{q,r}
	&\triangleq
	\left\{
	\textbf{z}=\left( z_{i} \right)_{i\in \mathbb{N}}\in c_0:
	\sum\limits_{n=1}^{\infty}
	n^{\frac{r}{q}-1}(z_n^*)^r<\infty
	\right\},\\
	\ell_{q,\infty}
	&\triangleq
	\left\{
	\textbf{z}=\left( z_{i} \right)_{i\in \mathbb{N}}\in c_0:
	\sup\limits_{n\in\mathbb{N}}n^{\frac{1}{q}}z_n^*<\infty
	\right\}.
\end{aligned}
$$

\begin{theorem}\label{dingli-duoxiangshishoulianyu}
	For every $1\leqslant p\leqslant\infty$,
	\begin{equation}
		\mathrm{dom}\mathcal{P}^{1}(\ell^p)=\ell^p,
		\qquad
		\mathrm{dom}\mathcal{P}^{1}(c_0)=c_0.
		\label{shi-yiciduoxiangshishoulian}
	\end{equation}
	For every $m\in\mathbb{N}$,
	\begin{equation}
		\mathrm{dom}\mathcal{P}^{m}(\ell^1)=\ell^1.
		\label{shi-yikeduoxiangshishoulian}
	\end{equation}
	For every $m\geqslant2$,
	\begin{equation}
		\mathrm{dom}\mathcal{P}^{m}(c_0)
		=
		\mathrm{dom}\mathcal{P}^{m}(\ell^\infty)
		=
		\ell_{\frac{2m}{m-1},\infty}.
		\label{shi-youjieduoxiangshishoulian}
	\end{equation}
	For $1<p<\infty$ and $m\geqslant2$, define
	\begin{equation}
		\frac{1}{q_m}
		\triangleq
		\frac{1}{p}+
		\frac{m-1}{m}
		\left(1-\frac{1}{\min\{p,2\}}\right).
		\label{shi-duoxiangshiqm}
	\end{equation}
	Then
	\begin{equation}
		\begin{aligned}
			\mathrm{dom}\mathcal{P}^{m}(\ell^p)
			&=
			\ell_{q_m,\max\{p,m\}}
		\end{aligned}
		\label{shi-duoxiangshishoulianyu}
	\end{equation}
\end{theorem}
\begin{remark}
If we additionally require that, at every point of monomial convergence, the corresponding monomial series converges to the original polynomial, then \namedref{Proposition}{mingti-jishubiaoshi} shows that all the above characterizations remain unchanged except for $\mathrm{dom}\mathcal{P}^{1}(\ell^{\infty})$. Indeed, by \namedref{Proposition}{mingti-jishubiaoshi}, every $\textbf{z}\in c_0$ belongs to the corresponding set of monomial convergence. Conversely, if $\textbf{z}\in\ell^\infty\setminus c_0$, the Hahn--Banach theorem gives $G\in(\ell^\infty)^*$ such that $G|_{c_0}=0$ and $G(\textbf{z})\neq0$. Since $e_i\in c_0$ for every $i\in\mathbb N$, we have
$
\sum\limits_{i=1}^{\infty}G(e_i)z_i=0.
$
Therefore, the corresponding set of monomial convergence for $\mathcal{P}^{1}(\ell^{\infty})$ is exactly $c_0$.
\end{remark}
\begin{remark}\label{rem-11}
Let $X=c_0$ or $X=\ell^p$ with $1\leqslant p\leqslant\infty$. The identities in \eqref{shi-yiciduoxiangshishoulian} are elementary, since
$$
\sum\limits_{i=1}^{\infty}
\left|c_{e_i}(G)z_i\right|
\leqslant
\|G\|\,\|\textbf{z}\|_X,
\qquad
G\in\mathcal{P}^{1}(X),\quad
\textbf{z}=(z_i)_{i\in\mathbb{N}}\in X.
$$
The identity \eqref{shi-yikeduoxiangshishoulian} follows directly
from \eqref{shi-yikeheshoulian}, since
$\mathcal{P}^m(\ell^1)\subset H(\ell^1)$; see also
\cite[Theorem~3.1, p.~19]{DMP09}.
For $m\geqslant2$, the identities in
\eqref{shi-youjieduoxiangshishoulian} were established in
\cite[Theorem~2.1 and equations~(15)--(16)]{BDFMS17}.
Moreover, when $2\leqslant m\leqslant p<\infty$, the sufficient
inclusion in \eqref{shi-duoxiangshishoulianyu} already follows
from \cite[Theorem~5.1\,(3)]{BDS19}.
\end{remark}

\begin{proof}
\raggedbottom
By \namedref{Remark}{rem-11}, it remains to prove
\eqref{shi-duoxiangshishoulianyu}.
Fix $1<p<\infty$ and $m\geqslant2$, and put
\begin{equation}
\sigma\triangleq1-\frac{1}{\min\{p,2\}},
\qquad r\triangleq\max\{p,m\}.
\label{shi-duoxiangshi-canshu}
\end{equation}
Thus 
\begin{equation}
\frac{1}{q_m}=\frac{1}{p}+\frac{m-1}{m}\sigma,
\qquad 1<q_m<p\leqslant r,
\qquad r\geqslant m.
\label{shi-duoxiangshi-q-sigma}
\end{equation}
For an $m$-homogeneous polynomial $G$, write
$$
S_{\textbf{z}}(G)\triangleq
\sum\limits_{\substack{\alpha\in\mathbb{N}_0^{(\mathbb{N})}\\|\alpha|=m}}
|c_\alpha(G)\textbf{z}^{\alpha}|,
$$
allowing the value $+\infty$.
We first prove that there exist $c,C>0$, depending only on
$m,p$, such that, for every $\textbf{z}\in\mathbb{C}^{(\mathbb{N})}$,
\begin{equation}
\begin{aligned}
c\left(\sum\limits_{n=1}^{\infty}
n^{\frac{r}{q_m}-1}(z_n^*)^r\right)^{\frac{m}{r}}
&\leqslant
\sup\limits_{\substack{
G\in\mathscr{P}_0\cap\mathcal{P}^{m}(\ell^p)\\\|G\|\leqslant1}}
S_{\textbf{z}}(G)&\leqslant
C\left(\sum\limits_{n=1}^{\infty}
n^{\frac{r}{q_m}-1}(z_n^*)^r\right)^{\frac{m}{r}}.
\end{aligned}
\label{shi-duoxiangshi-liangceguji}
\end{equation}

\proofstep{1}
Without loss of generality, we may assume
$z_1\geqslant z_2\geqslant\cdots\geqslant0$.
Put
$$
D_k\triangleq\{2^k,\cdots,2^{k+1}-1\},
\qquad k\in\mathbb{N}_0.
$$
Note that for every $k\in \mathbb{N}_0$, 
$$
2^{\frac{kr}{q_m}}z_{2^{k+1}}^r
\leqslant\sum\limits_{n\in D_k}n^{\frac{r}{q_m}-1}z_n^r
\leqslant2^{\frac{r}{q_m}-1}
\left(2^{\frac{k}{q_m}}z_{2^k}\right)^r.
$$
Summing the above inequalities over $k=0,1,2,\cdots$, we obtain
\begin{equation}
2^{1-\frac{r}{q_m}}
\sum\limits_{n=1}^{\infty}n^{\frac{r}{q_m}-1}z_n^r
\leqslant\sum\limits_{k=0}^{\infty}
\left(2^{\frac{k}{q_m}}z_{2^k}\right)^r
\leqslant\left(1+2^{\frac{r}{q_m}}\right)
\sum\limits_{n=1}^{\infty}n^{\frac{r}{q_m}-1}z_n^r.
\label{shi-duoxiangshi-erjinzhi}
\end{equation}

\proofstep{2}

We prove the upper estimate in \eqref{shi-duoxiangshi-liangceguji}.
Suppose first that $p\geqslant m$.
Put $\textbf{u} \triangleq \left( j^{\frac{m-1}{2m}}z_{j} \right)_{j\in \mathbb{N}}=\left( u_{j} \right)_{j\in \mathbb{N}}$ and
$\textbf{v} \triangleq \left( j^{-\frac{m-1}{2m}} \right)_{j\in \mathbb{N}}$.
By \eqref{shi-youjieduoxiangshishoulian},
$\textbf{v}\in\mathrm{dom}\mathcal{P}^m(c_0)$.
The closed graph argument in the proof of
\namedref{Lemma}{yinli-bituguji}, applied to $\mathcal{P}^m(c_0)$,
therefore gives $C_m>0$ such that
\begin{equation}
	S_{\textbf{v}}(Q)\leqslant C_m\|Q\|\label{shi-yy}
\end{equation} for every
$Q\in\mathcal{P}^m(c_0)$.

For $G\in\mathscr{P}_0\cap\mathcal{P}^m(\ell^p)$,
note that
$$
\widetilde{G}:c_0\to\mathbb{C},
\quad
\textbf{x}=(x_j)_{j\in\mathbb N}
\mapsto G\left((u_jx_j)_{j\in\mathbb N}\right)
$$
belongs to $\mathcal{P}^m(c_0)$.
Applying \eqref{shi-yy} and using
$c_\alpha(\widetilde{G})
=\textbf{u}^\alpha c_\alpha(G)$
for every $\alpha\in\mathbb{N}_0^{(\mathbb{N})}$,
we obtain
$$
\begin{aligned}S_{\textbf{z}}(G)&=S_{\textbf{v}}(\widetilde{G} )\leqslant C_{m}\| \widetilde{G} \| \leqslant C_{m}\sup\limits_{\textbf{x} \in B_{c_{0}}} \| G\| \left\| (u_{j}x_{j})_{j\in \mathbb{N}} \right\|_{\ell^{p}}^{m}\\ &\leqslant C_{m}\| G\| \| \textbf{u} \|_{\ell^{p}}^{m} =C_{m}\| G\| \left( \sum\limits_{j=1}^{\infty} j^{\frac{p\left( m-1 \right)}{2m}}z_{j}^{p} \right)^{\frac{m}{p}}\\ &=C_{m}\| G\| \left( \sum\limits_{j=1}^{\infty} j^{\frac{p}{q_{m}} -1}z_{j}^{p} \right)^{\frac{m}{p}} =C_{m}\| G\| \left( \sum\limits_{j=1}^{\infty} j^{\frac{r}{q_{m}} -1}z_{j}^{r} \right)^{\frac{m}{p}}\end{aligned}
$$

Now suppose that $1<p<m$.
Fix $G\in\mathscr{P}_0\cap\mathcal{P}^m(\ell^p)$ and choose
$N\in\mathbb{N}_0$ such that
$G=G\circ P_{2^{N+1}-1}$ and
$\mathrm{supp}\textbf{z}\subset\{1,\cdots,2^{N+1}-1\}$.
For $0\leqslant k_1\leqslant\cdots\leqslant k_m\leqslant N$, put
$$
A_{k_1,\cdots,k_m}\triangleq\left\{
\sum\limits_{j=1}^{m}e_{i_j}:
i_1\leqslant\cdots\leqslant i_m,\quad
i_j\in D_{k_j},\ 1\leqslant j\leqslant m
\right\}.\qquad \textbf{w}\triangleq
\sum\limits_{k\in\{k_1,\cdots,k_m\}}
2^{-\frac{k}{p}}\sum\limits_{i\in D_k}e_i.
$$
Then
$$
\left\{
\alpha\in\mathbb{N}_0^{(\mathbb{N})}:
|\alpha|=m,\quad
\mathrm{supp}\,\alpha\subset\{1,\cdots,2^{N+1}-1\}
\right\}
=
\bigsqcup\limits_{0\leqslant k_1\leqslant\cdots\leqslant k_m\leqslant N}
A_{k_1,\cdots,k_m}.
$$
Similarly to \eqref{shi-shanjianzhiji}, we define
$$
A_{k_1,\cdots,k_m}^-
\triangleq\left\{
\sum\limits_{j=1}^{m-1}e_{i_j}:
i_1\leqslant\cdots\leqslant i_{m-1},\quad
i_j\in D_{k_j},\ 1\leqslant j\leqslant m-1
\right\}.
$$
One has
$$
|A_{k_1,\cdots,k_m}^-|
\leqslant\prod\limits_{j=1}^{m-1}|D_{k_j}|
=2^{\sum\limits_{j=1}^{m-1}k_j},\qquad\| \textbf{w} \|_{\ell^{p}}^{p} =\sum\limits_{k\in \left\{ k_{1},k_{2},\cdots ,k_{m} \right\}} \frac{\left| D_{k} \right|}{2^{k}} =\left| \left\{ k_{1},\cdots ,k_{m} \right\} \right| \leqslant m.
$$

Applying \eqref{shi-zhijixishuguji}  and 
\eqref{shi-duoxiangshi-q-sigma}, we obtain
$$
\begin{aligned}\sum\limits_{\alpha \in A_{k_{1},\cdots ,k_{m}}} |c_{\alpha}(G)\textbf{z}^{\alpha} |&=\sum\limits_{\alpha \in A_{k_{1},\cdots ,k_{m}}} \left( |c_{\alpha}(G)|\prod\limits_{j=1}^{m} \left| z_{i_{j}} \right| \right)\leqslant \sum\limits_{\alpha \in A_{k_{1},\cdots ,k_{m}}} \left( |c_{\alpha}(G)|\prod\limits_{j=1}^{m} \left| z_{2^{k_{j}}} \right| \right)\\ &=\sum\limits_{\alpha \in A_{k_{1},\cdots ,k_{m}}} \left( |c_{\alpha}(G)|\cdot \prod\limits_{j=1}^{m} 2^{-\frac{k_{j}}{p}}\cdot \prod\limits_{j=1}^{m} \left( 2^{\frac{k_{j}}{p}}z_{2^{k_{j}}} \right) \right)\\ &=\sum\limits_{\alpha \in A_{k_{1},\cdots ,k_{m}}} \left( |c_{\alpha}(G)|\cdot \textbf{w}^{\alpha} \prod\limits_{j=1}^{m} \left( 2^{\frac{k_{j}}{p}}z_{2^{k_{j}}} \right) \right)\\ &\leqslant e^{2m}|A_{k_{1},\cdots ,k_{m}}^{-}|^{\sigma}\left\| \textbf{w} \right\|_{\ell^{p}}^{m} \| G\| \prod\limits_{j=1}^{m} \left( 2^{\frac{k_{j}}{p}}z_{2^{k_{j}}} \right)\\ &\leqslant e^{2m}m^{\frac{m}{p}}\| G\| 2^{\sigma \sum\limits_{j=1}^{m-1} k_{j}}\prod\limits_{j=1}^{m} \left( 2^{\frac{k_{j}}{p}}z_{2^{k_{j}}} \right)\\ &=e^{2m}m^{\frac{m}{p}}\| G\| 2^{\sigma \sum\limits_{j=1}^{m-1} k_{j}}\prod\limits_{j=1}^{m} \left( 2^{\frac{k_{j}}{q_{m}} -\frac{\sigma \left( m-1 \right) k_{j}}{m}}z_{2^{k_{j}}} \right)\\ &=e^{2m}m^{\frac{m}{p}}\| G\| 2^{-\frac{\sigma}{m} \sum\limits_{j=1}^{m-1} (k_{m}-k_{j})}\prod\limits_{j=1}^{m} \left( 2^{\frac{k_{j}}{q_{m}}}z_{2^{k_{j}}} \right) .\end{aligned}
$$
Summing the preceding inequality over
$0\leqslant k_1\leqslant\cdots\leqslant k_m\leqslant N$
and using $z_{2^k}=0$ for $k>N$, we obtain
\begin{equation}
S_{\textbf{z}}(G)\leqslant e^{2m}m^{\frac{m}{p}}\|G\|
\sum\limits_{0\leqslant k_1\leqslant\cdots\leqslant k_m}
2^{-\frac{\sigma}{m}\sum\limits_{j=1}^{m-1}(k_m-k_j)}
\prod\limits_{j=1}^{m}
\left(2^{\frac{k_j}{q_m}}z_{2^{k_j}}\right).
\label{shi-duoxiangshi-fenkuai}
\end{equation}
We have
\begin{equation}
\begin{aligned}
	&\sum\limits_{0\leqslant k_1\leqslant\cdots\leqslant k_m}
	2^{-\frac{\sigma}{m}\sum\limits_{j=1}^{m-1}(k_m-k_j)}
	\prod\limits_{j=1}^{m}
	\left(2^{\frac{k_j}{q_m}}z_{2^{k_j}}\right)
	\\
	&\qquad=
	\sum\limits_{0=k_1\leqslant\cdots\leqslant k_m}
	\prod\limits_{j=1}^{m-1}
	2^{-\frac{\sigma j}{m}(k_{j+1}-k_j)}
	\sum\limits_{k=0}^{\infty}
	\prod\limits_{j=1}^{m}
	\left(2^{\frac{k+k_j}{q_m}}z_{2^{k+k_j}}\right)
	\\
	&\qquad\leqslant
	\frac{1}{m}
	\sum\limits_{0=k_1\leqslant\cdots\leqslant k_m}
	\prod\limits_{j=1}^{m-1}
	2^{-\frac{\sigma j}{m}(k_{j+1}-k_j)}
	\sum\limits_{j=1}^{m}
	\sum\limits_{k=0}^{\infty}
	\left(2^{\frac{k+k_j}{q_m}}z_{2^{k+k_j}}\right)^m
	\\
	&\qquad\leqslant
	\sum\limits_{0=k_1\leqslant\cdots\leqslant k_m}
	\prod\limits_{j=1}^{m-1}
	2^{-\frac{\sigma j}{m}(k_{j+1}-k_j)}
	\sum\limits_{k=0}^{\infty}
	\left(2^{\frac{k}{q_m}}z_{2^k}\right)^m
	\\
	&\qquad=
	\prod\limits_{j=1}^{m-1}
	\left(\sum\limits_{k=0}^{\infty}2^{-\frac{\sigma j}{m}k}\right)
	\sum\limits_{k=0}^{\infty}
	\left(2^{\frac{k}{q_m}}z_{2^k}\right)^m
	\\
	&\qquad=
	\prod\limits_{j=1}^{m-1}
	\left(1-2^{-\frac{\sigma j}{m}}\right)^{-1}
	\sum\limits_{k=0}^{\infty}
	\left(2^{\frac{k}{q_m}}z_{2^k}\right)^m,
\end{aligned}
\label{shi-duoxiangshi-xianglinchaguji}
\end{equation}
where the first equality follows by subtracting the smallest index from each $k_j$ and using
$\sum\limits_{j=1}^{m-1}(k_m-k_j)
=\sum\limits_{j=1}^{m-1}j(k_{j+1}-k_j)$,
and the first inequality follows from the arithmetic--geometric mean inequality.

Combining \eqref{shi-duoxiangshi-fenkuai},
\eqref{shi-duoxiangshi-xianglinchaguji}, and
\eqref{shi-duoxiangshi-erjinzhi}, with $r=m$, proves the upper
estimate in \eqref{shi-duoxiangshi-liangceguji}.

\proofstep{3}

We prove the lower estimate in \eqref{shi-duoxiangshi-liangceguji}.
Put $C_0\triangleq16^{2\sigma}(m!\log m)^\sigma$.
For each $n\in\mathbb{N}$, let $R$ be the polynomial given by
\namedref{Lemma}{yinli-suijiduoxiangshi}, namely,
$$
R(\textbf{u})=
\sum\limits_{\substack{\alpha\in\mathbb{N}_0^n\\|\alpha|=m}}
\varepsilon_\alpha\frac{m!}{\alpha!}\textbf{u}^{\alpha},
\qquad\varepsilon_\alpha\in\{-1,1\}.
$$
Applying the classical inequality
$$
\|\textbf{u}\|_{\ell^{\min\{p,2\}}}
\leqslant
n^{\frac{1}{\min\{p,2\}}-\frac{1}{p}}
\|\textbf{u}\|_{\ell^p},
\quad \textbf{u}\in\mathbb{C}^n
$$
together with \eqref{shi-suijiduoxiangshiguji}, we obtain
$$
\begin{aligned}\sup\limits_{\| \textbf{u} \|_{\ell^{p}} \leqslant 1} |R(\textbf{u} )|&\leqslant \sup\limits_{\| \textbf{u} \|_{\ell^{\min\{p,2\}}} \leqslant n^{\frac{1}{\min\{p,2\}} -\frac{1}{p}}} |R(\textbf{u} )|\\ &=\sup\limits_{\| \textbf{u} \|_{\ell^{\min\{p,2\}}} \leqslant 1} \left| R\left( n^{\frac{1}{\min\{p,2\}} -\frac{1}{p}}\textbf{u} \right) \right|\\ &=n^{\frac{m}{\min\{p,2\}} -\frac{m}{p}} \sup\limits_{\| \textbf{u} \|_{\ell^{\min\{p,2\}}} \leqslant 1} \left| R\left( \textbf{u} \right) \right|\\ &\leqslant C_{0}n^{\frac{m}{\min\{p,2\}} -\frac{m}{p}}n^{\sigma}=C_{0}n^{m\left( 1-\frac{1}{q_{m}} \right)}.\end{aligned}
$$
For each $k\in\mathbb{N}_0$, choose such a polynomial $R_k$
in $n=2^k$ variables and define
$$
G_k(\textbf{x})\triangleq
\frac{R_k(x_{2^k},\cdots,x_{2^{k+1}-1})}
{C_0\,2^{km\left(1-\frac{1}{q_m}\right)}},
\qquad\textbf{x}\in\ell^p.
$$
Then $|G_k(\textbf{x})|\leqslant
\|P_{D_k}\textbf{x}\|_{\ell^p}^m$ and
\begin{equation}
	S_{\textbf{z}}(G_k)
	=\frac{\left(\sum\limits_{j\in D_k}z_j\right)^m}{C_0\,2^{km\left(1-\frac{1}{q_m}\right)}}
	\geqslant\frac{\left(2^kz_{2^{k+1}}\right)^m}{C_0\,2^{km\left(1-\frac{1}{q_m}\right)}}
	=\frac{1}{C_0}\left(2^{\frac{k}{q_m}}z_{2^{k+1}}\right)^m.
	\label{shi-duoxiangshi-kuai-zhixishu}
\end{equation}
and \begin{equation}
S_{\textbf{z}}(G_0)=\frac{z_1^m}{C_0}.
\label{shi-duoxiangshi-shoukuai}
\end{equation}

Let $\left(\frac{r}{m}\right)'=
\begin{cases}
	\frac{r}{r-m},&r>m,\\
	\infty,&r=m.
\end{cases}$.
For $N\in\mathbb{N}_0$ and
$\textbf{a}=(a_0,\cdots,a_N)\in\mathbb{C}^{N+1}$,
H\"older's inequality, $r\geqslant p$ and monotonicity of $q\mapsto \left\| \cdot \right\|_{\ell^{q}}$ gives
$$
\begin{aligned}\left|\sum\limits_{k=0}^{N} a_{k}G_{k}\left( \textbf{x} \right) \right|&\leqslant \sum\limits_{k=0}^{N} |a_{k}|\| P_{D_{k}}\textbf{x} \|_{\ell^{p}}^{m} \leqslant \| \textbf{a} \|_{\ell^{\left( \frac{r}{m} \right)^{\prime}}} \left( \sum\limits_{k=0}^{N} \| P_{D_{k}}\textbf{x} \|_{\ell^{p}}^{r} \right)^{\frac{m}{r}}\\ &\leqslant \| \textbf{a} \|_{\ell^{\left( \frac{r}{m} \right)^{\prime}}} \left( \sum\limits_{k=0}^{N} \| P_{D_{k}}\textbf{x} \|_{\ell^{p}}^{p} \right)^{\frac{m}{p}} \leqslant \| \textbf{a} \|_{\ell^{\left( \frac{r}{m} \right)^{\prime}}} \| \textbf{x} \|_{\ell^{p}}^{m}.\end{aligned}
$$

Thus 
$$\left\| \sum\limits_{k=0}^{N} a_{k}G_{k}\left( \cdot \right) \right\|\leqslant\|\textbf{a}\|_{\ell^{\left(\frac{r}{m}\right)'}},\qquad S_{\textbf{z}}\left(\sum\limits_{k=0}^{N} a_{k}G_{k}\left( \cdot \right) \right)=\sum\limits_{k=0}^{N} |a_{k}|S_{\textbf{z}}(G_{k}).$$
Finite-dimensional duality, followed by taking the supremum over
$N$, therefore yields
\begin{equation}
\begin{aligned}
	\sup\limits_{\substack{G\in\mathscr{P}_0\cap\mathcal{P}^{m}(\ell^p)\\\|G\|\leqslant1}}S_{\textbf{z}}(G)
	&\geqslant\sup\limits_{N\geqslant0}\sup\limits_{\|\textbf{a}\|_{\ell^{\left(\frac{r}{m}\right)'}}\leqslant1}S_{\textbf{z}}\left(\sum\limits_{k=0}^{N} a_{k}G_{k}\left( \cdot \right) \right)
	=\sup\limits_{N\geqslant0}\sup\limits_{\|\textbf{a}\|_{\ell^{\left(\frac{r}{m}\right)'}}\leqslant1}\sum\limits_{k=0}^{N}|a_k|S_{\textbf{z}}(G_k)\\
	&=\sup\limits_{N\geqslant0}\left(\sum\limits_{k=0}^{N}S^{\frac{r}{m}}_{\textbf{z}}(G_k)\right)^{\frac{m}{r}}
	=\left(\sum\limits_{k=0}^{\infty}S^{\frac{r}{m}}_{\textbf{z}}(G_k)\right)^{\frac{m}{r}}.
\end{aligned}
\label{shi-duoxiangshi-kuai-duiou}
\end{equation}
By \eqref{shi-duoxiangshi-kuai-zhixishu} and
\eqref{shi-duoxiangshi-shoukuai},
$$
\begin{aligned}
\sum\limits_{k=0}^{\infty}
\left(2^{\frac{k}{q_m}}z_{2^k}\right)^r
&=z_1^r+2^{\frac{r}{q_m}}
\sum\limits_{k=0}^{\infty}
\left(2^{\frac{k}{q_m}}z_{2^{k+1}}\right)^r\leqslant C_0^{\frac{r}{m}}\left(1+2^{\frac{r}{q_m}}\right)
\sum\limits_{k=0}^{\infty}S^{\frac{r}{m}}_{\textbf{z}}(G_k).
\end{aligned}
$$
Together with \eqref{shi-duoxiangshi-erjinzhi} and
\eqref{shi-duoxiangshi-kuai-duiou}, this proves the lower estimate
in \eqref{shi-duoxiangshi-liangceguji}.

\proofstep{4}

First, $\ell_{q_m,r}\subset\ell^p$. Indeed, for $\textbf{z}\in\ell_{q_m,r}$, since $r\geqslant p> q_m$ and $$\sum\limits_{j=1}^{n}j^{\frac{r}{q_m}-1}\geqslant\sum\limits_{j=1}^{n}\int\limits_{j-1}^{j}t^{\frac{r}{q_m}-1}\,\mathrm{d}t
=\int\limits_0^n t^{\frac{r}{q_m}-1}\,\mathrm{d}t
=\frac{q_m}{r}n^{\frac{r}{q_m}},$$ we have
$$
\begin{aligned}
	\sum\limits_{n=1}^{\infty}(z_n^*)^p
	&=\sum\limits_{n=1}^{\infty}n^{-\frac{p}{q_m}}
	\left(n^{\frac{r}{q_m}}(z_n^*)^r\right)^{\frac{p}{r}}
	\leqslant\sum\limits_{n=1}^{\infty}n^{-\frac{p}{q_m}}
	\left(\frac{r}{q_m}\sum\limits_{j=1}^{n}j^{\frac{r}{q_m}-1}(z_n^*)^r\right)^{\frac{p}{r}}\\
	&\leqslant\sum\limits_{n=1}^{\infty}n^{-\frac{p}{q_m}}
	\left(\frac{r}{q_m}\sum\limits_{j=1}^{n}j^{\frac{r}{q_m}-1}(z_j^*)^r\right)^{\frac{p}{r}}\\
	&\leqslant\left(\frac{r}{q_m}\sum\limits_{j=1}^{\infty}j^{\frac{r}{q_m}-1}(z_j^*)^r\right)^{\frac{p}{r}}
	\sum\limits_{n=1}^{\infty}n^{-\frac{p}{q_m}}<\infty.
\end{aligned}
$$

Arguing as in the proof of \namedref{Lemma}{yinli-gujituiguang},
we extend the upper estimate in \eqref{shi-duoxiangshi-liangceguji}
to $\textbf{z}\in\ell_{q_m,r}$ and
$G\in\mathcal{P}^m(\ell^p)$. Consequently,
$$
S_{\textbf{z}}(G)\leqslant C\|G\|
\left(\sum\limits_{j=1}^{\infty}
j^{\frac{r}{q_m}-1}(z_j^*)^r\right)^{\frac{m}{r}}<\infty,
$$
which proves
$\ell_{q_m,r}\subseteq\mathrm{dom}\mathcal{P}^m(\ell^p)$.

Conversely, let $\textbf{z}=\left( z_{i} \right)_{i\in \mathbb{N}}\in\mathrm{dom}\mathcal{P}^m(\ell^p)$.
By the same closed graph argument as in the proof of
\namedref{Lemma}{yinli-bituguji}, applied to the coefficient map
$G\mapsto(c_\alpha(G)\textbf{z}^{\alpha})_{|\alpha|=m}$
on the Banach space $\mathcal{P}^m(\ell^p)$,
there exists $K>0$ such that
\begin{equation}
	S_{\textbf{z}}(G)\leqslant K\|G\|,
	\qquad G\in\mathcal{P}^m(\ell^p).
	\label{shi-duoxiangshi-bitu}
\end{equation}

Since $S_{P_n\textbf{z}}(G)\leqslant S_{\textbf{z}}(G)$, the lower
estimate in \eqref{shi-duoxiangshi-liangceguji} implies
$$
c\left(\sum\limits_{j=1}^{\infty}
j^{\frac{r}{q_m}-1}\bigl((P_n\textbf{z})_j^*\bigr)^r
\right)^{\frac{m}{r}}\leqslant K.
$$
As $\textbf{z}\in\ell^p\subset c_0$,
$(P_n\textbf{z})_j^*\uparrow z_j^*$ for every $j\in \mathbb{N}$.
Monotone convergence yields
$$
\sum\limits_{j=1}^{\infty}j^{\frac{r}{q_m}-1}(z_j^*)^r
\leqslant\left(\frac{K}{c}\right)^{\frac{r}{m}}<\infty.
$$
Thus $\mathrm{dom}\mathcal{P}^m(\ell^p)\subseteq\ell_{q_m,r}$,
which completes the proof of \eqref{shi-duoxiangshishoulianyu}.
\end{proof}
\section*{Acknowledgements}
The author wishes to express his sincere gratitude to Professor Jiayang Yu for many valuable discussions, and to Professor Xu Zhang for his invaluable guidance throughout the author's doctoral studies. AI tools were used to assist in polishing the exposition of several proofs and improving the language of the manuscript. All research ideas, strategies, and mathematical arguments were developed independently by the author.

\end{document}